\documentclass[11pt]{article}

\usepackage[T1]{fontenc}
\usepackage[margin=1in]{geometry}
\usepackage{setspace}
\usepackage{amsmath,amssymb,amsthm}
\usepackage{bm}
\usepackage{graphicx}
\usepackage{xcolor}
\usepackage{algorithm,algorithmicx}
\usepackage{algpseudocode}
\usepackage{multirow}
\usepackage{natbib}
\bibpunct[, ]{(}{)}{,}{a}{}{,}
\usepackage{tikz}
\usetikzlibrary{arrows,fit,calc,decorations.pathreplacing}
\usepackage{subfigure}
\usepackage{wrapfig}
\usepackage{pdflscape}
\usepackage{hyperref}
\hypersetup{hidelinks}

\theoremstyle{plain}
\newtheorem{theorem}{Theorem}
\newtheorem{proposition}{Proposition}
\newtheorem{lemma}{Lemma}
\newtheorem{corollary}{Corollary}

\newtheorem{assumption}{Assumption}
\theoremstyle{definition}
\newtheorem{definition}{Definition}
\newtheorem{example}{Example}
\theoremstyle{remark}
\newtheorem{remark}{Remark}
\newcommand{\R}{\ensuremath{\mathbb{R}}}
\newcommand{\Z}{\ensuremath{\mathbb{Z}}}

\newcommand{\conv}{\ensuremath{\text{conv}}}

\newcommand*{\defeq}{\stackrel{\text{def}}{=}}

\newcommand{\cl}{\ensuremath{\text{cl}}}
\newcommand{\Diag}{\ensuremath{\text{Diag}}}

\newcommand{\Row}{\ensuremath{\text{Row}}}

\newcommand{\Col}{\ensuremath{\text{Col}}}

\newcommand{\supp}{\mathrm{supp}}

\newcommand{\Nodes}{N}
\newcommand{\Arcs}{A}
\newcommand{\val}[1]{\nu_{#1}}
\newcommand{\valvec}{\bm{\val{}}}
\newcommand{\rootnode}{r}
\newcommand{\terminalnode}{t}

\newcommand{\tail}[1]{\ensuremath{t_{#1}}}
\newcommand{\head}[1]{\ensuremath{h_{#1}}}
\newcommand{\DD}{\ensuremath{\mathcal{D}}}

\newcommand{\A}{\bm{A}}

\newcommand{\Q}{\bm{Q}}
\newcommand{\F}{\bm{F}}

\newcommand{\norm}[1]{\left\lVert#1\right\rVert}

	\tikzstyle{zero arc} = [draw,dashed, line width=0.5pt,->]
	\tikzstyle{one arc} = [draw,line width=0.5pt,->]
	\tikzstyle{zero arc infeasible} = [zero arc,
preaction={
	draw,gray!30!white,-,
	double=gray!30!white,
	double distance=3pt,
}]
	\tikzstyle{one arc infeasible} = [one arc,
preaction={
	draw,gray!30!white,-,
	double=gray!30!white,
	double distance=3pt,
}]
	\tikzstyle{main node} = [circle,fill=gray!50,font=\scriptsize, inner sep=1pt]
	\tikzstyle{text node} = [font=\scriptsize]
	\tikzstyle{arc text} = [font=\scriptsize]

	\tikzstyle{optimal arc} = [-,line width=3pt, color=blue!20!white]
	\tikzstyle{blocked arc} = [-,line width=3pt, color=red]

\begin{document}
	
	\title{Convexification of Mixed-Integer Quadratic Optimization via Decision Diagrams}
\author{%
Soobin Choi$^{1}$ \quad Salar Fattahi$^{2}$ \quad Shaoning Han$^{3}$\\
Andr\'es G\'omez$^{1}$ \quad Leonardo Lozano$^{4}$\\[0.8em]
\small $^{1}$Daniel J. Epstein Department of Industrial and Systems Engineering,\\
\small University of Southern California, Los Angeles, California 90089\\
\small $^{2}$Department of Industrial and Operations Engineering,\\
\small University of Michigan, Ann Arbor, Michigan 48109\\
\small $^{3}$Department of Mathematics, National University of Singapore, Singapore 119076\\
\small $^{4}$Operations, Business Analytics, and Information Systems Department,\\
\small University of Cincinnati, Cincinnati, Ohio 45221\\[0.5em]
\small \texttt{soobinch@usc.edu; fattahi@umich.edu; shaoninghan@nus.edu.sg}\\
\small \texttt{gomezand@usc.edu; leolozano@uc.edu}}
\date{}

\maketitle

\begin{abstract}
We study mixed-integer quadratic optimization (MIQO) problems with indicator variables. We propose a unified framework, based on decision diagrams, that serves both to solve the associated optimization problems and to construct ideal conic quadratic extended formulations of the closure of the convex hull of the underlying mixed-integer set. The construction applies to arbitrary quadratics and to any combinatorial constraints admitting a tractable dynamic programming representation. The resulting diagrams and the ensuing convex hull descriptions are of polynomial size when the quadratic is low-rank, or when the support graph of the Hessian or of its inverse is a tree, recovering and generalizing several results from the literature. For structured sparse and inverse-sparse quadratics, we show that approximate decision diagrams have size linear in the dimension while yielding solutions with arbitrarily low optimality gap. Computational experiments demonstrate the effectiveness of the proposed approach.
\end{abstract}

\noindent\textbf{Keywords:} mixed-integer quadratic optimization; indicator variables; decision diagrams; convexification
\medskip

	\section{Introduction}
	Mixed-integer quadratic optimization (MIQO) problems with indicator variables are ubiquitous in statistics, finance, and control. Formally, given vectors $\bm{c},\bm{d}\in \R^n$ and $\bm{b}\in \R^q$, a matrix $\bm{F}\in \R^{n\times m}$ and a set $Z\subseteq \{0,1\}^n$, we consider MIQO problems of the form
	\begin{subequations}\label{prob:miqp}
		\begin{align}
			\displaystyle\min_{x_0,\bm{x},\bm{z}}\quad & \bm{d}^\top \bm{x} + \bm{c}^\top \bm{z} + x_0\\
			\text{s.t.}\quad & \bm{Ax}+\bm{Gz}\leq \bm{b} \label{eq:miqp_linear}\\
			& (x_0,\bm{x},\bm{z})\in X_{\bm{F},Z}\label{eq:miqp_nonlinear}
		\end{align}
	\end{subequations}
	where $\bm{A}$, $\bm{G}$ and $\bm{b}$ are matrices and vectors of appropriate dimensions and
	$$X_{\bm{F},Z}=\left\{(x_0,\bm{x},\bm{z})\in \R\times\R^{n}\times Z: \|\bm{F^\top x}\|_2^2\leq x_0,\; \bm{x}\circ(\bm{1}-\bm{z})=\bm{0}\right\}.$$
	In words, set $X_{\bm{F},Z}$ collects the epigraph variable $x_0$ of the quadratic $\|\bm{F^\top x}\|_2^2$, the indicator constraints linking each continuous variable $x_i$ to its indicator $z_i$ (so that $x_i=0$ whenever $z_i=0$), and the combinatorial constraints $\bm{z}\in Z$. Throughout the paper, $\bm{1}$ is a vector of ones of appropriate dimension and ``$\circ$'' represents the Hadamard (entrywise) product of vectors; if $Z=\{0,1\}^n$, we write $X_{\bm{F}}$ instead of $X_{\bm{F},Z}$ for simplicity. The indicator constraints $\bm{x}\circ (\bm{1}-\bm{z})=\bm{0}$ are usually linearized as $-M\bm{z}\leq \bm{x}\leq M\bm{z}$, provided a valid value of $M$ can be inferred from constraints \eqref{eq:miqp_linear} or set $X_{\bm{F},Z}$. 
	
Problems of the form \eqref{prob:miqp} arise naturally in finance, statistics, and control. In cardinality-constrained portfolio optimization, factor models decompose covariance into diagonal specific risk and a low-rank common-factor term \citep{bienstock1996computational}; in slowly varying regression and online monitoring, local temporal interactions produce sparse block-banded quadratics \citep{bertsimas2024slowly,gomez2024real}; and in hybrid model predictive control, eliminating the state variables typically produces a dense quadratic with a sparse or block-banded inverse \citep{lee2024convexification}. 

	Solution approaches for mixed-integer optimization (MIO) problems often rely on effective convex relaxations. A natural convex relaxation of \eqref{prob:miqp}, obtained by replacing the indicator constraints with big-M constraints and constraints $\bm{z}\in Z$ with $\bm{z}\in \conv(Z)$, where ``$\conv(Z)$'' denotes the convex hull of $Z$, is typically weak. As a result, branch-and-bound algorithms based on this natural relaxation require substantial enumeration before being able to prove optimality. A stronger relaxation, which is the main focus of this paper, is given by \begin{subequations}\label{eq:miqpRelax}
		\begin{align}
			\displaystyle\min_{x_0,\bm{x},\bm{z}}\quad & \bm{d}^\top \bm{x} + \bm{c}^\top \bm{z} + x_0\\
			\text{s.t.}\quad & \bm{Ax}+\bm{Gz}\leq \bm{b} \label{eq:miqpRelax_linear}\\
			& (x_0,\bm{x},\bm{z})\in \text{cl conv}(X_{\bm{F},Z})\label{eq:miqpRelax_nonlinear}
		\end{align}
	\end{subequations}
	where ``cl'' denotes the closure of a set. Relaxation \eqref{eq:miqpRelax} dominates the natural relaxation in terms of strength, and results in an exact reformulation of the problem if constraints \eqref{eq:miqp_linear} are removed. Unfortunately, since optimization over set $X_{\bm{F},Z}$ is NP-hard in general even if $Z=\{0,1\}^n$ \citep{natarajan1995sparse}, computing the convex hull of $X_{\bm{F},Z}$ is a daunting task: explicit descriptions often require complicated nonlinear functions defined by pieces, an exponential number of constraints, or an exponential number of additional variables.
	Nonetheless, if additional assumptions are imposed on set $X_{\bm{F},Z}$, then optimization over this set may become polynomial-time solvable, and compact conic quadratic descriptions of $\text{cl conv}(X_{\bm{F},Z})$ are sometimes possible. Indeed, several such results exist in the literature, each relying on different convexification techniques.
	
	Convexifications of $X_{\bm{F},Z}$ are available if either the dimension $n$ or the rank of the quadratic form---given by the number of columns $m$ of $\bm{F}$---is fixed. On the one hand, set $X_{\bm{F},Z}$ can be represented as a disjunction of $|Z|\leq 2^n$ convex sets, each corresponding to a possible value of the binary variables. Disjunctive programming techniques \citep{balas1979disjunctive,ceria1999convex} can thus be used to produce representations of $\text{cl conv}(X_{\bm{F},Z})$ with $\mathcal{O}(|Z|)$ additional variables and constraints. The widely used perspective reformulation \citep{frangioni2006perspective,akturk2009strong,gunluk2010perspective} can be interpreted as a special case of disjunctive programming convexifications with $n=1$. \citet{frangioni2020decompositions} computationally tested convexifications of low-dimensional sub-blocks; the approach appears to be effective only for blocks of dimension at most two, as the burden of handling the additional variables and constraints otherwise outweighs the benefits of the stronger relaxations. Exact convex hulls for the case $n=2$ have also been derived in lifted spaces containing the matrix variable $\bm{X}=\bm{xx^\top}$, often incorporating nonnegativity or box constraints on $\bm{x}$ \citep{anstreicher2021quadratic,han2x2,derosa2024explicit}; related lifted inequalities for general MIQO are studied by \citet{dong2013valid}. On the other hand, descriptions of $\text{cl conv}(X_{\bm{F},Z})$ for $m=1$ were first obtained by \citet{atamturk2025rank}; \citet{shafiee2024constrained} and \citet{wei2022ideal} extended these results to more general constraints. Moreover, \citet{Han2024c} show that disjunctive programming techniques can also be applied to settings with $m\geq 2$, deriving compact formulations of $\text{cl conv}(X_{\bm{F},Z})$ of size polynomial in $n$ but exponential in $m$.

	If both $n$ and $m$ are large, disjunctive programming techniques cannot be applied effectively, and obtaining compact convexifications requires exploiting the structure of matrix $\bm{F}$. \citet{liu2023graph} and \citet{lee2024convexification} consider the full-rank case, and derive compact extended formulations for $\text{cl conv}(X_{\bm{F}})$ when matrix $\bm{FF^\top}$ is tridiagonal or its inverse is tridiagonal, respectively. Both papers rely on dynamic programming to solve optimization problems over $X_{\bm{F}}$: the first then uses lifting \citep{richard2010lifting} to derive the convex hull description, whereas the second exploits an underlying polyhedrality of the set $\text{cl conv}(X_{\bm{F}})$ \citep{wei2024convex}. 
	Moreover, \citet{gomez2024real} use decision diagrams and the results of \citet{wei2024convex} to either describe or approximate $\text{cl conv}(X_{\bm{F}})$ when matrix $\bm{FF^\top}$ is full rank and banded. For other matrix classes, \citet{atamturk2018strong, liu2025polyhedral} derive strong formulations for quadratics with Stieltjes matrices.

	\paragraph{Contributions and outline}
	
	We develop a unified, structure-adaptive decision-diagram framework for optimizing over $X_{\bm{F},Z}$ and constructing ideal conic quadratic extended formulations of $\operatorname{cl}\operatorname{conv}(X_{\bm{F},Z})$. Decision diagrams are compact graphical representations of discrete sets derived from dynamic programs \citep{bergman2016,castro2022decision}; our construction extends their recent use for banded MIQO \citep{gomez2024real} to arbitrary matrices $\bm{F}$ and combinatorial sets $Z$. Its states are defined through Gram--Schmidt projections of the rows of $\bm{F}$, and exploitable structure emerges automatically through the merging of equivalent states. The framework yields polynomial-size exact diagrams for low-dimensional, low-rank, tree, and inverse-tree structures, removes prior full-rank requirements, incorporates structured constraints on $\bm{z}$, and can exploit multiple structures simultaneously. Moreover, for sparse or inverse-sparse matrices we construct $\epsilon$-exact diagrams by merging nearly identical states and prove that, under polynomial volume growth and uniformly bounded boundary size, they have size linear in $n$ for fixed $\epsilon>0$. 

	The rest of the paper is organized as follows. In \S\ref{sec:notation} we summarize the notation used in the paper. In \S\ref{sec:preliminaries} we review preliminaries on decision diagrams for discrete optimization and on the Gram-Schmidt process. In \S\ref{sec:DDMIQO} we introduce the proposed decision diagrams for MIQO, in \S\ref{sec:pseudocode} we discuss their construction and implementation, and in \S\ref{sec:convexification} we prove that they yield exact convexifications of $X_{\bm{F},Z}$. In \S\ref{sec:polycase} we show that the resulting diagrams and convexifications are of polynomial size for low-rank, tree and inverse tree structures, and in \S\ref{sec:approximation} we construct approximations resulting in arbitrarily low errors for sparse and inverse-sparse structures. In \S\ref{sec:computations} we illustrate how the approximation theory can inform variable ordering, and in \S\ref{sec:realData} we present computational results with real data.
	All proofs of propositions, theorems, lemmas, and corollaries in \S\ref{sec:DDMIQO} onwards are provided in the online companion (\S\ref{sec:online-companion-proofs}).

	\section{Notation}\label{sec:notation}
	
	In this section, we collect the notation used throughout the paper. For ease of reference, the notation is grouped according to the section in which it is first used.
	
	\paragraph{General notation}
	We denote by $\bm{0}_n$ and $\bm{1}_n$ the $n$-dimensional vectors of zeros and ones, and by $\bm{0}_{n\times m}$ the $n\times m$ matrix of zeros. When the dimensions are clear from the context, the subscripts may be omitted. Let $\bm{e_i}\in \R^n$ denote the $i$-th elementary vector of $\R^n$. For integers $j\geq i\geq 0$, let $[i:j]:=\{i,i+1,\ldots,j-1,j\}$. In particular, if $i=1$, we may simply write $[j]:=[1:j]$. Given a vector $\bm{z}\in \R^n$, $\Diag(\bm{z})\in \R^{n\times n}$ is the square matrix whose diagonal elements correspond to $\bm{z}$.
	For two matrices (or vectors) $\bm{V}$ and $\bm{W}$ of the same dimensions, the operation $\bm{V} \circ \bm{W}$ denotes their Hadamard (entrywise) product, that is, $(\bm{V} \circ \bm{W})_{ij} = V_{ij} W_{ij}$. Moreover, $\norm{\bm{V}}_{2,\infty}$ denotes the maximum of the row-wise two-norms of a matrix $\bm{V}$. Given two sets $L_1,L_2\subseteq \R^n$, we let $L_1+L_2\defeq \left\{\bm{x_1}+\bm{x_2}: \bm{x_1}\in L_1,\; \bm{x_2}\in L_2\right\}$ denote their Minkowski sum. Given a set $L\subseteq \R^n$, ``$\conv(L)$'' denotes its convex hull and ``$\cl(L)$'' its closure.
	
	\paragraph{Submatrices}
	We now define the notation used for submatrices. Consider a matrix $\bm{F}\in \R^{n\times m}$. Given any subset $S\subseteq [n]$, $\bm{F_S}\in \R^{|S|\times m}$ is the matrix obtained by deleting rows with indices not in $S$; more generally, given subsets $S\subseteq [n]$ and $S'\subseteq [m]$, $\bm{F_{S,S'}}\in \R^{|S|\times |S'|}$ is the submatrix obtained by additionally deleting columns with indices not in $S'$. We let $\bm{F_\downarrow}\defeq\bm{F_{[2:n]}}$ denote the matrix obtained by deleting the first row of $\bm{F}$. Similarly, given a vector $\bm{z}\in \{0,1\}^n$, we let $\bm{F_z}\defeq \bm{F_S}$ where $S=\left\{i\in [n]: z_i=1\right\}$ is the support of $\bm{z}$. Given two vectors $\bm{z_1}\in\{0,1\}^k$ and $\bm{z_2}\in \{0,1\}^{n-k}$ with $1\leq k<n$, we let $\bm{F_{(z_1,z_2)}}\defeq\bm{F_z}$ where $\bm{z^\top}=(\bm{z_1^\top}\;\bm{z_2^\top})$, that is, we treat vectors in the subscript as row vectors for notational convenience.  Moreover, given $\bm{z}\in \{0,1\}^k$ with $1\leq k<n$, we let $\bm{F_z}\defeq\bm{F_{(z,0_{n-k})}}$ (in words, if the dimension of the binary vector is less than the number of rows of $\bm{F}$, then the submatrix operation is applied to the first $k$ rows and the rest are discarded). Throughout the paper, the submatrix operation takes precedence over other operations; for example, $\bm{F_z}^\top=\left(\bm{F_z}\right)^\top$ and $\bm{F_{S,S}^{-1}}=\left(\bm{F_{S,S}}\right)^{-1}$. Moreover, when the matrix to be indexed is itself the result of an operation or already carries a subscript, we enclose it in parentheses: for example, $\left(\bm{Q_{T,T}^{-1}}\right)_{S}$ denotes the submatrix, with rows indexed by $S$, of the inverse of $\bm{Q_{T,T}}$. We note that all these operations extend naturally to vectors by treating them as matrices with a single column. 
	
	\begin{example}
		Consider matrix $\bm{F}=\begin{pmatrix} 1& 2\\ 3&4\\5&6\end{pmatrix}$ and vectors $\bm{z_1^\top}=\begin{pmatrix}1&0\end{pmatrix}$, $\bm{z_2^\top}=\begin{pmatrix}1\end{pmatrix}$ and $\bm{z^\top}=\begin{pmatrix}1&0&1\end{pmatrix}$. Then
		{\footnotesize\begin{align*}
				&\bm{F_z}=\bm{F_{(z_1,z_2)}}=\begin{pmatrix}1&2\\5&6
				\end{pmatrix},\; \bm{F_{[2]}}=\bm{F_{(z_2,z_1)}}=\begin{pmatrix}1&2\\3&4
				\end{pmatrix},\;
\bm{F_\downarrow}=\bm{F_{[2:3]}}=\begin{pmatrix}3&4\\5&6
				\end{pmatrix},\;
				\bm{F_{z_1}}=\bm{F_{z_2}}=\bm{F_{\{1\}}}=\begin{pmatrix}1&2\end{pmatrix}.	\hfill\blacksquare
		\end{align*}}
	
	\end{example}
	
	A closely related object is $\Diag(\bm{z})\bm{F}\in \R^{n\times m}$, the matrix whose $i$-th row is $\bm{F_{\{i\}}}$ if $z_i=1$ and $\bm{0}_m^\top$ if $z_i=0$. In other words, $\Diag(\bm{z})\bm{F}$ is related to $\bm{F_z}$, with the rows outside the support of $\bm{z}$ padded with zeros rather than removed. Matrix $\Diag(\bm{z})\bm{F}$ is used extensively, in part to avoid dimension mismatches. An important identity, used throughout the paper, is that if $\bm{x}\circ (\bm{1}-\bm{z})=\bm{0}$ holds, then $\bm{x}=\bm{z}\circ\bm{x}=\Diag(\bm{z})\bm{x}$ and
	$$\bm{F^\top x}=\bm{F^\top}(\bm{z}\circ \bm{x})=\bm{F^\top}\Diag(\bm{z}) \bm{x}=\left(\Diag(\bm{z})\bm{F}\right)^\top \bm{x}.$$

	\paragraph{Support graphs (\S\ref{sec:polycase})}
	Given a symmetric matrix $\Q \in \mathbb{R}^{n \times n}$, its \emph{support graph}, denoted by $\supp(\Q)$, is defined as the simple unweighted graph with vertex set $[n]$, where an edge $(i,j)$ exists if and only if $\Q_{ij} \neq 0$ for $i \neq j$. For simplicity, and by a slight abuse of notation, we identify $\supp(\bm{Q})$ with the set of edges defining the graph, and thus write $(i,j)\in \supp(\bm{Q})$ to denote that $(i,j)$ is an edge of $\supp(\bm{Q})$.
	For any subset \(S \subseteq [n]\), define \(\Gamma_{\Q}(S)\) as the set of neighbors of \(S\) in $\supp(\Q)$, i.e., \(\Gamma_{\Q}(S) := \{\, j \in [n] \setminus S \;|\; \exists\, i \in S \text{ with } (i,j) \in \supp(\Q) \,\}\).
	
	\paragraph{Spectral quantities and graph distances (\S\ref{sec:approximation})}
	We denote by $\mu_{\min}(\Q)$ and $\mu_{\max}(\Q)$ the smallest and largest eigenvalues of $\Q$, by $\kappa_2\defeq \mu_{\max}(\Q)/\mu_{\min}(\Q)$ its condition number, and we define the \emph{decay rate}
	$\sigma\defeq \frac{\sqrt{\kappa_2}-1}{\sqrt{\kappa_2}+1}\in [0,1).$
	
	For any node $u$ of a graph $G$ with vertex set $[n]$ and any integer $\eta\geq 1$, let $d_G$ denote the shortest-path distance in $G$ and refer to $\{v\in[n]:d_G(u,v)\leq\eta\}$ as the \emph{$\eta$-neighborhood} of $u$. Let $\Delta_{G,u,\eta} \defeq \left|\{v\in[n]:d_G(u,v)\leq\eta\}\right|$ and $\Delta_{G,\eta} \defeq \max_{u \in [n]} \Delta_{G,u,\eta}$; the latter denotes the maximum size of an $\eta$-neighborhood in $G$. Since we exclusively work with the graphs $G=\supp(\Q)$ or $G=\supp(\Q^{-1})$, we omit the subscript $G$ whenever it is clear from the context.

	\section{Preliminaries}\label{sec:preliminaries}
	
	In this section, we review preliminary material for the paper. \S\ref{sec:ddComb} reviews decision diagrams for combinatorial problems, and \S\ref{sec:linearAlgebra} covers fundamental topics on linear algebra that will be used throughout the paper.
	
	\subsection{Decision diagrams for discrete feasible spaces}\label{sec:ddComb}
	
	Decision diagrams have primarily been used to solve purely combinatorial problems, that is, problems of the form 
	\begin{align}
		\min_{\bm{z}\in \{0,1\}^n}\;\bm{c^\top z}\; \text{  s.t. }\; \bm{z}\in Z.\label{eq:combinatorial}
	\end{align}
	Naturally, problem \eqref{prob:miqp} is a generalization of \eqref{eq:combinatorial}. In this paper, we assume that there is a tractable decision-diagram-based approach for the simpler problem \eqref{eq:combinatorial}, as we discuss next.

	A binary decision diagram $\DD = (\Nodes,\Arcs,\valvec)$ for \eqref{eq:combinatorial} encodes the set of points $Z \subseteq \{0,1\}^n$ as a directed acyclic graph with node set $\Nodes$ and arc set $A \subseteq \Nodes \times \Nodes$. The set $\Nodes$ is partitioned into $n+1$ layers $\Nodes^1, \dots, \Nodes^{n+1}$, where $\Nodes^1 = \{\rootnode\}$ contains a \emph{root} node $\rootnode$ and $\Nodes^{n+1}$ contains a set of \emph{terminal nodes}. Arcs $a \in A$ connect nodes in consecutive layers and represent value assignments for the components of $\bm{z} \in Z$. Let $\ell(a) \in \{1,\dots,n\}$ be the layer from which arc $a \in A$ emanates. We denote the tail node of an arc by $\tail{a} \in \Nodes^{\ell(a)}$, the head node by $\head{a} \in \Nodes^{\ell(a)+1}$, and its value assignment by $\val{a} \in \{0,1\}$.
	
	A decision diagram $\DD$ encodes $Z$ through paths from the root $\rootnode$ to a terminal node $\terminalnode\in \Nodes^{n+1}$, as follows. Each component of $\bm{z} \in Z$ corresponds to a layer in $\DD$. An arc-specified $\rootnode-\terminalnode$ path $(a_1, \dots, a_n)$, where $\head{a_i} = \tail{a_{i+1}}$ for $i=1,\dots,n-1$, encodes the vector $\bm{z} = (\val{a_1}, \dots, \val{a_n})^\top$. Assigning length $c_{\ell(a)}\nu_a$ to each arc $a\in \Arcs$, a decision diagram is \emph{exact} if every point $\bm{z} \in Z$ maps to exactly one $\rootnode$--$\terminalnode$ path with length equal to $\bm{c^\top z}$ and vice versa. 
	
	Decision diagrams are commonly generated from the state-transition graph of a dynamic programming representation of the space \citep{bergman2016}, consisting of a \emph{state space} and a \emph{transition function}. A dynamic program sequentially sets values for the decision variables, storing in each state the relevant information about the partial solution obtained after fixing a subset of the variables. The transition function establishes how the system transitions between states. A key consideration when representing discrete feasible spaces is whether a partial solution can be extended to obtain a complete feasible solution and in how many ways this can be done. 
	
	\begin{definition}[Feasible completion set] \label{def:completionSet}
		Given any partial solution $\bm{v}\in \{0,1\}^{\ell-1}$, define the \underline{feasible completion set associated with $\bm{v}$} as
		$$Z^\ell(\bm{v})\defeq\left\{\bm{w}\in \{0,1\}^{n+1-\ell}: (\bm{v},\bm{w})\in Z\right\}.$$
		By convention we let $Z^1=Z$.\hfill$\blacksquare$
	\end{definition}
	
	A state $\rho^\ell$ at decision stage $\ell$, corresponding to a partial solution $\bm{v}\in \{0,1\}^{\ell-1}$, should ideally characterize $Z^\ell(\bm{v})$, and partial solutions $\bm{v_1}$ and $\bm{v_2}$ for which $Z^\ell(\bm{v_1}) = Z^\ell(\bm{v_2})$ should ideally correspond to the same state. Feasible completion sets can be efficiently computed or approximated for several combinatorial structures (e.g., knapsack constraints); however, characterizing them is challenging in general as even determining if $Z^\ell(\bm{v})\neq\emptyset$ is NP-hard. For spaces defined by linear constraints $\bm{Gz}\leq \bm{b}$, a natural state representation for partial solution $\bm{v}\in \{0,1\}^{\ell-1}$ uses a vector that stores the partial contributions to the left-hand side of the constraints at each decision stage, that is $\rho^\ell(\bm{v})=\bm{G_{[n],[\ell-1]}}\bm{v}$. 
	An illustrative knapsack construction, including its state-transition graph and corresponding decision diagram, is provided in the online companion (\S\ref{sec:knapsack-example}).

	We assume that there is a dynamic programming representation of $Z$ of reasonable size with states $\rho^\ell$, in which the state variable is a vector that records the state of the system after having assigned values to $\ell-1$ variables. In particular, all partial solutions corresponding to the same state must have the same completion set. %
    The state space $S_Z$ includes an initial state $\rho^1$ and an infeasible state ``\text{inf}'', the set of terminal states $T_Z$ includes the infeasible state and every state for which all variables have been already decided, and transition function $\phi_Z: (S_Z \setminus T_Z) \times \{0,1\} \rightarrow S_Z$ is defined as $\rho^{\ell+1} = \phi_Z(\rho^{\ell},\bar{z}_{\ell})$, where 
	\begin{equation}\label{eq:transitionCombinatorial}\phi_Z(\rho^{\ell},\bar{z}_{\ell})=\begin{cases}
			\text{inf}&\text{if } Z^\ell((\bm{v},\bar{z}_{\ell})) = \emptyset \ \text{ for } \bm{v} \text{ corresponding to } \rho^\ell \\
			\psi(\rho^{\ell},\bar{z}_{\ell})&\text{otherwise}\end{cases}\end{equation} 
	and $\psi$ is an update function that guarantees that all partial solutions corresponding to a state must share the same completion set.

	Decision diagrams have been used to model independent set, set covering and matching problems \citep{bergman2016discrete}, sequencing and routing constraints \citep{cire2013multivalued,castro2020mdd}, scheduling applications \citep{cire2019network}, graph coloring \citep{vanhoeve2022graph}, two-stage stochastic and bilevel problems \citep{lozano2018Binary,lozano2025}, and nonlinear objectives over discrete feasible regions \citep{bergman2018discrete,bergman2021dd,castro2020combinatorial}; see \citet{castro2022decision} for a recent survey. 
	Finally, we note that for the unconstrained case where the feasible space is given by  $Z=\{0,1\}^n$, we can define the state space as a singleton $S_{\{0,1\}^n}=\{0\}$ by considering transition function $\phi(\rho^{\ell},\bar{z}_{\ell})=0$, which results in a diagram with only one node per layer.

	\subsection{Linear algebra review}
	\label{sec:linearAlgebra}
	In this section we review the fundamental concepts of linear algebra used throughout the paper, chiefly the Gram-Schmidt process for constructing orthonormal bases of linear subspaces of $\R^m$; we refer the reader to \citet{golub2013matrix} for a comprehensive treatment.
	
	\begin{definition}[Span, column space and row space]
		$\bullet$ Given a set of $m$-dimensional vectors $\{\bm{v_1},\dots,\bm{v_n}\}$, the \underline{span} is the set of all linear combinations,
		$$\text{span}\{\bm{v_1},\dots,\bm{v_n}\}\defeq\left\{\bm{y}\in \R^m: \exists \bm{x}\in \R^n \text{such that }\bm{y}=\sum_{i=1}^n \bm{v_i}x_i\right\}.$$
		$\bullet$ Given a matrix $\bm{F}\in \R^{n\times m}$, the \underline{column space} is the span of all column vectors and the \underline{row space} is the span of all row vectors, that is,
		\begin{align*}\Col(\bm{F})&\defeq\left\{\bm{y}\in \R^n:\exists \bm{x}\in \R^m\text{ s.t. }\bm{y}=\bm{Fx}\right\},\; \text{and}\\
			\Row(\bm{F})&\defeq\Col(\bm{F^\top})=\left\{\bm{y}\in \R^m:\exists \bm{x}\in \R^n\text{ s.t. }\bm{y}=\bm{ F^\top x}\right\}.\blacksquare
		\end{align*}
	\end{definition}
	Observe that for any $\bm{z}\in \{0,1\}^n$, $\Row(\bm{F_z})=\Row\left(\Diag(\bm{z})\bm{F}\right)$ because matrix $\bm{F_z}$ differs from $\Diag(\bm{z})\bm{F}$ only because of the presence of zero rows.
	
	\begin{definition}[Orthonormal basis] An orthonormal basis of a subspace $L\subseteq \R^m$ is a set of orthonormal vectors that span $L$. We say that a matrix $\bm{B}\in \R^{\text{dim}(L)\times m}$ defines an \underline{orthonormal basis} of $L$ if its row vectors are an orthonormal basis of $L$, that is, if $L=\Row(\bm{B})$ and $\bm{BB^\top}=\bm{I}$.$\hfill\blacksquare$
	\end{definition}

	\begin{definition}[Projection onto a subspace] Given a subspace $L\subseteq \R^m$ with orthonormal basis given by $V=\left\{\bm{v_1},\dots,\bm{v_q}\right\}$, the orthogonal projection matrix onto $L$ is $\bm{G}\defeq \sum_{i=1}^q\bm{v_iv_i^\top}$. Equivalently, if $\bm{B}$ defines an orthonormal basis of $L$, then the \underline{orthogonal projection matrix onto $L$} is $\bm{G}=\bm{B^\top B}.\hfill\blacksquare$
	\end{definition}
	Observe that $\bm{G}$ indeed satisfies the two properties of orthogonal projection matrices, that is, $\bm{G^\top}=\bm{G}$ and $\bm{G^2}=\bm{G}$. Given any $\bm{v}\in \R^m$, $\bm{Gv}$ is the closest point to $\bm{v}$ in $L$ (using the usual Euclidean distance). More importantly for this paper, $(\bm{I}-\bm{G})\bm{v}$ is the projection of $\bm{v}$ onto the orthogonal complement $L^\perp$ of $L$, that is, $L^\perp\defeq\{\bm{y}\in \R^m: \bm{y^\top x}=0,\;\forall \bm{x}\in L\}=\{\bm{y}\in \R^m: \bm{By}=\bm{0}\}$. Note that if $\bm{v}\in \text{span}(L)$, then $\bm{Gv}=\bm{v}$ and $(\bm{I}-\bm{G})\bm{v}=\bm{0}$.
	Finally, we point out that although the basis of a subspace is not unique, orthogonal projection matrices are unique and do not depend on the choice of basis. As Proposition~\ref{prop:closedFormProjection} shows, an orthogonal projection matrix
	onto $\Row(\bm{F})$ admits an explicit closed form description if $\bm{F}$ has full row rank. 
	
	\begin{proposition}\label{prop:closedFormProjection}
		Given matrix $\bm{F}\in \R^{n\times m}$, if $\bm{FF^\top}$ is invertible, then the orthogonal projection matrix onto $\Row(\bm{F})$ is $\bm{G}=\bm{F^\top}\left(\bm{FF^\top}\right)^{-1}\bm{F}$.
	\end{proposition}
	
	\subsubsection*{The Gram-Schmidt process}
	
	Given a set of $m$-dimensional vectors $\{\bm{f_1},\dots,\bm{f_n}\}\subset \mathbb{R}^m$, the Gram-Schmidt process recursively constructs an orthonormal basis for $\text{span}\{\bm{f_1},\dots,\bm{f_n}\}$. Let $\bm{F}\in \R^{n\times m}$ be the matrix whose $i$-th row corresponds to $\bm{f_i}$. The Gram-Schmidt process relies on the orthogonal decomposition theorem, which states that given a vector $\bm{v}\in \R^m$ and subspace $L\subseteq \R^m$, this vector can be written uniquely as a sum of a vector in $L$ and a vector in the orthogonal complement $L^\perp$. Letting $\bm{G}$ be the projection matrix onto $L$, this theorem corresponds to the identity $\bm{v}=\bm{Gv}+(\bm{I}-\bm{G})\bm{v}$. It follows that since $\bm{Gv}\in L$, we have that $L+\text{span}\{\bm{v}\}=L+\text{span}\{(\bm{I}-\bm{G})\bm{v}\}$. Proposition~\ref{prop:GramSchmidt} states this property more formally.
	\begin{proposition}\label{prop:GramSchmidt}
		Given any $\ell\in [2:n+1]$, let $\bm{G^\ell}$ be the orthogonal projection matrix onto $\Row\left(\bm{F_{[\ell-1]}}\right)$. Moreover, let $\bm{F_\perp^\ell}\defeq\bm{F_{[\ell:n]}}\left(\bm{I}-\bm{G^\ell}\right)$ be the matrix whose $i$-th row corresponds to the projection of vector $\bm{f_{\ell+i-1}}$ onto the orthogonal complement of $\Row\left(\bm{F_{[\ell-1]}}\right)$. Then
		$$\Row(\bm{F})=\Row\left(\bm{F_{[\ell-1]}}\right)+\Row\left(\bm{F_\perp^\ell}\right).$$
	\end{proposition}
	
	The Gram-Schmidt process sequential applies the following inductive property.
	\begin{proposition}\label{prop:GSInduction}
		Suppose that $V=\{\bm{v_1},\dots,\bm{v_q}\}$ is an orthonormal basis for $\Row(\bm{F_{[\ell-1]}})$ (implying $\bm{G^\ell}=\sum_{i=1}^q\bm{v_{i}v_i^\top}$), let matrix $\bm{H}=\bm{F_{[\ell:n]}}\left(\bm{I}-\bm{G^\ell}\right)$, and let $\bm{\bar v^\top}=\bm{H_{\{1\}}}$ be the first row of $\bm{H}$. Then:
		\begin{enumerate}
			\item If $\bm{\bar v}=\bm{0}$, then $V$ is an orthonormal basis for $\Row(\bm{F_{[\ell]}})$ and $$\bm{H_\downarrow}=\bm{F_{[\ell+1:n]}}\left(\bm{I}-\bm{G^{\ell+1}}\right).$$
			\item If $\bm{\bar v}\neq \bm{0}$, then $V\cup\{\bm{\bar v}/\|\bm{\bar v}\|_2\}$ is an orthonormal basis for $\Row(\bm{F_{[\ell]}})$ and $$\bm{H_\downarrow}\left(\bm{I}-\bm{\bar v\bar v^\top}/\|\bm{\bar v}\|_2^2\right)=\bm{F_{[\ell+1:n]}}\left(\bm{I}-\bm{G^{\ell+1}}\right).$$
		\end{enumerate}
	\end{proposition}
	
	The modified Gram-Schmidt process is shown in Algorithm~\ref{alg:MGS}. The inductive steps in Proposition~\ref{prop:GSInduction} are repeated in lines~\ref{line:remove}--\ref{line:endIf}. The output of this algorithm is a set of vectors $V$ that forms an orthonormal basis for $\Row(\bm{F})$. %
	
	\begin{algorithm}[H]
		\singlespacing
		\caption{Modified Gram-Schmidt process}
		\label{alg:MGS}
		\begin{algorithmic}[1]
			\Require Matrix $\bm{F}\in \R^{n\times m}$.
			\Ensure Set $V$ containing the orthonormal vectors.
			
			\State $V=\emptyset$  \Comment{Set of orthonormal vectors, initially empty}
			
			\State $\bm{H}\gets \bm{F}$ \Comment{Vectors to be processed}

			\For{$\ell=1,\dots,n$}\label{line:loop}
			
			\State $\bm{v^\top}\gets \bm{H_{\{1\}}},\; \bm{H}\gets \bm{H_\downarrow}$\Comment{Removes the first row $\bm{v}$ of $\bm{H}$}\label{line:remove}
			
			\If{$\bm{v}\neq \bm{0}$} \Comment{$\bm{v}=\bm{0}$ if and only if $\bm{f_\ell}\in \text{span}(\bm{V})$}\label{line:isNotZero}
			
			\State $\bm{v}\gets \bm{v}/\|\bm{v}\|_2$ \Comment{Normalizes vector $\bm{v}$}
			
			\State $\bm{H}\gets \bm{H}\left(\bm{I}-\bm{vv^\top}\right)$
			\Comment{Projects row vectors of $\bm{H}$ onto $\text{span}\{\bm{v}\}^\perp$}\label{line:projection}
			
			\State $V\gets V\cup\{\bm{v}\}$ \Comment{Adds vector to set of orthonormal vectors}\label{line:addBasis}
			\EndIf\label{line:endIf}

			\EndFor
		\end{algorithmic}
	\end{algorithm}

	\section{Decision diagrams for MIQO}\label{sec:DDMIQO}
	
	In this section we consider optimization of a linear function over set $X_{\bm{F},Z}$:
	\begin{equation}\label{eq:optimizationX}
			\min_{\substack{\bm{x}\in \R^n,\bm{z}\in \{0,1\}^n\\\bm{x}\circ (\bm{1}-\bm{z})=\bm{0}}}\;\bm{d^\top x}+\bm{c^\top z}+\left\|\bm{F^\top x}\right\|_2^2
	\end{equation}
	Problem \eqref{eq:optimizationX} corresponds to problem \eqref{prob:miqp} without the linear constraints \eqref{eq:miqp_linear} and with the epigraph variable $x_0$ projected out. We discuss identification of whether \eqref{eq:optimizationX} has an optimal solution or is unbounded, and then develop the decision diagrams used to compute optimal solutions. 

	\subsection{Detecting unboundedness}
	We first discuss unboundedness of problem \eqref{eq:optimizationX}. Note that if there exists $\bm{\bar z}\in Z$ and $\bm{\bar x}\in \R^n$ such that $\bm{F^\top}\Diag(\bm{\bar z})\bm{\bar x}=\bm{0}$ and $\bm{d^\top}\Diag(\bm{\bar z})\bm{\bar x}=-1$, then the problem is unbounded by letting $\bm{z}=\bm{\bar z}$ and $\bm{x}=\lambda (\bm{\bar x}\circ \bm{\bar z})$ with $\lambda\to\infty$. In fact, problem \eqref{eq:optimizationX} is unbounded if and only if such a pair $(\bm{\bar z},\bm{\bar x})$ exists and, by Farkas' lemma, no such pair exists precisely when
	\begin{equation}\label{eq:FarkasUnbounded}
		\forall \bm{z}\in Z,\; \exists\bm{\delta_z}\in \R^m \text{ such that }  \Diag(\bm{z})\left(\bm{F\delta_z}-\bm{d}\right)=\bm{0}.
	\end{equation}
	
	Depending on the structure of set $Z$, detecting whether \eqref{eq:optimizationX} is unbounded can be a difficult task. Indeed, suppose that $Z=\left\{\bm{z}\in \{0,1\}^n: \sum_{i=1}^{n-1}z_i\leq k, z_n=1\right\}$. Then the conditions $\bm{F^\top}\Diag(\bm{\bar z})\bm{\bar x}=\bm{0}$ and $\bm{d^\top}\Diag(\bm{\bar z})\bm{\bar x}=-1$ amount to asking whether there exists a $k$-sparse vector satisfying a system of linear equalities, which is precisely the decision variant of the ``Minimum relevant variables in linear system'', known to be NP-complete \citep{garey1979computers}.
	
	In most of the situations we are concerned with, unboundedness is not a concern because the problem is known to be bounded a priori, or the constraints are such that unbounded situations are easy to identify. In particular, if $\bm{d}\in \Col(\bm{F})$, then problem \eqref{eq:optimizationX} is not unbounded: indeed, $\bm{d}\in \Col(\bm{F})$ if and only if there exists $\bm{\delta}\in \R^m$ such that $\bm{F\delta}=\bm{d}$, and this single vector $\bm{\delta}$ serves as a certificate that all conditions \eqref{eq:FarkasUnbounded} are satisfied.
	Under Assumption~\ref{ass:unbounded} below, this condition would also be necessary for problem~\eqref{eq:optimizationX} not to be unbounded. %
	\begin{assumption}\label{ass:unbounded}
		Matrix $\bm{F}$ and constraints $Z$ are such that problem \eqref{eq:optimizationX} is not unbounded if and only if there exists $\bm{\delta}\in \R^m$ such that $\bm{F\delta}=\bm{d}$.	
	\end{assumption}
	Assumption~\ref{ass:unbounded} is trivially satisfied if $\bm{1}\in Z$, since conditions \eqref{eq:FarkasUnbounded} with $\bm{z}=\bm{1}$ require the existence of this vector. The condition is also satisfied if $\bm{F}$ has full row rank and thus $\bm{FF^\top}\succ \bm{0}$, since in this case the problems are never unbounded and $\bm{\delta}=\bm{F^\top}\left(\bm{FF^\top}\right)^{-1}\bm{d}$ satisfies Assumption~\ref{ass:unbounded}. 
    
	\subsection{Towards a dynamic program for optimization}
	
	Now suppose that Assumption~\ref{ass:unbounded} holds and problem \eqref{eq:optimizationX} is not unbounded. In this case we find that 
	\begin{align}
		&\min_{\substack{\bm{x}\in \R^n,\bm{z}\in Z\\ \bm{x}\circ(\bm{1}-\bm{z})=\bm{0}}}\;\bm{d^\top x}+\bm{c^\top z}+\left\|\bm{F^\top x}\right\|_2^2 \notag\\
		=&\min_{\substack{\bm{x}\in \R^n,\bm{z}\in Z\\ \bm{x}\circ(\bm{1}-\bm{z})=\bm{0}}}\;\bm{\delta^\top F^\top x}+\bm{c^\top z}+\left\|\bm{F^\top x}\right\|_2^2\tag{$\because$ Assumption~\ref{ass:unbounded}}\notag\\
		=&\min_{\substack{\bm{y}\in \R^m,\bm{z}\in Z\\\bm{y}\in \Row(\Diag(\bm{z})\bm{F})}}\;\bm{\delta^\top y}+\bm{c^\top z}+\left\|\bm{y}\right\|_2^2\tag{$\because\; \bm{y}=\bm{F^\top x}=\left(\Diag(\bm{z})\bm{F}\right)^\top\bm{x}$}\notag\\
		=&\min_{\substack{\bm{y}\in \R^m,\bm{z}\in Z\\\bm{y}\in \Row(\bm{F_z})}}\;\bm{\delta^\top y}+\bm{c^\top z}+\left\|\bm{y}\right\|_2^2\label{eq:optXSimple}
	\end{align}
	where the last equality holds since $\bm{F_z}$ and $\Diag(\bm{z})\bm{F}$ have the same non-zero rows.
	To define a dynamic programming algorithm, suppose that the first $\ell-1$ coordinates of $\bm{z}$ are fixed, that is, $\bm{z_{[\ell-1]}}=\bm{v}$ for some $\bm{v}\in \{0,1\}^{\ell-1}$. In addition to the completion sets given in Definition~\ref{def:completionSet}, we provide two additional quantities related to the state of solving optimization problem \eqref{eq:optimizationX}.

	\begin{definition}[Partial projection matrix onto complement] \label{def:proj}
		Given any partial solution $\bm{v}\in \{0,1\}^{\ell-1}$, let $\bm{B^\ell}(\bm{v})\in \R^{q\times m}$, with $q=\text{rank}(\bm{F_v})$, define an orthonormal basis of $\Row(\bm{F_v})$, and
		define the \underline{projection matrix associated with $\bm{v}$} as
		$$\bm{G_\perp}^\ell(\bm{v})\defeq \bm{I}_{m\times m}-\bm{B^\ell}(\bm{v})^\top\bm{B^\ell}(\bm{v}).$$
		By convention we let $\bm{G_\perp^1}=\bm{I}$. \hfill $\blacksquare$
	\end{definition}

	\begin{definition}[Partial cost]\label{def:functionG}
		\leavevmode\par
		Given any partial solution $\bm{v}\in \{0,1\}^{\ell-1}$, define the \underline{objective cost associated with $\bm{v}$} as 
		\begin{equation}\label{eq:defG}
			\begin{aligned}
				g^\ell(\bm{v})\defeq\bm{c_{[\ell-1]}^\top v}+\min_{\bm{y}\in \Row\left(\bm{F_v}\right)}\;&\bm{\delta^\top y}+\left\|\bm{y}\right\|_2^2
			\end{aligned}
		\end{equation}
		By convention we let $g^1=0$. \hfill $\blacksquare$
	\end{definition}

	Observe that function $g^\ell$ corresponds to the best cost that can be attained assuming that $\bm{z_{[\ell:n]}}=\bm{0}$ and constraints $\bm{z}\in Z$ are ignored. Note that problem \eqref{eq:optXSimple} corresponds to $\min_{\bm{z}\in Z}g^{n+1}(\bm{z})$. As Lemma~\ref{prop:Gexplicit} shows, since $g^\ell$ corresponds to the minimization of a relatively simple convex function, it can be computed in closed form.
	
	\begin{lemma}\label{prop:Gexplicit}
		Letting $\bm{B^\ell}(\bm{v})$ be any orthonormal basis of $\Row(\bm{F_v})$, function $g^\ell(\bm{v})$ satisfies
		$$g^\ell(\bm{v})=\bm{c_{[\ell-1]}^\top v}-\frac{1}{4}\bm{\delta^\top B^\ell}(\bm{v})^\top \bm{B^\ell}(\bm{v})\bm{\delta}.$$
	\end{lemma}
	
	Function $g^\ell$ completely ignores constraints $\bm{z}\in Z$ and assumes the variables $\bm{z_{[\ell:n]}}$ are fixed. Removing these restrictions,
	the best possible objective that can be achieved by changing the remaining variables $\bm{z_{[\ell:n]}}$ and $\bm{y}$ while fixing $\bm{z_{[\ell-1]}}=\bm{v}$ is
	\begin{equation}\label{eq:optFixedV}
		\begin{aligned}
			h^\ell(\bm{ v})\defeq 
			\bm{c_{[\ell-1]}^\top v}+\min_{\substack{\bm{w}\in \{0,1\}^{n+1-\ell}\\\bm{y}\in \R^m}}\;&\bm{c_{[\ell:n]}^\top w}+\bm{\delta^\top y}+\left\|\bm{y}\right\|_2^2\\
			\text{ s.t. }\;&\bm{w}\in Z^\ell(\bm{v}),\;\bm{y}\in \Row\left(\bm{F_{(v,w)}}\right)
		\end{aligned}
	\end{equation}
	where we use the convention that $h^\ell\left(\bm{ v}\right)=\infty$ when constraints  $\bm{w}\in Z^\ell(\bm{v})$ cannot be satisfied for the fixed choice of $\bm{v}$. We now show in Proposition~\ref{prop:hDecomposition} that function $h^\ell$ can be decomposed as the sum of two simpler functions, one corresponding to the direct value incurred by fixing $\bm{z_{[\ell-1]}}=\bm{v}$, and the other corresponding to future value by optimally choosing  variables $\bm{z_{[\ell:n]}}$.

	\begin{proposition}\label{prop:hDecomposition}
		For $\ell\in [2:n]$, function $h^\ell$ can be rewritten as
		{\small\begin{equation}\label{eq:hDecomposition}
				\begin{aligned}
					h^\ell(\bm{v})=
					g^\ell(\bm{v})+\min_{\substack{\bm{w}\in\{0,1\}^{n+1-\ell}\\
							\bm{y}\in \R^m}}\;&\bm{c_{[\ell:n]}^\top w}+\bm{\delta^\top y}+\left\| \bm{y}\right\|_2^2\\
					\text{s.t.}\;&\bm{w}\in Z^\ell(\bm{v}),\;\bm{y}\in\Row\left(\Diag(\bm{w})\bm{F_{[\ell:n]}}\bm{G_\perp^\ell}(\bm{v})\right),
				\end{aligned}
		\end{equation}}
		where $\bm{G_\perp^\ell}(\bm{v})$ and $g^\ell(\bm{v})$ are defined in Definition~\ref{def:proj} and \ref{def:functionG}, respectively.
	\end{proposition}

Intuitively, the second term in \eqref{eq:hDecomposition} is the best potential additional improvement over $g^\ell(\bm{v})$ obtained by changing variables $\bm{z_{[\ell:n]}}$. 
	Obviously, this minimization problem in \eqref{eq:hDecomposition} is hard to compute, due both to the presence of combinatorial constraints $\bm{w}\in Z^\ell(\bm{v})$ and non-trivial linking constraints $\bm{y}\in\Row\left(\Diag(\bm{w})\bm{F_{[\ell:n]}}\bm{G_\perp^\ell}(\bm{v})\right)$. However, the key property we exploit in dynamic programming
    is that the inner minimization problem depends on $\bm{v}$ only through sets $Z^\ell(\bm{v})$ and matrix $\bm{F_{[\ell:n]}}\bm{G_\perp^\ell}(\bm{v})$. Thus, if several different partial solutions result in the same set and matrix, the inner minimization in \eqref{eq:hDecomposition} needs to be computed only once, reducing the overall complexity.

	\subsection{Decision diagram}\label{sec:dd}
	We are now ready to define a decision diagram for tackling problem \eqref{eq:optimizationX}.  The state variables at layer $\ell$ will have two components: a component representing the completion set $Z^\ell(\bm{v})$, and a $(n+1-\ell)\times m$ matrix representing $\bm{F_{[\ell:n]}}\bm{G_\perp^\ell}(\bm{v})$ where $\bm{v}\in \{0,1\}^{\ell-1}$ is a vector denoting the partial assignment of the indicator variables. 
	
	More formally, we assume the existence of a dynamic programming representation of $Z$ of reasonable size, involving a transition function $\phi_Z$ as defined in \eqref{eq:transitionCombinatorial}. We define the initial state $s^1=\{\rho^1,\bm{F}\}$ and the transition functions as follows. 
	Given any state $s^\ell=\{\rho^\ell,\bm{H^\ell}\}$ where $\bm{H^\ell}\in \R^{(n+1-\ell)\times m}$, let $\bm{u^\top}=\bm{H_{\{1\}}^\ell}$ denote the first row of $\bm{H^\ell}$, let $s^{\ell+1}=\phi(s^\ell,\bar z_\ell)$ where $\bar z_\ell$ is the value assigned to variable $z_\ell$ and
	\begin{align}\label{eq:transition}
		\phi(s^\ell,\bar z_\ell)=\begin{cases}\text{inf}&\text{if }\phi_Z(\rho^\ell,\bar z_\ell)=\text{inf}\\
			\{\psi(\rho^\ell,\bar z_\ell),\bm{H_{\downarrow}^\ell}\}&\text{if }\bar z_\ell=0\text{, or }\bar z_\ell=1 \text{ and }\bm{u}=\bm{0}\\
			\{\psi(\rho^\ell,\bar z_\ell),\bm{H_{\downarrow}^\ell}\left(\bm{I}-\bm{uu^\top}/\|\bm{u}\|_2^2\right)\}&\text{if }\bar z_\ell=1 \text{ and }\bm{u}\neq\bm{0}.
		\end{cases}
	\end{align}
	Observe that the first component of the transition function \eqref{eq:transition} is used to track the status of completion set $Z^\ell(\bm{v})$; the second component replicates the steps of the Gram--Schmidt process, lines~\ref{line:remove}-\ref{line:endIf} of Algorithm~\ref{alg:MGS} when $\bar z_\ell=1$, and corresponds to the status of matrix $\bm{F_{[\ell:n]}}\bm{G_\perp^\ell}(\bm{v})$. 
	According to Proposition~\ref{prop:hDecomposition}, these two components are sufficient to compute the inner minimization of function $h^\ell(\bm{v})$. In the unconstrained case with $Z=\{0,1\}^n$, we can ignore the first component and write more compactly
	\begin{align}\label{eq:transitionUnconstrained}
		\phi(s^\ell,\bar z_\ell)=\begin{cases}
			\{\bm{H_{\downarrow}^\ell}\}&\text{if }\bar z_\ell=0\text{, or }\bar z_\ell=1 \text{ and }\bm{u}=\bm{0}\\
			\{\bm{H_{\downarrow}^\ell}\left(\bm{I}-\bm{uu^\top}/\|\bm{u}\|_2^2\right)\}&\text{if }\bar z_\ell=1 \text{ and }\bm{u}\neq\bm{0}.
		\end{cases}
	\end{align}

	\begin{definition}[Exact decision diagram] \label{def:exactDD} Given $\bm{F}$ and $Z$, define the state transition graph $\mathcal{G}=(N,A)$ by applying the state transition function $\phi$ in \eqref{eq:transition} recursively, starting from the initial state $s^1$. Define the exact decision diagram $\mathcal{D}=(N,A,\bm{\nu},\bm{u})$ for $X_{\bm{F},Z}$ as the resulting graph in which each node corresponds to a state, there is a directed arc $(s^i,\phi(s^i, \bar z))$ for  $\bar z\in \{0,1\}$, the value assignment function $\bm{\nu}:A\to \{0,1\}$ is such that $\nu_{(s^i,\phi(s^i, \bar z))}=\bar z$, and the transition vectors $\bm{u}:A\to \R^m$ are such that $\bm{u_a}=\bm{0}$ if $\nu_a=0$ or $\bm{u}=\bm{0}$, and $\bm{u_a}=\bm{u}/\|\bm{u}\|_2$ otherwise, where $\bm{u^\top}=\bm{H_{\{1\}}^\ell}$ and $\bm{H^\ell}$ is the matrix stored by the tail node $t_{a}$.\hfill$\blacksquare$
	\end{definition}

    \begin{wrapfigure}{r}{0.31\textwidth}
			\vspace{-0.5\baselineskip}
			\centering
			\includegraphics[width=\linewidth,trim={2.8cm 20.1cm 11.5cm 2.6cm},clip]{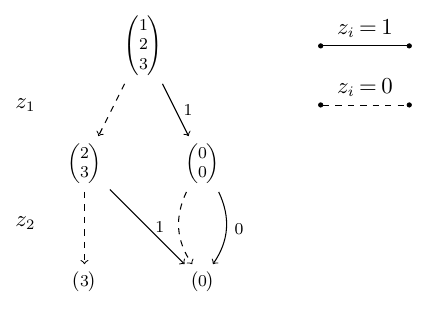}
			\caption{Exact DD for the rank-one instance.}
			\label{fig:rank-one}
			\vspace{-1\baselineskip}
		\end{wrapfigure}
	Decision diagrams constructed according to Definition~\ref{def:exactDD} are in general of exponential size even if $Z=\{0,1\}^n$. Nevertheless, we now provide three examples where some reduction in the size of the decision diagram occurs. In \S\ref{sec:polycase} we indeed formally show that the examples illustrate settings where the decision diagrams can be shown to be of polynomial size, as shown in \S\ref{sec:polycase}. 
	
	\begin{example}[Rank-one case]\label{ex:rank-one}
		Consider the setting where $Z=\{0,1\}^n$ and $\bm{F}=\begin{pmatrix}1\\2\\3\end{pmatrix}$, in which case matrix $\bm{Q}=\bm{FF^\top}=\begin{pmatrix}1 & 2& 3\\
			2&4&6\\
			3&6&9\end{pmatrix}$ has rank one.
Figure~\ref{fig:rank-one} shows the resulting exact decision diagram, where each node is represented by its state and transition vectors are shown next to bold arcs (other arcs, corresponding to decisions $z_i=0$ have zero transition vectors). Since $F$ has $m=1$ columns, the state matrices are vectors and the transition vectors are scalars. In this setting, every arc with value assignment $\nu_a=1$ leads to an absorbing state $\bm{H^\ell}=\bm{0}$, resulting in a substantial reduction of the number of states. 
        \hfill $\blacksquare$
	\end{example}
	
	\begin{example}[Tridiagonal case] \label{ex:tridiagonal}
    
		Consider the setting where $Z=\{0,1\}^n$ and $\bm{F}=\begin{pmatrix}1& 0 &0\\-1&1&0\\0&-1&1\end{pmatrix}$, in which case matrix $\bm{Q}=\bm{FF^\top}=\begin{pmatrix}1 & -1& 0\\
			-1&2&-1\\
			0&-1&1\end{pmatrix}$ is tridiagonal. 	\end{example}

		\begin{wrapfigure}{r}{0.40\textwidth}
			\vspace{-0.8\baselineskip}
			\centering
			\includegraphics[width=\linewidth,trim={2.8cm 18.1cm 8.4cm 2.6cm},clip]{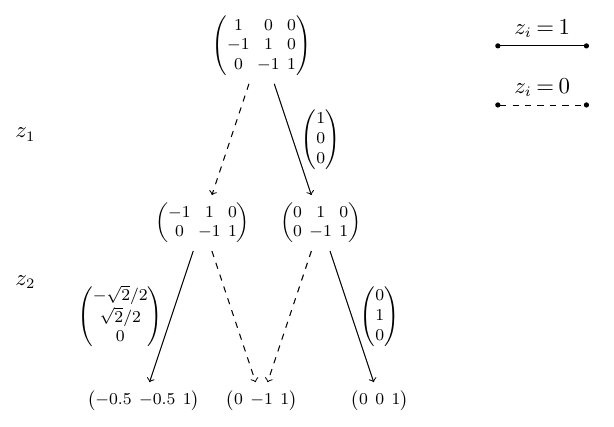}
			\caption{Exact DD for the tridiagonal instance.}
			\label{fig:tridiag}
			\vspace{-1\baselineskip}
		\end{wrapfigure}
	\noindent Figure~\ref{fig:tridiag} illustrates this sparsity pattern: the support of any row vector $\bm{f_i}$ does not overlap with the support of rows $\bm{f_j}$ with $j\geq i+2$ (in this example with $n=3$, the supports of first and third rows do not overlap). These vectors are thus orthogonal, and the Gram-Schmidt orthogonalization associated with setting $z_\ell=1$ affects only the following row. As a consequence, any decision $z_\ell=0$ reverts to the base state $\bm{H^{\ell+1}}=\bm{F_{[\ell+1:n]}}$, regardless of state $\bm{H^\ell}$, resulting in a compression of the diagram. %
\hfill $\blacksquare$

	\begin{example}[Inverse tridiagonal case]\label{ex:inverseTridiagonal}
		Consider the setting where $Z=\{0,1\}^n$ and $\bm{F}=\begin{pmatrix}1& 1/2 &1/4\\0&1&1/2\\0&0&1\end{pmatrix}$, in which case matrix $\bm{Q}=\bm{FF^\top}\approx\begin{pmatrix}1.31 & 0.63& 0.25\\
			0.63&1.25&0.50\\
			0.25&0.50&1\end{pmatrix}$ is full rank (shown with two digits of precision), and $\bm{Q^{-1}}=\begin{pmatrix}1 & -1/2& 0\\
			-1/2&5/4&-1/2\\
			0&-1/2&5/4\end{pmatrix}$ is tridiagonal. Figure~\ref{fig:invTridiag} depicts the decision diagram for $X_{\bm{F}}$. The decision diagram in this setting is a mirror image of the tridiagonal case discussed in Example~\ref{ex:tridiagonal}, with a state reset associated with decisions $z_\ell=1$.\hfill $\blacksquare$
		\begin{figure}[!ht]
			\begin{center}
				\includegraphics[width=0.7\textwidth, trim={2cm 18cm 5.5cm 2cm},clip]{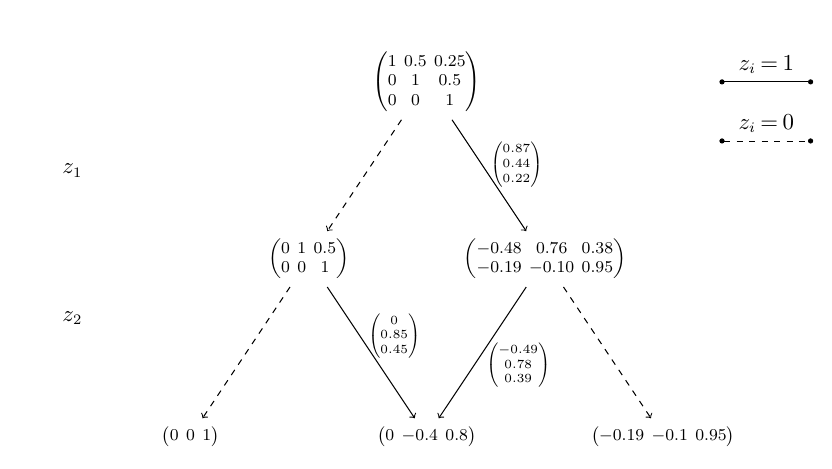}
				\caption{Decision diagram associated with Example~\ref{ex:inverseTridiagonal} at depth $\ell=3$. For each node we show the associated state $\bm{H^\ell}$, and for each arc with value assignment $\nu_a=1$ we show the associated transition vector $\bm{u_a}$; transition vectors of remaining arcs are $\bm{0}$. Numbers are depicted with two decimal digits of precision.}
				\label{fig:invTridiag}
			\end{center}
		\end{figure}
	\end{example}
	
	We now introduce the two fundamental properties of exact decision diagrams. The first one (Proposition~\ref{prop:state}) states that the matrix stored at each state indeed corresponds to $\bm{F_{[\ell:n]}}\bm{G_\perp^\ell}(\bm{v})$ for a suitable vector $\bm{v}$. The second one (Proposition~\ref{prop:pathLength}) discusses how to solve optimization problem \eqref{eq:optimizationX} as a shortest path problem on the decision diagram.
	
	\begin{proposition}\label{prop:state} Given any arc-specified path $(a_1,a_2,\ldots,a_{\ell-1})$ in $\mathcal{G}$ such that $\tail{a_1}=s^1$ and $\head{a_{\ell-1}}=s^\ell$, let $\bm{\bar H^\ell}$ be the matrix stored by state $s^\ell$ and $\bm{v}\in \{0,1\}^{\ell-1}$ be the partial solution represented by the path, that is, $v_i=\nu_{a_i}$. Then,
		\begin{enumerate}
			\item the non-zero transition vectors in $\{\bm{u_{a_1}},\dots, \bm{u_{a_{\ell-1}}}\}$ define an orthonormal basis of $\Row\left(\bm{F_v}\right)$;
			\item the identity $\bm{\bar H^\ell}=\bm{F_{[\ell:n]}}\bm{G_\perp^\ell}(\bm{v})$ holds.
		\end{enumerate}
	\end{proposition}
	
	\begin{proposition}\label{prop:pathLength} Assume that Assumption~\ref{ass:unbounded} holds. Given an exact decision diagram $\DD$, vectors $\bm{c},\bm{d}\in \R^n$ such that problem \eqref{eq:optimizationX} is not unbounded, and vector $\bm{\delta}\in \R^m$ satisfying $\bm{F\delta}=\bm{d}$, define the length of an arc $a\in A$ as $l_a=c_{\ell(a)}\nu_{a}-\frac{1}{4}\left(\bm{\delta^\top u_{a}}\right)^2$. Then the length of any arc-specified path $(a_1,a_2,\ldots,a_{\ell-1})$ in $\mathcal{G}$  is the partial cost $g^\ell(\bm{v})$ defined in Definition~\ref{def:functionG}, where 
		$\bm{v}\in \{0,1\}^{\ell-1}$ is the partial solution represented by the path, that is, the vector satisfying $v_i=\nu_{a_i}$. In particular, any shortest path between the root and a terminal node corresponds to an optimal solution of~\eqref{eq:optimizationX}.
	\end{proposition}

	Finally, we close this section by recording a closed-form expression for the state matrices of the decision diagram, which holds whenever $\bm{Q}$ is invertible. 
	
	\begin{lemma}[Closed form of the states]\label{lem:closedFormState}
		Suppose that $\bm{Q}=\bm{FF^\top}$ is invertible. For any $\ell\in [2:n+1]$ and $\bm{v}\in \{0,1\}^{\ell-1}$ with support $T=\left\{i\in [\ell-1]: v_i=1\right\}$, the state matrix of the node of the exact decision diagram corresponding to $\bm{v}$ is
		\[
		\bm{H^\ell}(\bm{v}) = \bm{F_{[\ell:n]}}-\bm{Q_{[\ell:n],T}} \left(\bm{Q_{T,T}}\right)^{-1} \bm{F_{T}},
		\]
		with the convention that $\bm{H^\ell}(\bm{v})=\bm{F_{[\ell:n]}}$ if $T=\emptyset$.
	\end{lemma}
	
	\section{Pseudocode and implementation details}\label{sec:pseudocode}
	
	In this section, we provide pseudocode for the construction of the decision diagrams introduced in \S\ref{sec:dd}, and discuss the associated implementation details. The algorithm builds the state transition graph layer by layer: it applies the transition function \eqref{eq:transition} to every node of the current layer, and then merges the nodes of the next layer that represent equivalent states. The pseudocode assumes that standard finite-precision libraries are used to carry out the matrix operations, and explicitly encodes the operations whose outcome depends on a numerical tolerance $\epsilon\geq 0$. Nonetheless, as we discuss at the end of the section, if $\bm{F}\in \mathbb{Q}^{n\times m}$, then the decision diagram can also be constructed using exact rational arithmetic.
	
	\paragraph{Data structures} The algorithm maintains a set of nodes and a set of arcs. A \emph{node} is a tuple $(i,\rho,\bm{H})$, where $i\in \Z_+$ is a unique identifier of the node, $\rho$ is a state of the dynamic programming representation of $Z$ (see \S\ref{sec:ddComb}), and $\bm{H}\in \R^{(n+1-\ell)\times m}$ is the state matrix of the node, with $\ell$ denoting the layer of the node. An \emph{arc} is a tuple $(i,j,\ell,\bar z,\bm{u})$, where $i$ and $j$ are the identifiers of the tail and head nodes, $\ell\in [n]$ is the layer from which the arc emanates, $\bar z\in \{0,1\}$ is the value assigned to variable $z_\ell$, and $\bm{u}\in \R^m$ is the associated transition vector. In particular, the value assignments $\bm{\nu}$ and transition vectors $\bm{u}$ of Definition~\ref{def:exactDD} are encoded directly in the arcs. In the unconstrained case $Z=\{0,1\}^n$, the component $\rho$ is constant and can simply be omitted.
	
	\paragraph{Elementary operation} The pseudocode relies on the following elementary operation. 
	
	\noindent $\bullet$ $\texttt{merge}(L,A,\epsilon)$: given a list of nodes $L$ and a list of arcs $A$, while two nodes $(i_1,\rho_1,\bm{H_1}),\allowbreak (i_2,\rho_2,\bm{H_2})\in L$ \emph{are identified to satisfy} $\rho_1=\rho_2$ and $\norm{\bm{H_1}-\bm{H_2}}_{2,\infty}\leq \epsilon$: remove node $i_2$ from $L$, and redirect every arc of $A$ with head $i_2$ to head $i_1$. 

	The wording ``are identified to satisfy'' in the routine \texttt{merge} is intentionally vague, as it depends on the implementation.
	A first implementation, which we call \underline{maximal merging}, compares all pairs of nodes in the layer, requiring a quadratic number of state comparisons. In \S\ref{sec:approximation} we describe a second implementation of merging called \underline{neighborhood merging}, which is based on exploiting sparsity of matrix $\bm{FF^\top}$ or its inverse, and is useful for proving approximation guarantees. However that implementation may too be impractical, as it requires a full a priori understanding of $\bm{FF^\top}$ and its application may lead to unnecessarily large decision diagrams.
	In practice, we implement \texttt{merge} using a single-pass hash-merging procedure described in the online companion (\S\ref{sec:hash-merging}, Algorithm~\ref{alg:mergeHash}); we refer to it as \underline{hash merging}.

\paragraph{Construction algorithm} The complete procedure is summarized in Algorithm~\ref{alg:constructionDD}. It processes the layers sequentially, maintaining the list $L$ of nodes in the current layer and the list $L'$ of nodes in the next layer. For each node of the current layer and each feasible decision $\bar z\in \{0,1\}$ for variable $z_\ell$ (line~\ref{line:feasCheck}), a new node and arc are created according to the transition function \eqref{eq:transition}: the combinatorial component of the state is updated through $\phi_Z$, and the state matrix is updated by removing its first row (if $\bar z=0$, or if $\bar z=1$ and $\bm{u}=\bm{0}$) or by performing the Gram-Schmidt projection step of Algorithm~\ref{alg:MGS} (if $\bar z=1$ and $\bm{u}\neq \bm{0}$, line~\ref{line:GSStep}). Once all nodes of the current layer have been processed, equivalent nodes of the next layer are merged (line~\ref{line:mergeCall}). After processing layer $n$, the state matrices are empty and all remaining nodes are merged into a terminal node.
	
	\begin{algorithm}[H]
		\singlespacing
		\caption{Construction of the decision diagram for $X_{\bm{F},Z}$}
		\label{alg:constructionDD}
		\begin{algorithmic}[1]
			\Require Matrix $\bm{F}\in \R^{n\times m}$; initial state $\rho^1$ and transition function $\phi_Z$ of a dynamic programming representation of $Z$; tolerance $\epsilon\geq 0$.
			\Ensure Node set $N$ and arc set $A$ encoding a decision diagram for $X_{\bm{F},Z}$.
			
			\State $N\gets \emptyset,\; A\gets \emptyset,\; idx\gets 0$ \Comment{Sets of nodes and arcs, initially empty}
			
			\State $L\gets \left\{(idx,\rho^1,\bm{F})\right\}$, $idx\gets idx+1$ \Comment{Layer $1$ consists of the root node}
			
			\For{$\ell=1,\dots,n$}
			
			\State $N\gets N\cup L$, $L'\gets \emptyset$ \Comment{$L'$ collects the nodes of layer $\ell+1$}
			
			\For{\textbf{each} $(i,\rho,\bm{H})\in L$}
			
			\State $\bm{u}\gets \bm{H_{\{1\}}}$, $\bm{R}\gets \bm{H_\downarrow}$ \Comment{Removes the first row $\bm{u}$ of $\bm{H}$}
			
			\For{$\bar z\in \{0,1\}$} \Comment{Attempts both decisions for variable $z_\ell$}
			
			\If{$\phi_Z(\rho,\bar z)\neq \text{inf}$} \label{line:feasCheck}\Comment{Skips decisions that are infeasible for $Z$}
			
			\If{$\bar z=1$ \textbf{and} $\|\bm{u}\|_2> \epsilon$}\label{line:dependenceCheck}
			
			\State $\bm{u_a}\gets \bm{u}/\|\bm{u}\|_2$ \Comment{Normalized transition vector}\label{line:normalization}
			
			\State $\bm{G}\gets \bm{R}\left(\bm{I}-\bm{u_au_a^\top}\right)$ \label{line:GSStep}\Comment{Gram-Schmidt step}
			
			\Else
			
			\State $\bm{G}\gets \bm{R}$ \Comment{State matrix is unchanged}
			
			\State $\bm{u_a}\gets \bm{0}$ \Comment{Transition vector is $\bm{0}$}
			
			\EndIf
			
			\State $L'\gets L'\cup \left\{\left(idx,\phi_Z(\rho,\bar z),\bm{G}\right)\right\}$ \Comment{Appends new node to $L'$}
			
			\State $A\gets A\cup \left\{(i,idx,\ell,\bar z, \bm{u_a})\right\}$
			
			\State  $idx\gets idx+1$
			
			\EndIf
			
			\EndFor
			
			\EndFor
			
			\State $\texttt{merge}(L',A,\epsilon)$ \label{line:mergeCall}\Comment{Merges equivalent nodes of layer $\ell+1$}
			
			\State $L\gets L'$
			
			\EndFor
			
			\State $N\gets N\cup L$ \Comment{Adds the terminal nodes}
		\end{algorithmic}
	\end{algorithm}

	\paragraph{Exact and $\epsilon$-exact decision diagrams}
	
	If $\epsilon=0$ and \texttt{merge} is implemented via maximal merging or the hash-merging procedure of \S\ref{sec:hash-merging}, then Algorithm~\ref{alg:constructionDD} outputs precisely the exact decision diagram of Definition~\ref{def:exactDD}. As explained in \ref{sec:exact arithmetic}, if the data $\bm{F}\in \mathbb{Q}^{n\times m}$ is rational, then exact decision diagrams can indeed be constructed using rational arithmetic and avoiding division by a potentially irrational number in line ~\ref{line:normalization} of Algorithm~\ref{alg:constructionDD}.
    Exact diagrams are interesting from a theoretical perspective, and the theoretical results in \S\ref{sec:convexification} and \S\ref{sec:polycase} pertain to this case. In contrast, $\epsilon$-exact decision diagrams, defined next, are practical, and in \S\ref{sec:approximation} we study their size and solution quality.
	
	\begin{definition}[$\epsilon$-exact decision diagram]\label{def:epsDD}
		An \underline{$\epsilon$-exact decision diagram} is a diagram obtained from the construction in Algorithm~\ref{alg:constructionDD} with $\epsilon>0$.\hfill$\blacksquare$
	\end{definition}
	
	We point out that $\epsilon$-exact decision diagrams for a given problem are not unique, as they depend on the implementation of the routine $\texttt{merge}$. 
	Additional technical improvements to Algorithm~\ref{alg:constructionDD} are discussed in the online companion (\S\ref{sec:additional-improvements}).

	\section{Convexification}\label{sec:convexification}
	
	From Proposition~\ref{prop:pathLength}, we see that under a technical assumption, we can use decision diagrams to solve optimization problems over set $X_{\bm{F},Z}$. We repeat the definition of this set for convenience:
	$$X_{\bm{F},Z}=\left\{(x_0,\bm{x},\bm{z})\in \R\times \R^n\times Z: \|\bm{F^\top x}\|_2^2\leq x_0,\; \bm{x}\circ (\bm{1}-\bm{z})=\bm{0}\right\}.$$
	In this section, we show that we can also use decision diagrams to construct conic quadratic relaxations of $\cl\;\conv(X_{\bm{F},Z})$, and that the relaxations are ideal under the same technical assumption. 
	
	\begin{definition}[SOCP relaxation]\label{def:socp} Given  $\bm{F}\in \R^{n\times m}$, set $Z\subseteq \{0,1\}^n$ and an exact decision diagram $\mathcal{D}=(N,A,\bm{\nu},\bm{u})$ for $X_{\bm{F},Z}$, let $R_{\bm{F},Z}\subseteq \R^{1+2n+2|A|}$ be the set of points $(x_0,\bm{x},\bm{z},\bm{w},\bm{r})$ satisfying
    \vspace{-1em}
		\begin{subequations}\label{eq:convexhull}
			\begin{align}
				&x_0\geq \sum_{a\in A}\frac{w_a^2}{r_a}\label{eq:convexhull_nonlinear}\\
				&\bm{F^\top x}=\sum_{a\in A}\bm{u_a}w_a\label{eq:convexhull_x}\\
				&\bm{z}=\sum_{a\in A:\nu_a=1}\bm{e_{\ell(a)}}r_a\label{eq:convexhull_z}\\
				&\sum_{a\in A: h_a=v}r_a=\sum_{a\in A: t_a=v}r_a\qquad \forall v\in N:2 \leq \ell(v)\leq n\label{eq:convexhull_f3}\\
				&\sum_{a\in A:\ell(a)=1}r_a = 1; \; \sum_{a\in A:\ell(a)=n}r_a = 1; \; \bm{r}\geq \bm{0}\label{eq:convexhull_nonneg}.
			\end{align}
		\end{subequations}\hfill$\blacksquare$
	\end{definition}
	
	In \eqref{eq:convexhull_nonlinear}, each term $w_a^2/r_a$ is interpreted via closures: $w_a^2/r_a\defeq 0$ if $w_a=r_a=0$, and $w_a^2/r_a\defeq+\infty$ if $r_a=0$ and $w_a\neq 0$. Variables $\bm{r}$ and flow conservation constraints \eqref{eq:convexhull_f3}-\eqref{eq:convexhull_nonneg} indicate a flow in the graph $\mathcal{G}$ defining the decision diagram. Constraints \eqref{eq:convexhull_z} ensure that for any binary vector $\bm{z}\in \{0,1\}^n$, variables $\bm{r}$ are precisely the path associated with this solution. Moreover, variables $\bm{w}$ and constraints \eqref{eq:convexhull_x} are used to encode that $\bm{F^\top x}\in \Row(\bm{F_z})$. Since constraints \eqref{eq:convexhull_x}-\eqref{eq:convexhull_nonneg} are linear and \eqref{eq:convexhull_nonlinear} is SOCP-representable as a sum of quadratic-over-linear functions \citep{lobo1998applications}, set $R_{\bm{F},Z}$ is also SOCP-representable and thus convex.
	
	\begin{proposition}[Validity] \label{prop:valid}
		It holds that $X_{\bm{F},Z}\subseteq \text{proj}_{(x_0,\bm{x},\bm{z})}R_{\bm{F},Z}$.
	\end{proposition}
	
	Proposition~\ref{prop:valid} verifies that set $R_{\bm{F},Z}$ in Definition~\ref{def:socp} is indeed a valid extended relaxation of $X_{\bm{F},Z}$, regardless of whether Assumption~\ref{ass:unbounded} is satisfied or not. The next result indicates that this relaxation is ideal if the assumption is satisfied.
	
	\begin{theorem}\label{theo:hullDD}
		If Assumption~\ref{ass:unbounded} holds, then $\text{cl conv}(X_{\bm{F},Z})= \text{proj}_{(x_0,\bm{x},\bm{z})}R_{\bm{F},Z}$.
	\end{theorem}

	Note that formulation \eqref{eq:convexhull} involves $\mathcal{O}(|A|)$ additional variables and constraints: the convexification is thus compact precisely when the underlying decision diagram is small. Bounding the size of the diagrams for structured classes of matrices is the subject of \S\ref{sec:polycase} and \S\ref{sec:approximation}.
	Finally, we point out that even if the data $\bm{F}$ is rational, vectors $\bm{u_a}$ may be irrational. Nonetheless, as explained in \ref{sec:exact arithmetic}, if the decision diagram was constructed exactly over the rationals, then an extended formulation of $R_{\bm{F},Z}$ with rational coefficients only can also be obtained.

	\section{Decision diagrams of polynomial size}\label{sec:polycase}
	
	The decision diagrams proposed in \S\ref{sec:DDMIQO} can be prohibitively large in general, involving $\mathcal{O}(2^n)$ nodes, even if $Z=\{0,1\}^n$. As a result, the convexifications in \S\ref{sec:convexification} may require an exponential number of variables and constraints. Nonetheless, Examples~\ref{ex:rank-one}, \ref{ex:tridiagonal}, and \ref{ex:inverseTridiagonal} show situations where state merging occurs, and the size of the decision diagram and resulting convexifications can be tractable. In this section, we prove that for certain classes of matrices $\bm{Q}=\bm{FF^\top}$, namely, low-rank, tree, and inverse-tree structures, the decision diagrams and convexifications are indeed of polynomial size. In the process, we recover or generalize several results from the literature. We point out that all results of this section pertain to exact decision diagrams.
	
	\subsection{Low-rank structures}\label{sec:lowRank}
	\citet{Han2024c} study the set
	\[
	X_{\theta}=\left\{(x_0,\bm{x},\bm{z})\in \R\times \R^n\times \{0,1\}^n: \theta\left(\bm{F^\top x}\right)\leq x_0,\; \bm{x}\circ (\bm{1}-\bm{z})=\bm{0}\right\}.
	\]
	where $\bm{F}\in \R^{n\times m}$ and $\theta:\R^m\to\R$ is a convex function. Using disjunctive programming arguments \citep{ceria1999convex}, \citet{Han2024c} propose extended formulations of $X_\theta$ with $\mathcal{O}(n^m)$ additional variables and constraints, which is polynomial if $m$ is treated as a constant. 
	Observe that for $\theta(\bm{y})=\|\bm{y}\|_2^2$, set $X_\theta$ is the special case of $X_{\bm{F},Z}$ where $Z=\{0,1\}^n$. In Example~\ref{ex:rank-one} we discussed an example with $m=1$. 
	We now formally state that exact decision diagrams have at most $\mathcal{O}(n^m)$ nodes, matching the results of \citet{Han2024c} for the quadratic case.
	\begin{proposition}\label{prop:lowRankSize}
		The number of nodes in the exact decision diagram for $X_{\bm{F}}$ is $\mathcal{O}(n^m)$.
	\end{proposition}

	On the one hand, the results of \citet{Han2024c} are more general in the sense that the description of $\text{cl conv}(X_\theta)$ allows for more general nonlinear functions; moreover, \citet{Han2024c} also show that their results hold in the presence of constraints $\bm{x}\geq \bm{0}$. On the other hand, the results of this paper are more general since the set $X_{\bm{F},Z}$ allows for the inclusion of additional constraints. %
	Finally, we close this subsection by noting that, for the case with $m=1$, \citet{atamturk2025rank} propose an extended SOCP formulation for $\text{cl}\; \text{conv}(X_{\bm{F}})$ with one additional variable, instead of $\mathcal{O}(n)$ as suggested by Proposition~\ref{prop:lowRankSize}. This formulation could be derived relatively easily from \eqref{eq:convexhull} by projecting out additional variables, but we refrain from doing so in the interest of space. These results were later generalized to sets with additional constraints on $\bm{z}$ and $\bm{x}$ by \citet{wei2020convexification} and \citet{shafiee2024constrained}, where the size of the formulations depends on the specific constraints.

	\subsection{Tree structures}\label{sec:tree}
	In this section, we assume that matrix $\bm{Q}= \bm{FF^\top}$ is invertible and $Z=\{0,1\}^n$. For this subsection and next, recall the definitions of support and neighbor sets from \S\ref{sec:notation}.  We show that if $\supp(\Q)$ forms a rooted tree with $k$ leaves, then the resulting decision diagram contains at most $\mathcal{O}(n^{k+1})$ nodes. As a consequence, the extended formulation in Definition~\ref{def:socp} is an extended formulation for $\text{cl conv}(X_{\bm{F}})$ of polynomial size if the number of leaves is treated as a constant. Note that in Example~\ref{ex:tridiagonal}, the matrix $\bm{Q}$ is tridiagonal, implying that $\supp(\bm{Q})$ forms a path, i.e., a tree with one leaf (if the root is taken to be either node $1$ or node $n$), corresponds to the special case $k = 1$. In fact, extended formulations with $\mathcal{O}(n^2)$ additional variables for $\Q$ tridiagonal were proposed by \citet{liu2023graph} (although the formulations were SDPs rather than SOCPs). In this subsection we recover and generalize these results to arbitrary tree graphs.

	Assume that $\supp(\Q)$ is a rooted tree, which we may take to be connected without loss of generality (as otherwise set $X_{\bm{F}}$ can be decomposed). We assign labels to its nodes in a bottom-up manner, so that every node receives a label smaller than the label of its parent and, in particular, the root is assigned label $n$, see Figure~\ref{fig:oneCut}(a) for an example. A useful consequence of this labeling is that, for every $\ell\in [n]$, the set $[\ell:n]$ induces a connected subgraph of $\supp(\Q)$ containing the root. %
	The nodes in the decision diagram are then processed according to this labeling, starting from the leaves. To complete the construction, we introduce a source node $s_0$ connected to the root and a sink node $s_1$ connected to all leaves. The resulting augmented graph is denoted by $\mathcal{G}_{\Q}=(V_{\bm{Q}},E_{\bm{Q}})$.
	
	\begin{definition}[Minimal cut]
		A minimal cut of the augmented graph $\mathcal{G}_{\Q}$ is a subset $S\subseteq V_{\Q}$ such that $s_0\in S$, $s_1\not \in S$ and both $S$ and $V_{\bm{Q}}\setminus S$ are connected on $\mathcal{G}_{\Q}$. The cutset associated with cut $S$ is $E_S=\{(i,j)\in E_{\Q}:i\in S\text{ and }j\not\in S\}$.\hfill$\blacksquare$
	\end{definition}
	
	Given the structure of the graph $\mathcal{G}_{\Q}$, the number of distinct minimal cuts is polynomial in the number of nodes but exponential in the number of leaves. 
	
	\begin{lemma}\label{lem::number-mincut}
		If $\supp(\Q)$ is a connected rooted tree with $n$ nodes and
		$k$ leaves, then the number of distinct minimal cuts in
		$\mathcal{G}_{\Q}$ is at most $
		\left(\frac{n-1}{k}+2\right)^k.$
	\end{lemma}

	We now proceed to associate with each partial solution $\bm{v}\in \{0,1\}^{\ell-1}$ a minimal cut in the graph.
	
	\begin{definition}[One-cut]\label{def:one-cut}
		Suppose that $\supp(\Q)$ is a rooted tree with $k$ leaves and consider the augmented graph $\mathcal{G}_{\Q}$.
		For any $\ell\in [2:n]$ and $\bm{v} \in \{0,1\}^{\ell-1}$, let $\mathcal{G}_{\bm{Q}}(\bm{v};1)$ be the subgraph of $\mathcal{G}_{\Q}$ induced by nodes $V_{\bm{Q}}(\bm{v};1)\defeq\{i\in [\ell-1]:v_i=1\}\cup[\ell:n]$, let $\tilde S_1(\bm{v})$ be the connected component of $\mathcal{G}_{\bm{Q}}(\bm{v};1)$ containing node $n$, and define the \underline{one-cut} as $S_1^\ell(\bm{v})=\tilde S_1(\bm{v})\cap [\ell-1]$.\hfill$\blacksquare$
	\end{definition}
	
	\begin{wrapfigure}{r}{0.54\textwidth}
		\vspace{-1.0\baselineskip}
		\centering
		\subfigure[]{
			\begin{tikzpicture}[scale=0.72,
				treenode/.style={circle,fill=gray!50,font=\scriptsize,inner sep=2pt},
				auxnode/.style={circle,draw,font=\scriptsize,inner sep=1.5pt}]
				\node[auxnode] (s0) at (2.7,5.4) {$s_0$};
				\node[treenode] (n7) at (2.7,4.3) {$7$};
				\node[treenode] (n5) at (1.3,3.3) {$5$};
				\node[treenode] (n6) at (3.3,3.3) {$6$};
				\node[treenode] (n1) at (0.7,2.1) {$1$};
				\node[treenode] (n2) at (1.9,2.1) {$2$};
				\node[treenode] (n4) at (3.3,2.1) {$4$};
				\node[treenode] (n3) at (3.3,0.9) {$3$};
				\node[auxnode] (s1) at (1.9,-0.4) {$s_1$};
				\draw (s0)--(n7); \draw (n7)--(n5); \draw (n7)--(n6);
				\draw (n5)--(n1); \draw (n5)--(n2); \draw (n6)--(n4); \draw (n4)--(n3);
				\draw (n1)--(s1); \draw (n2)--(s1); \draw (n3)--(s1);
			\end{tikzpicture}
		}
		\hspace{0.8em}
		\subfigure[]{
			\begin{tikzpicture}[scale=0.72,
				freenode/.style={circle,fill=gray!50,font=\scriptsize,inner sep=2pt},
				onenode/.style={circle,fill=black,text=white,font=\scriptsize,inner sep=2pt},
				zeronode/.style={circle,draw,fill=white,font=\scriptsize,inner sep=2pt},
				auxnode/.style={circle,draw,font=\scriptsize,inner sep=1.5pt}]
				\node[auxnode] (s0) at (2.7,5.4) {$s_0$};
				\node[freenode] (n7) at (2.7,4.3) {$7$};
				\node[zeronode] (n5) at (1.3,3.3) {$5$};
				\node[freenode] (n6) at (3.3,3.3) {$6$};
				\node[onenode] (n1) at (0.7,2.1) {$1$};
				\node[zeronode] (n2) at (1.9,2.1) {$2$};
				\node[onenode] (n4) at (3.3,2.1) {$4$};
				\node[onenode] (n3) at (3.3,0.9) {$3$};
				\node[auxnode] (s1) at (1.9,-0.4) {$s_1$};
				\draw[gray!60] (s0)--(n7); \draw[gray!60] (n7)--(n5);
				\draw[gray!60] (n5)--(n1); \draw[gray!60] (n5)--(n2);
				\draw[gray!60] (n1)--(s1); \draw[gray!60] (n2)--(s1); \draw[gray!60] (n3)--(s1);
				\draw[line width=1.1pt] (n7)--(n6); \draw[line width=1.1pt] (n6)--(n4); \draw[line width=1.1pt] (n4)--(n3);
				\draw[thick] ($(n7)!0.5!(n5)+(0.08,-0.11)$) -- ($(n7)!0.5!(n5)+(-0.08,0.11)$);
				\draw[thick] ($(n3)!0.5!(s1)+(0.10,-0.10)$) -- ($(n3)!0.5!(s1)+(-0.10,0.10)$);
				\node[draw,dashed,rounded corners=7pt,inner sep=8pt,fit=(n7)(n6)(n4)(n3),label={[font=\scriptsize]above right:$\tilde S_1(\bm{v})$}] {};
				\node[draw,dotted,rounded corners=5pt,inner sep=3pt,fit=(n4)(n3),label={[font=\scriptsize]right:$S_1^{6}(\bm{v})$}] {};
			\end{tikzpicture}
		}
		\caption{Augmented tree (a) and one-cut $S_1^6(\bm{v})=\{3,4\}$ (b) for the instance in Definition~\ref{def:one-cut}.}
		\label{fig:oneCut}
	\vspace{-0.5\baselineskip}
	\end{wrapfigure}
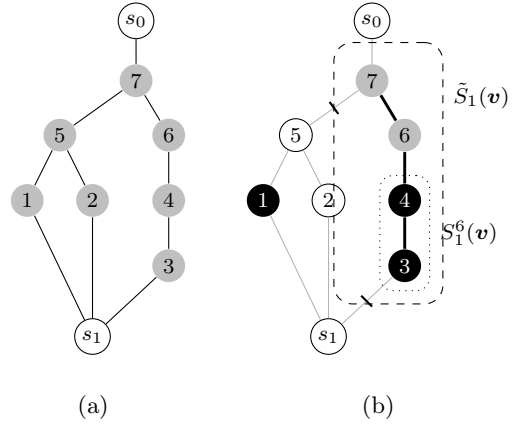
	
	Figure~\ref{fig:oneCut} depicts the augmented graph and illustrates the construction of a one-cut. In panel (b), black, white, and gray nodes are fixed to one, fixed to zero, and undecided, respectively. The bold edges induce the component $\tilde S_1(\bm{v})=\{3,4,6,7\}$ (dashed region), while the dotted region is the one-cut $S_1^6(\bm{v})=\{3,4\}$; the strokes mark the associated cutset. The main properties of one-cuts are: $\bullet$ if $i\in S_1^\ell(\bm{v})$ then $v_i=1$; and $\bullet$ if $(i,j)\in \supp(\Q)$ with $i\in[\ell-1]\setminus S_1^\ell(\bm{v})$ and $j\in S_1^\ell(\bm{v})$, then $v_i=0$.
	Moreover, it is easy to check that the set $\tilde S_1(\bm{v})\cup\{s_0\}$ in Definition~\ref{def:one-cut} is a minimal cut, thus from Lemma~\ref{lem::number-mincut} there are at most $(n+1)^k$ such sets. As a direct consequence, there are also at most $(n+1)^k$ distinct sets $S_1^\ell(\bm{v})$.

	We now establish the following key result, which plays a central role in characterizing the size of the decision diagram when $\supp(\Q)$ is a connected tree. Specifically, we show that the state of the decision diagram associated with decision $\bm{v}$ depends only on $S_1^\ell(\bm{v})$.
	
	\begin{theorem}\label{thm::identical-states-zero}
		Suppose that $\supp(\Q)$ is a connected tree. For any $\ell\in [2:n]$ and $\bm{v} \in \{0,1\}^{\ell-1}$ let $S = S_1^\ell(\bm{v})$ be the associated one-cut. Then, in the decision diagram, the state matrix of the node corresponding to assignment $\bm{v}$ is
		\[
		\bm{H^\ell}(\bm{v}) 
		= \bm{F_{[\ell:n]}}-\bm{Q_{[\ell:n], S}}\left(\bm{Q_{S, S}}\right)^{-1}\bm{F_{S}}.
		\]
	\end{theorem}
	
	An important implication of the above theorem is that if two sparsity patterns $\bm{v_1}, \bm{v_2} \in \{0,1\}^{\ell-1}$ share the same one-cut, then their corresponding states in the decision diagram are identical, i.e., $\bm{H^{\ell}}(\bm{v_1}) = \bm{H^{\ell}}(\bm{v_2})$. This observation, combined with the bound on the number of minimal cuts in Lemma~\ref{lem::number-mincut}, leads to the following important corollary.
	
	\begin{corollary}\label{cor::polsize-zero}
		Suppose that $\supp(\Q)$ is a connected rooted tree with $k$ leaves. In this case, the number of nodes in the decision diagram constructed according to the variable order induced by the labeling of $\supp(\Q)$ is at most $n\left(\frac{n-1}{k}+2\right)^k=\mathcal{O}\left(n^{k+1}\right)$.
	\end{corollary}
	
	\begin{remark}
		Note that, unlike the argument given in Example~\ref{ex:tridiagonal}, the result in Theorem~\ref{thm::identical-states-zero} does not rely on the sparsity pattern (or other properties) of $\bm{F}$. In particular, decision diagrams applied to tree-structured matrices $\bm{Q}$ produce graphs and extended formulations of size $\mathcal{O}(n^{k+1})$ independent of the decomposition $\bm{Q}=\bm{FF^\top}$ used. \hfill $\blacksquare$
	\end{remark}
	The above result, together with our construction of $\cl\conv(X_{\bm{F}})$ from the decision diagram, implies the existence of an SOCP representation of $\cl\conv(X_{\bm{F}})$ of size $\mathcal{O}\left(n^{k+1}\right)$.
	Our result strictly generalizes that of \citet{liu2023graph}, who provide an SDP characterization of $\cl \conv(X)$ for the special case $k = 1$; in this case, our bound yields an SOCP representation of size $O(n^2)$, but also applies to $k\geq 1$.
	However, when the number of leaves scales linearly with $n$ (e.g., star graphs with $k = n-1$), the resulting representation becomes exponential in size. We point that polynomial-time algorithms for tree structured problems exist regardless of the number of leaves \citep{bhathena2025parametric}. %
	
	\subsection{Inverse tree structures}\label{sec:invTree}
	We now turn to the case where the support graph of $\bm{Q^{-1}}$ is a tree. As in \S\ref{sec:tree}, we assume throughout this section that $\bm{Q}=\bm{FF^\top}$ is invertible and $Z=\{0,1\}^n$. To streamline the presentation, we define $\bm{A}\defeq\bm{Q^{-1}}$ and assume that $\supp(\bm{A})$ is a connected tree with $k$ leaves. 
	
	We employ the same graph construction of \S\ref{sec:tree}, now applied to $\supp(\bm{A})$ instead of $\supp(\Q)$: we root the tree and label its nodes so that every node receives a label smaller than the label of its parent (thus the root is assigned label $n$). We process the variables in the order given by the labels, and let $\mathcal{G}_{\bm{A}}=(V_{\bm{A}},E_{\bm{A}})$ be the augmented graph obtained by adding a source node $s_0$ adjacent to the root and a sink node $s_1$ adjacent to all leaves. As discussed in \S\ref{sec:tree}, this labeling guarantees that, for every $\ell\in[n]$, the set $[\ell:n]$ induces a connected subgraph of $\supp(\bm{A})$ containing the root.
	
	The key difference with \S\ref{sec:tree} is the relevant notion of cut. For tree structures, the state associated with a partial solution $\bm{v}$ is determined by the indices fixed to \emph{one} that remain connected to the undecided portion of the graph. For inverse tree structures, the roles of zeros and ones are interchanged: the state is determined by the connected component of undecided indices and indices fixed to \emph{zero} that contains the undecided portion of the graph, together with its boundary of indices fixed to one. This mirror-image behavior is already apparent in Example~\ref{ex:inverseTridiagonal}, where state resets are associated with decisions $z_\ell=1$ instead of $z_\ell=0$.
	
	\begin{definition}[Zero-cut]\label{def:zero-cut}
		Suppose that $\supp(\bm{A})$ is a rooted tree. For any $\ell\in [2:n]$ and $\bm{v} \in \{0,1\}^{\ell-1}$, let $\mathcal{G}_{\bm{A}}(\bm{v};0)$ be the subgraph of $\supp(\bm{A})$ induced by nodes $V_{\bm{A}}(\bm{v};0)\defeq\{i\in [\ell-1]:v_i=0\}\cup[\ell:n]$. Define the \underline{zero-cut} $S_0^\ell(\bm{v})$ as the connected component of $\mathcal{G}_{\bm{A}}(\bm{v};0)$ containing node $n$.\hfill$\blacksquare$
	\end{definition}
	
	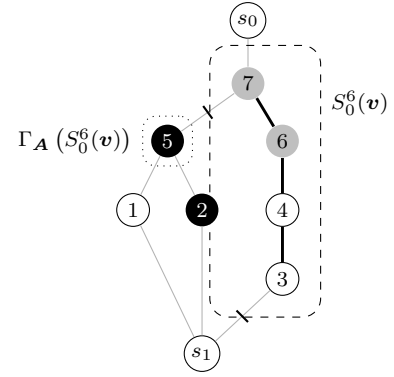
\begin{wrapfigure}{R}{0.36\textwidth}
		\vspace{-1.0\baselineskip}
		\centering
		\begin{tikzpicture}[scale=0.76,
			freenode/.style={circle,fill=gray!50,font=\scriptsize,inner sep=2pt},
			onenode/.style={circle,fill=black,text=white,font=\scriptsize,inner sep=2pt},
			zeronode/.style={circle,draw,fill=white,font=\scriptsize,inner sep=2pt},
			auxnode/.style={circle,draw,font=\scriptsize,inner sep=1.5pt}]
			\node[auxnode] (s0) at (2.7,5.4) {$s_0$};
			\node[freenode] (n7) at (2.7,4.3) {$7$};
			\node[onenode] (n5) at (1.3,3.3) {$5$};
			\node[freenode] (n6) at (3.3,3.3) {$6$};
			\node[zeronode] (n1) at (0.7,2.1) {$1$};
			\node[onenode] (n2) at (1.9,2.1) {$2$};
			\node[zeronode] (n4) at (3.3,2.1) {$4$};
			\node[zeronode] (n3) at (3.3,0.9) {$3$};
			\node[auxnode] (s1) at (1.9,-0.4) {$s_1$};
			\draw[gray!60] (s0)--(n7); \draw[gray!60] (n7)--(n5);
			\draw[gray!60] (n5)--(n1); \draw[gray!60] (n5)--(n2);
			\draw[gray!60] (n1)--(s1); \draw[gray!60] (n2)--(s1); \draw[gray!60] (n3)--(s1);
			\draw[line width=1.1pt] (n7)--(n6); \draw[line width=1.1pt] (n6)--(n4); \draw[line width=1.1pt] (n4)--(n3);
			\draw[thick] ($(n7)!0.5!(n5)+(0.08,-0.11)$) -- ($(n7)!0.5!(n5)+(-0.08,0.11)$);
			\draw[thick] ($(n3)!0.5!(s1)+(0.10,-0.10)$) -- ($(n3)!0.5!(s1)+(-0.10,0.10)$);
			\node[draw,dashed,rounded corners=7pt,inner sep=8pt,fit=(n7)(n6)(n4)(n3),label={[font=\scriptsize]above right:$S_0^{6}(\bm{v})$}] {};
			\node[draw,dotted,rounded corners=5pt,inner sep=3pt,fit=(n5),label={[font=\scriptsize]left:$\Gamma_{\bm{A}}\left(S_0^{6}(\bm{v})\right)$}] {};
		\end{tikzpicture}
		\caption{Zero-cut for the instance in Definition~\ref{def:zero-cut}.}
		\label{fig:zeroCut}
	\vspace{-0.5\baselineskip}
	\end{wrapfigure}
	
	Figure~\ref{fig:zeroCut} illustrates the construction of a zero-cut. Black, white, and gray nodes are fixed to one, fixed to zero, and undecided, respectively. The bold edges form the component $S_0^6(\bm{v})=\{3,4,6,7\}$ (dashed region), and the dotted halo marks its neighbor set; thus the roles of zeros and ones are interchanged relative to Figure~\ref{fig:oneCut}(b). In contrast with Definition~\ref{def:one-cut}, it is convenient here to include the undecided indices $[\ell:n]$ in the zero-cut. The main properties of zero-cuts are: $\bullet$ $[\ell:n]\subseteq S_0^\ell(\bm{v})$, since $[\ell:n]$ induces a connected subgraph of $\mathcal{G}_{\bm{A}}(\bm{v};0)$ containing node $n$; $\bullet$ if $i\in S_0^\ell(\bm{v})\cap[\ell-1]$, then $v_i=0$; $\bullet$ all neighbors of the zero-cut correspond to indices already fixed to one, that is, $\Gamma_{\bm{A}}\big(S_0^\ell(\bm{v})\big)\subseteq\{i\in[\ell-1]:v_i=1\}$; $\bullet$ the set $S_0^\ell(\bm{v})\cup\{s_0\}$ is a minimal cut of the augmented graph $\mathcal{G}_{\bm{A}}$. As a consequence of this last point, similar to the arguments in \S\ref{sec:tree}, Lemma~\ref{lem::number-mincut} implies that at any given layer $\ell$ there are at most $(n+1)^k$ distinct zero-cuts.

	The key to obtain similar results as the tree case is a matrix identity allowing translations between inverses of submatrices of $\bm{Q}$ and the corresponding inverses of submatrices of $\bm{A}$. 
	\begin{lemma}\label{lem::inverse}
		Let $S$, $T$ and $V$ be a partition of $[n]$, with $S,T\neq \emptyset$ and $V$ possibly empty, such that no edge of $\supp(\bm{A})$ connects $S$ and $V$, that is, $\bm{A_{S,V}}=\bm{0}$. Then
		\begin{align}
			\bm{Q_{S,T}} \left(\bm{Q_{T,T}}\right)^{-1} \;=\; - \left(\bm{A_{S,S}}\right)^{-1} \bm{A_{S,T}}.
		\end{align}
	\end{lemma}
	
	We now state the analog of Theorem~\ref{thm::identical-states-zero} for the case where $\supp(\bm{A})$ is a tree: the state of the decision diagram associated with a partial solution $\bm{v}$ depends only on the zero-cut $S_0^\ell(\bm{v})$.
	\begin{theorem}\label{thm::identical-states-one}
		Suppose that $\supp(\bm{A})$ is a connected tree. For any $\ell\in [2:n]$ and $\bm{v} \in \{0,1\}^{\ell-1}$, let $S = S_0^\ell(\bm{v})$ be the associated zero-cut and let $\Gamma=\Gamma_{\bm{A}}(S)$ be the set of neighbors of $S$ in $\supp(\bm{A})$. Then, in the decision diagram, the state matrix of the node corresponding to assignment $\bm{v}$ is
		\[
		\bm{H^\ell}(\bm{v})
		= \bm{F_{[\ell:n]}} + \left(\bm{A_{S,S}^{-1}}\right)_{[\ell:n]}\bm{A_{S,\Gamma}} \bm{F_{\Gamma}},
		\]
		where $\left(\bm{A_{S,S}^{-1}}\right)_{[\ell:n]}$ denotes the submatrix of $\bm{A_{S,S}^{-1}}$ corresponding to rows with indices in $[\ell:n]\subseteq S$, and the second term is understood to be $\bm{0}$ if $\Gamma=\emptyset$.
	\end{theorem}
	An immediate consequence of Theorem~\ref{thm::identical-states-one} is that the state associated with a partial solution depends on it only through its zero-cut: if $\bm{v_1}, \bm{v_2} \in \{0,1\}^{\ell-1}$ satisfy $S_0^\ell(\bm{v_1})=S_0^\ell(\bm{v_2})$, then $\bm{H^{\ell}}(\bm{v_1}) = \bm{H^{\ell}}(\bm{v_2})$ and the corresponding states are merged in the decision diagram. This observation yields the following counterpart of Corollary~\ref{cor::polsize-zero}.
	
	\begin{corollary}\label{cor::polsize-one}
		Suppose that $\supp(\bm{A})$ is a connected rooted tree with $k$ leaves. Then the number of nodes in the decision diagram constructed according to the node ordering induced by the labeling of $\supp(\bm{A})$ is at most $n\left(\frac{n-1}{k}+2\right)^k=\mathcal{O}\left(n^{k+1}\right)$.
	\end{corollary}
	
	Corollary~\ref{cor::polsize-one}, together with Theorem~\ref{theo:hullDD}, implies the existence of an ideal SOCP extended formulation of $\cl\conv(X_{\bm{F}})$ of size $\mathcal{O}\left(n^{k+1}\right)$ whenever $\supp(\bm{Q^{-1}})$ is a tree with $k$ leaves. The special case $k=1$ corresponds to matrices $\bm{Q}$ whose inverse is tridiagonal, illustrated in Example~\ref{ex:inverseTridiagonal}: in this case, the decision diagram has $\mathcal{O}(n^2)$ nodes, recovering the setting studied by \citet{lee2024convexification} and generalizing it to arbitrary tree supports of $\bm{Q^{-1}}$. %

	\section{Approximation}\label{sec:approximation}
	In the previous section, we showed that the constructed decision diagrams have polynomial size in three settings. In this section, we consider more general sparsity patterns of $\Q$ and $\Q^{-1}$, for which exact decision diagrams may be exponentially large, and construct decision diagrams of reduced size by merging nodes whose state matrices are close but not necessarily identical, as discussed in \S\ref{sec:pseudocode}. Recall from \S\ref{sec:pseudocode} that $\epsilon$-exact diagrams are the diagrams produced by Algorithm~\ref{alg:constructionDD} with tolerance $\epsilon>0$, and that their size depends on the implementation of the routine \texttt{merge}; the goal of this section is to exhibit implementations for which the resulting diagrams are provably small while still guaranteeing high-quality solutions. Throughout the section, we assume that $Z=\{0,1\}^n$ and that $\bm{Q}=\bm{FF^\top}\succ \bm{0}$. Moreover, to reduce clutter, throughout this section we omit the comma in double subscripts, writing for example $\Q_{VW}$ instead of $\Q_{V,W}$.
	
	Recall from \S\ref{sec:notation} the definitions of $\eta$-neighborhoods, their size maximum $\Delta_{G,\eta}$, distances $d_G(i,j) $as well as the definition of the decay rate $\sigma$ which is closely associated with the condition number. The \emph{volume growth function} characterizes the rate at which $\Delta_{G,\eta}$ increases with $\eta$.
	
	\begin{definition}[Polynomial volume growth]\label{def::vol-Q}
		A graph $G$ is said to have a \underline{polynomial volume growth} with parameters $\gamma, \delta > 0$ if $\Delta_{G,\eta} \le \delta\, \eta^\gamma$ for all $\eta \ge 1$.\hfill$\blacksquare$
	\end{definition}
	
	Many well-studied graph families satisfy this property. For instance, banded graphs with bandwidth $k$ exhibit \emph{linear} volume growth, i.e., $\Delta_{G,\eta} \le 2k \eta+1$. More generally, any $d$-dimensional grid graph (and its subgraphs) satisfies $\Delta_{G,\eta} = \mathcal{O}(\eta^d)$.
	In contrast, trees with $\mathcal{O}(n)$ leaves may violate this assumption. Notable examples include complete binary trees, for which $\Delta_{G,\eta} = \mathcal{O}(2^\eta)$, and star graphs, where $\Delta_{G,1} = n$.	
	Our next definition quantifies how similar two partial solutions are, as seen from a given node of the graph.
	
	\begin{definition}[$\eta$-similarity]\label{def:similarity}
		We say that $\bm{v_1}, \bm{v_2} \in \{0,1\}^{\ell-1}$ are \underline{$\eta$-similar} in $G$ with respect to a node $i \in [n]$ if $(\bm{v_1})_j = (\bm{v_2})_j$ for every $j \in [\ell-1]$ satisfying $d_G(i,j) \le \eta$.\hfill$\blacksquare$
	\end{definition}
	
	The results of this section state that partial solutions that are $\eta$-similar with respect to all undecided nodes lead to nearly identical states. In this section we will consider decision diagrams that merge nodes that are $\eta$-similar (a process we refer to as neighborhood merging) and establish that if $\eta$ is properly chosen, the decision diagrams are $\epsilon$-exact. Moreover, the number of distinct states of such $\epsilon$-exact decision diagram is thus governed by the number of nodes near the boundary between decided and undecided variables, which motivates the following definition.
	
	\begin{definition}[$\eta$-boundary]\label{def:boundary}
		Given a graph $G$ with vertex set $[n]$, a layer $\ell\in [2:n]$ and an integer $\eta\geq 1$, define the \underline{$\eta$-boundary at layer $\ell$} as
		$$\Lambda_{G,\ell}^{\eta} \defeq \left\{j\in [\ell-1]: d_G(i,j)\leq \eta \text{ for some } i \in [\ell:n]\right\},$$
		and let $\omega_{G}(\eta)\defeq \max_{\ell\in [2:n]}\left|\Lambda_{G,\ell}^{\eta}\right|$ denote the maximum size of an $\eta$-boundary.\hfill$\blacksquare$
	\end{definition}
	
	If $G$ is banded with bandwidth $k$ (with respect to the natural order), then $\Lambda_{G,\ell}^{\eta}\subseteq \{\ell-k\eta,\dots,\ell-1\}$ and thus $\omega_G(\eta)\leq k\eta$. %
	Figure~\ref{fig:boundary} illustrates the $\eta$-boundary of a banded support graph, along with the induced merging of nodes of the decision diagram.
	
	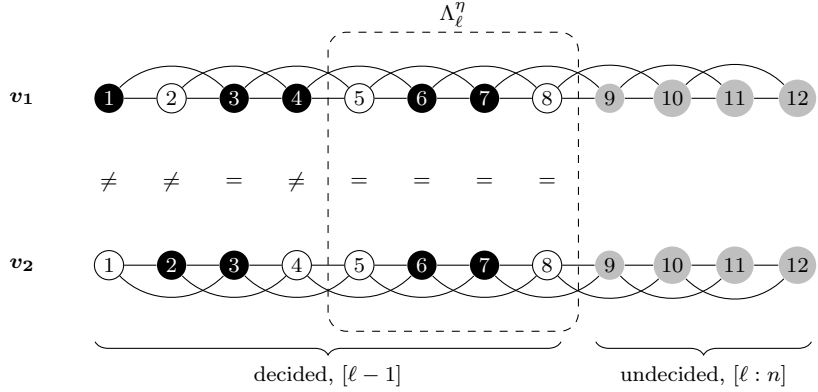
\begin{figure}[!ht]
		\centering
		\begin{tikzpicture}[scale=0.92,
			onenode/.style={circle,fill=black,text=white,font=\scriptsize,inner sep=1.5pt,minimum size=11pt},
			zeronode/.style={circle,draw,fill=white,font=\scriptsize,inner sep=1.5pt,minimum size=11pt},
			freenode/.style={circle,fill=gray!50,font=\scriptsize,inner sep=1.5pt,minimum size=11pt}]
			\def\ya{0}
			\node[font=\scriptsize] at (-0.35,\ya) {$\bm{v_1}$};
			\node[onenode]  (a1)  at (0.9,\ya) {$1$};
			\node[zeronode] (a2)  at (1.8,\ya) {$2$};
			\node[onenode]  (a3)  at (2.7,\ya) {$3$};
			\node[onenode]  (a4)  at (3.6,\ya) {$4$};
			\node[zeronode] (a5)  at (4.5,\ya) {$5$};
			\node[onenode]  (a6)  at (5.4,\ya) {$6$};
			\node[onenode]  (a7)  at (6.3,\ya) {$7$};
			\node[zeronode] (a8)  at (7.2,\ya) {$8$};
			\node[freenode] (a9)  at (8.1,\ya) {$9$};
			\node[freenode] (a10) at (9.0,\ya) {$10$};
			\node[freenode] (a11) at (9.9,\ya) {$11$};
			\node[freenode] (a12) at (10.8,\ya) {$12$};
			\foreach \i/\j in {1/2,2/3,3/4,4/5,5/6,6/7,7/8,8/9,9/10,10/11,11/12}{\draw (a\i)--(a\j);}
			\foreach \i/\j in {1/3,2/4,3/5,4/6,5/7,6/8,7/9,8/10,9/11,10/12}{\draw (a\i) to[bend left=45] (a\j);}
			\def\yb{-2.4}
			\node[font=\scriptsize] at (-0.35,\yb) {$\bm{v_2}$};
			\node[zeronode] (b1)  at (0.9,\yb) {$1$};
			\node[onenode]  (b2)  at (1.8,\yb) {$2$};
			\node[onenode]  (b3)  at (2.7,\yb) {$3$};
			\node[zeronode] (b4)  at (3.6,\yb) {$4$};
			\node[zeronode] (b5)  at (4.5,\yb) {$5$};
			\node[onenode]  (b6)  at (5.4,\yb) {$6$};
			\node[onenode]  (b7)  at (6.3,\yb) {$7$};
			\node[zeronode] (b8)  at (7.2,\yb) {$8$};
			\node[freenode] (b9)  at (8.1,\yb) {$9$};
			\node[freenode] (b10) at (9.0,\yb) {$10$};
			\node[freenode] (b11) at (9.9,\yb) {$11$};
			\node[freenode] (b12) at (10.8,\yb) {$12$};
			\foreach \i/\j in {1/2,2/3,3/4,4/5,5/6,6/7,7/8,8/9,9/10,10/11,11/12}{\draw (b\i)--(b\j);}
			\foreach \i/\j in {1/3,2/4,3/5,4/6,5/7,6/8,7/9,8/10,9/11,10/12}{\draw (b\i) to[bend right=45] (b\j);}
			\foreach \x/\m in {0.9/\neq,1.8/\neq,2.7/=,3.6/\neq,4.5/=,5.4/=,6.3/=,7.2/=}{
				\node[font=\scriptsize] at (\x,-1.2) {$\m$};}
			\draw[dashed,rounded corners=7pt] (4.05,-3.35) rectangle (7.65,0.95);
			\node[font=\scriptsize] at (5.85,1.2) {$\Lambda_{\ell}^{\eta}$};
			\draw[decorate,decoration={brace,mirror,amplitude=5pt}] (0.7,-3.5) -- (7.4,-3.5)
			node[midway,below=5pt,font=\scriptsize] {decided, $[\ell-1]$};
			\draw[decorate,decoration={brace,mirror,amplitude=5pt}] (7.9,-3.5) -- (11.0,-3.5)
			node[midway,below=5pt,font=\scriptsize] {undecided, $[\ell:n]$};
		\end{tikzpicture}
		\caption{Illustration of Definition~\ref{def:boundary} and Theorem~\ref{thm:approxSummary} for a support graph $\supp(\Q)$ with bandwidth $k=2$, with $\ell=9$ and $\eta=2$. Nodes fixed to one are black, nodes fixed to zero are white, and undecided nodes are gray. Partial solutions $\bm{v_1}$ and $\bm{v_2}$ coincide on the $\eta$-boundary $\Lambda_{\ell}^{\eta}=\{5,6,7,8\}$ but differ elsewhere.%
		}
		\label{fig:boundary}
	\end{figure}
	
	Equipped with these definitions, the main results of this section can be summarized as follows.
	
	\begin{theorem}[Summary of the main results of the section]\label{thm:approxSummary}
		Suppose that $G=\supp(\Q)$ or $G=\supp(\Q^{-1})$ has polynomial volume growth with parameters $\gamma,\delta>0$, and that the associated $\eta$-boundaries satisfy $\omega_{G}(\eta)\leq \delta'\eta^{\gamma'}$ for all $\eta\geq 1$, for some parameters $\gamma',\delta'>0$. Then, for every $\epsilon>0$, there exists an $\epsilon$-exact decision diagram for $X_{\bm{F}}$ with the number of nodes bounded by
		\[
		(n+1)\cdot 2^{\delta'\eta_\epsilon^{\gamma'}},
		\] where $\eta_\epsilon=\mathcal{O}\left(\log(1/\epsilon)\right)$ with constants depending only on $\gamma$, $\delta$ and the extreme eigenvalues of $\Q$.
	\end{theorem}
	
	In other words, for any fixed $\epsilon>0$, the size of the $\epsilon$-exact decision diagrams is polynomial---in fact, linear---in the dimension $n$, but grows exponentially with $\eta_\epsilon^{\gamma'}$, and thus doubly exponentially with the boundary exponent $\gamma'$. In the special case $\gamma'=1$, the size is polynomial in $1/\epsilon$ as well: in particular, if $\Q$ or $\Q^{-1}$ is banded with bandwidth $k$, then $\gamma=\gamma'=1$ and $\delta'\leq k$, and the resulting diagrams have at most $(n+1)\cdot(1/\epsilon)^{\mathcal{O}\left(k/\log(1/\sigma)\right)}$ nodes. A refined version of Theorem~\ref{thm:approxSummary} is stated in \S\ref{sec:sparse} for the sparse case. Guarantees for the inverse-sparse case can be obtained by relating submatrices of $\bm{Q}$ and its inverse (similarly to Lemma~\ref{lem::inverse}) and using similar steps, and we defer its discussion to \ref{sec:invSparse}.%

    \subsection{Sparse structures}\label{sec:sparse}
 We first state a lemma to control the influence of decisions taken far from the corresponding node, which follows directly from \citet[Lemma 6]{bhathena2026solving}.
	
	\begin{lemma}\label{lem::decay}
		Fix any $i \in [n]$ and $\bm{v} \in \{0,1\}^{\ell-1}$. Let $V = \{j\in [\ell-1]: v_j=1,\; d(i,j)\leq \eta\}$ and $W = \{j\in [\ell-1]: v_j=1,\; d(i,j)> \eta\}$, where distances are computed in $\supp(\Q)$. Then, we have
		\begin{align*}
			\norm{\Q_{\{i\},V}(\Q_{VV})^{-1}\Q_{VW}}_2\leq \frac{2\mu^2_{\max}(\Q)}{\mu_{\min}(\Q)}\sqrt{\Delta_{i,1}\Delta_{i,\eta}}\, \sigma^\eta.
		\end{align*}
	\end{lemma}

	The next theorem is the key result of this section: it shows that the rows of the state matrices depend only weakly on the decisions taken far away (in $\supp(\Q)$) from the corresponding node.
	
	\begin{theorem}\label{thm::approx-Q}
		Fix any $\ell\in [2:n]$ and $i \in [\ell:n]$, and consider $\bm{v_1}, \bm{v_2} \in \{0,1\}^{\ell-1}$ that are $\eta$-similar in $\supp(\Q)$ with respect to $i$. Then,
		\begin{align*}
			\norm{\bm{H^\ell}(\bm{v_1})_{\{i\}}-\bm{H^\ell}(\bm{v_2})_{\{i\}}}_2\leq \left(\frac{10\,\mu^{3.5}_{\max}(\Q)}{\mu^{3}_{\min}(\Q)}\right)\sqrt{\Delta_{i,1}\Delta_{i,\eta}}\,\sigma^\eta,
		\end{align*}
		where $\bm{H^\ell}(\bm{v})_{\{i\}}$ is the row of the state matrix $\bm{H^\ell}(\bm{v})$ corresponding to index $i$.
	\end{theorem}

	Theorem~\ref{thm::approx-Q} has a key implication for the merging procedure. Fix $\ell\in[2:n]$, and suppose that $\bm{v_1},\bm{v_2}\in \{0,1\}^{\ell-1}$ coincide on the $\eta$-boundary $\Lambda_{\ell}^{\eta}$ of $\supp(\Q)$, that is, $(\bm{v_1})_j=(\bm{v_2})_j$ for all $j\in \Lambda_{\ell}^{\eta}$. Then, by Definition~\ref{def:boundary}, vectors $\bm{v_1}$ and $\bm{v_2}$ are $\eta$-similar with respect to every $i\in [\ell:n]$, and Theorem~\ref{thm::approx-Q} yields
	\[
	\norm{\bm{H^\ell}(\bm{v_1}) - \bm{H^\ell}(\bm{v_2})}_{2,\infty}
	\;\leq\;
	C_{\Q} \, \sqrt{\Delta_{1} \, \Delta_{\eta}} \, \sigma^\eta,
	\qquad\text{where}\qquad
	C_{\Q}\defeq \frac{10\,\mu^{3.5}_{\max}(\Q)}{\mu^{3}_{\min}(\Q)}
	\]
	and $\Delta_{\eta} \defeq \max_{i \in [n]} \Delta_{i,\eta}$. In other words, all partial solutions that coincide on the $\eta$-boundary can be merged into a single node of an $\epsilon$-exact decision diagram, provided that $\eta$ is sufficiently large. This observation leads to the following main result.
	\begin{corollary}\label{cor::epsilon-Q-size}
		Suppose $\supp(\Q)$ has polynomial volume growth with parameters $\gamma,\delta>0$ (Definition~\ref{def::vol-Q}). Given any $\epsilon>0$, there exists an $\epsilon$-exact decision diagram with at most $(n+1)\cdot 2^{\omega_{\Q}(\eta)}$ nodes, where $\omega_{\Q}(\eta)$ is the maximum size of the $\eta$-boundaries of $\supp(\Q)$ (Definition~\ref{def:boundary}) and
		\begin{align}\label{eq::m-lb-Q}
			\eta \;=\;
			\left\lceil\max\left\{
			\frac{2\log(C_{\Q}\delta / \epsilon)}{\log(1/\sigma)}, \;
			\frac{2\gamma}{\log(1/\sigma)}, \;
			\left(\frac{\gamma}{\log(1/\sigma)}\right)^2
			\right\}\right\rceil.
		\end{align}
	\end{corollary}

	Since the value of $\eta$ in \eqref{eq::m-lb-Q} is independent of $n$ and grows only logarithmically in $1/\epsilon$, Corollary~\ref{cor::epsilon-Q-size} shows that the size of the $\epsilon$-exact decision diagram scales linearly in $n$, multiplied by the factor $2^{\omega_{\Q}(\eta)}$ determined by the boundary widths. In particular, if $\Q$ is banded with bandwidth $k$, then $\omega_{\Q}(\eta)\leq k\eta$ and the size is at most $(n+1)\cdot 2^{k\eta}$, which is polynomial in $1/\epsilon$ (with exponent proportional to $k/\log(1/\sigma)$). Recently, \citet{gomez2024real} established a similar result for banded structures; our guarantee matches theirs in this case and extends it to general sparsity patterns.
    
\subsection{Quality of the resulting solutions}\label{sec:approxQuality}
	
	We close the section by showing that the compression does not come at a significant cost in terms of solution quality.%
	
	\begin{proposition}\label{prop:epsQuality}
		Suppose that $\bm{Q}\succ\bm{0}$, let $\bm{\delta}\in\R^m$ satisfy $\bm{F\delta}=\bm{d}$, and let $\mathcal{D}_\epsilon$ be an $\epsilon$-exact decision diagram obtained via neighborhood merging as in Corollary~\ref{cor::epsilon-Q-size} or Corollary~\ref{cor::epsilon-A-size}, with arc lengths defined as in Proposition~\ref{prop:pathLength} using the states stored in the diagram. For $\bm{z}\in Z$, let $L(\bm{z})$ denote the length of the unique root--terminal path of $\mathcal{D}_\epsilon$ encoding $\bm{z}$ and $g^{n+1}(\bm{z})$ is the actual objective value associated with $\bm{z}$. Then
		$$\left|L(\bm{z})-g^{n+1}(\bm{z})\right|\leq \frac{\|\bm{\delta}\|_2^2}{\sqrt{\mu_{\min}(\Q)}}\, n\epsilon\qquad \forall \bm{z}\in Z.$$
	\end{proposition}
	
    Analogous guarantees were previously known only for banded matrices \citep{gomez2024real}.	In \S\ref{sec:computations}, we illustrate how the theory developed in this section can inform variable ordering through an example arising from stochastic hybrid model predictive control.

\section{Computational study on variable ordering}\label{sec:computations}

The size of the decision diagrams usually depends on the order in which variables are processed (an exception being the low-rank setting in \S\ref{sec:lowRank}). The polynomial-time proofs in \S\ref{sec:tree} and \S\ref{sec:invTree} for trees and inverse trees, respectively, assume that the variables are processed starting from the leaves of the trees and finishing at the root. However, in practice and when using $\epsilon$-exact diagrams, this order might not achieve the best performance. Indeed, as we illustrate computationally in this section, better orders can be obtained by controlling the sizes of the $\eta$-boundaries throughout the tree.

\subsection{Motivating setting}

We consider a simple one-dimensional stochastic hybrid model predictive control problem. A scalar state evolves according to linear dynamics driven by a continuous control input, and a binary indicator determines whether that control is active; the objective combines quadratic state-tracking error with fixed activation costs \citep{bemporad1999control,borrelli2017predictive}. The first-period control is selected before the subsequent dynamics, reference trajectories, and costs are revealed. This uncertainty is represented by $k$ scenarios, yielding the following two-stage formulation:
\begin{subequations}\label{eq:stochasticMPC}
	\begin{align}
		\min_{\substack{s_1,x_0,z_0\\\{\bm{s^j},\bm{x^j},\bm{z^j}\}_{j=1}^k}}\;&(s_1-\bar s_1)^2+\sum_{j=1}^k\sum_{t=2}^{T}(s_t^j-\bar s_t^j)^2+c_0z_0+\sum_{j=1}^k\sum_{t=1}^{T-1}c_t^j z_t^j\label{eq:stochasticMPC_obj}\\
		\text{s.t.}\;&s_1=x_0\label{eq:stochasticMPC_dynamic1}\\
		&s_2^j=a_1^j s_1+x_1^j \qquad\forall j\in [k]\label{eq:stochasticMPC_dynamic2}\\
		&s_{t+1}^j=a_t^j s_t^j+x_t^j \qquad\forall j\in [k],\; \forall t\in [2:T-1]\label{eq:stochasticMPC_dynamic3}\\
		&x_{0}(1-z_0)=0\\
		&x_t^j(1-z_t^j)=0\qquad\forall j\in [k],\; \forall t\in [T-1]\\
		&(s_1,x_0,z_0)\in \R\times \R\times \{0,1\}\\
		&(\bm{s^j},\bm{x^j},\bm{z^j})\in \R^{[2:T]}\times\R^{[T-1]}\times\{0,1\}^{[T-1]}\qquad \forall j\in [k],
	\end{align}
\end{subequations}
where $(\bm{a^j},\bm{s^j},\bm{c^j})$ are the dynamics, trajectories and costs under scenario $j\in[k]$.

After projecting out variables $\bm{s}$ using dynamics \eqref{eq:stochasticMPC_dynamic1}-\eqref{eq:stochasticMPC_dynamic3}, the resulting quadratic objective in variables $\bm{x}$ involves a fully dense matrix $\bm{Q}$. However, the inverse $\bm{A}=\bm{Q^{-1}}$ is sparse, see Figure~\ref{fig:varOrders}(a) for an example. Specifically, $\supp(\A)$ includes a $(k+1)$-dimensional clique associated with variables $x_0$ and $x_1^j$ for all $j\in [k]$, but the only other arcs are $(x_t^j, x_{t+1}^{j})$ for $j\in [k]$ and $t\in [T-1]$. 

\subsection{Variable orders}

We consider three potential orders $\bm{\pi}$ of the $1+(T-1)k$ variables $x_0$ and $\{x_t^j\}_{j, t}$ to tackle \eqref{eq:stochasticMPC}. The support graph of $\bm{A}$ and the three orders are depicted in Figure~\ref{fig:varOrders}.

\paragraph{Natural order} The natural order, depicted in Figure~\ref{fig:varOrders}(b), in which variables are ordered  first according to their time component and ties are broken according to the scenario index, that is, $$\bm{\pi}=(x_0,x_1^1, x_1^2,\dots,x_1^k, x_2^1, x_2^2,\dots,x_2^k,\dots,x_{T-1}^{1},\dots, x_{T-1}^{k-1},x_{T-1}^{k}).$$
Observe that using the natural order, the resulting matrix $\bm{A}$ has bandwidth $k$, and thus $\supp(\A)$ has linear volume growth and, using Theorem~\ref{thm:approxSummary}, we find that $\epsilon$-exact decision diagrams have size proportional to $T(1/\epsilon)^k$ (where additional terms depending on the eigenvalues of $\bm{A}$ are omitted).

\paragraph{Reverse order} The reverse of the natural order, depicted in Figure~\ref{fig:varOrders}(c), which preserves the bandwidth $k$. This order is also similar to the one used in \S\ref{sec:invTree} to prove polynomial sizes for the inverse-tree setting. While the results in \S\ref{sec:invTree} are not directly applicable since $\supp(\A)$ is not a tree, the proof techniques can be reused to prove that, with the reverse order, exact decision diagrams are still polynomial in $T$ but exponential in $k$. Indeed, during the first $(T-2)k$ layers the decision diagram behaves identically to a tree; since the number of nodes can at most double from one layer to the next, the number of nodes in the last $k+1$ grows by a factor of at most $2^{k+1}$. We leave a formal proof of this result to the reader. 

\paragraph{Leg order}  In this order, depicted in Figure~\ref{fig:varOrders}(d), variables are ordered first according to their scenario index, ties broken in reverse order of the time index, and variable $x_0$ is last, that is, $$\bm{\pi}=(x_{T-1}^1,x_{T-2}^1,\dots, x_1^1,x_{T-1}^2,x_{T-2}^2,\dots, x_{T-1}^k,\dots,x_2^k,x_1^k,x_0).$$
Note that, under this permutation, the matrix $\bm{A}$ is no longer banded. Nonetheless, the rationale for using this order is that its $\eta$-boundaries are typically smaller than those of the natural and reverse orders. Indeed, for the natural and reverse orders, the $\eta$-boundaries have size $\eta k$ after the first $k$ levels. In contrast, for the leg order, $\eta$-boundaries have size $\eta$ while processing nodes associated with the first scenario, $2\eta$ while processing nodes associated with the second scenario and so forth. While the size of the worst-case $\eta$-boundary is still $k\eta$ after processing nodes associated with the last scenario (resulting in the same worst-case bounds according to Theorem~\ref{thm:approxSummary}), smaller sizes throughout most of the execution of the algorithm can yield additional merging and smaller decision diagrams. Figure~\ref{fig:boundaryOrders} illustrates the $\eta$-boundaries induced by the reverse and leg orders in the example of Figure~\ref{fig:varOrders}.

\begin{figure}[!ht]
	\centering
	\subfigure[Support graph $\supp(\bm{A})$]{
		\begin{tikzpicture}[scale=0.72,
			varnode/.style={circle,draw,fill=white,font=\tiny,inner sep=1.2pt},
			hubnode/.style={circle,draw,fill=gray!30,font=\tiny,inner sep=1.2pt}]
			\fill[gray!15] (0,0) circle (1.45);
			\node[font=\scriptsize,text=gray] at (150:2.1) {clique};
			\node[hubnode] (x0) at (0,0) {$x_0$};
			\foreach \j/\a in {1/90,2/210,3/330}{
				\foreach \t in {1,...,4}{
					\node[varnode] (x\j\t) at ({\a}:{1.08*\t}) {$x_{\t}^{\j}$};
				}
			}
			\draw (x0)--(x11); \draw (x0)--(x21); \draw (x0)--(x31);
			\draw (x11)--(x21); \draw (x21)--(x31); \draw (x31)--(x11);
			\foreach \j in {1,2,3}{\foreach \t/\s in {1/2,2/3,3/4}{\draw (x\j\t)--(x\j\s);}}
		\end{tikzpicture}
	}
	\hspace{2em}
	\subfigure[Natural order]{
		\begin{tikzpicture}[scale=0.62,
			numnode/.style={circle,draw,fill=white,font=\tiny,inner sep=1.4pt},
			hubnode/.style={circle,draw,fill=gray!30,font=\tiny,inner sep=1.4pt}]
			\fill[gray!15] (0,0) circle (1.45);
			\node[hubnode] (x0) at (0,0) {$1$};
			\foreach \j/\a/\na/\nb/\nc/\nd in {1/90/2/5/8/11, 2/210/3/6/9/12, 3/330/4/7/10/13}{
				\node[numnode] (x\j1) at ({\a}:0.95) {$\na$};
				\node[numnode] (x\j2) at ({\a}:1.9) {$\nb$};
				\node[numnode] (x\j3) at ({\a}:2.85) {$\nc$};
				\node[numnode] (x\j4) at ({\a}:3.8) {$\nd$};
			}
			\draw (x0)--(x11); \draw (x0)--(x21); \draw (x0)--(x31);
			\draw (x11)--(x21); \draw (x21)--(x31); \draw (x31)--(x11);
			\foreach \j in {1,2,3}{\foreach \t/\s in {1/2,2/3,3/4}{\draw (x\j\t)--(x\j\s);}}
		\end{tikzpicture}
	}
	
	\subfigure[Reverse order]{
		\begin{tikzpicture}[scale=0.62,
			numnode/.style={circle,draw,fill=white,font=\tiny,inner sep=1.4pt},
			hubnode/.style={circle,draw,fill=gray!30,font=\tiny,inner sep=1.4pt}]
			\fill[gray!15] (0,0) circle (1.45);
			\node[hubnode] (x0) at (0,0) {$13$};
			\foreach \j/\a/\na/\nb/\nc/\nd in {1/90/12/9/6/3, 2/210/11/8/5/2, 3/330/10/7/4/1}{
				\node[numnode] (x\j1) at ({\a}:0.95) {$\na$};
				\node[numnode] (x\j2) at ({\a}:1.9) {$\nb$};
				\node[numnode] (x\j3) at ({\a}:2.85) {$\nc$};
				\node[numnode] (x\j4) at ({\a}:3.8) {$\nd$};
			}
			\draw (x0)--(x11); \draw (x0)--(x21); \draw (x0)--(x31);
			\draw (x11)--(x21); \draw (x21)--(x31); \draw (x31)--(x11);
			\foreach \j in {1,2,3}{\foreach \t/\s in {1/2,2/3,3/4}{\draw (x\j\t)--(x\j\s);}}
		\end{tikzpicture}
	}
	\hspace{2em}
	\subfigure[Leg order]{
		\begin{tikzpicture}[scale=0.62,
			numnode/.style={circle,draw,fill=white,font=\tiny,inner sep=1.4pt},
			hubnode/.style={circle,draw,fill=gray!30,font=\tiny,inner sep=1.4pt}]
			\fill[gray!15] (0,0) circle (1.45);
			\node[hubnode] (x0) at (0,0) {$13$};
			\foreach \j/\a/\na/\nb/\nc/\nd in {1/90/4/3/2/1, 2/210/8/7/6/5, 3/330/12/11/10/9}{
				\node[numnode] (x\j1) at ({\a}:0.95) {$\na$};
				\node[numnode] (x\j2) at ({\a}:1.9) {$\nb$};
				\node[numnode] (x\j3) at ({\a}:2.85) {$\nc$};
				\node[numnode] (x\j4) at ({\a}:3.8) {$\nd$};
			}
			\draw (x0)--(x11); \draw (x0)--(x21); \draw (x0)--(x31);
			\draw (x11)--(x21); \draw (x21)--(x31); \draw (x31)--(x11);
			\foreach \j in {1,2,3}{\foreach \t/\s in {1/2,2/3,3/4}{\draw (x\j\t)--(x\j\s);}}
		\end{tikzpicture}
	}
	\caption{Illustration of the support graph of $\bm{A}=\bm{Q^{-1}}$ and of the three variable orders for the stochastic hybrid MPC problem \eqref{eq:stochasticMPC} with $k=3$ scenarios and $T=5$ time periods.}
	\label{fig:varOrders}
\end{figure}
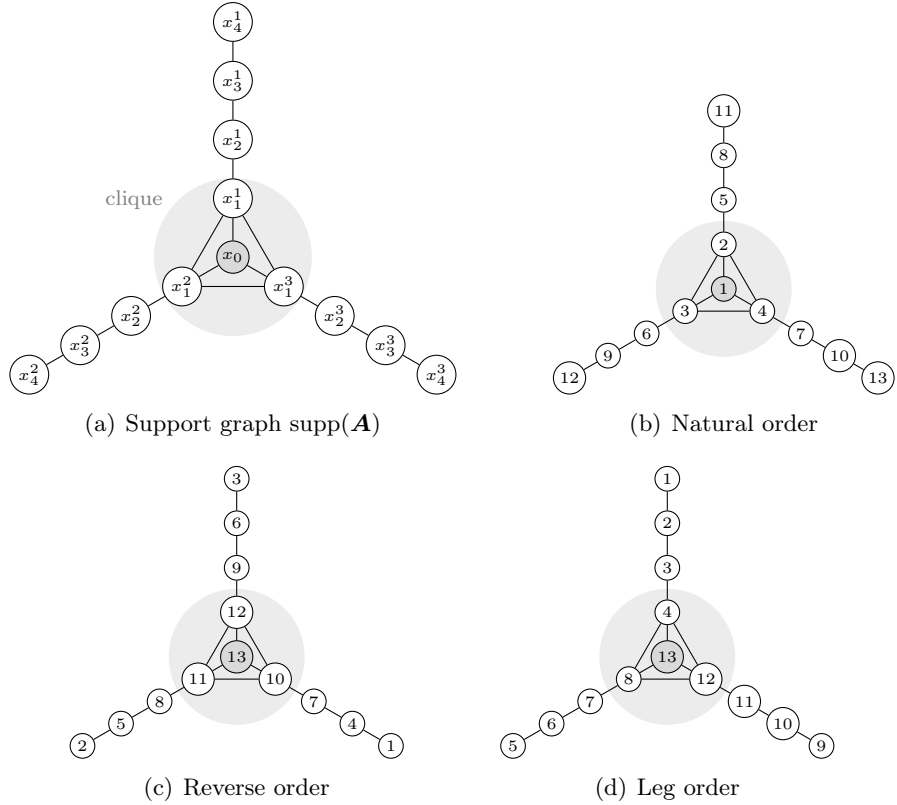

\begin{figure}[!ht]
	\centering
	\subfigure[Reverse order, $\ell=9$]{
		\begin{tikzpicture}[scale=0.62,
			dnode/.style={circle,draw,fill=white,font=\tiny,inner sep=1.4pt},
			unode/.style={circle,fill=gray!50,font=\tiny,inner sep=1.4pt}]
			\node[unode] (x0) at (0,0) {$13$};
			\foreach \j/\a/\na/\fa/\nb/\fb/\nc/\fc/\nd/\fd in {1/90/12/unode/9/unode/6/dnode/3/dnode, 2/210/11/unode/8/dnode/5/dnode/2/dnode, 3/330/10/unode/7/dnode/4/dnode/1/dnode}{
				\node[\fa] (x\j1) at ({\a}:0.95) {$\na$};
				\node[\fb] (x\j2) at ({\a}:1.9) {$\nb$};
				\node[\fc] (x\j3) at ({\a}:2.85) {$\nc$};
				\node[\fd] (x\j4) at ({\a}:3.8) {$\nd$};
			}
			\draw (x0)--(x11); \draw (x0)--(x21); \draw (x0)--(x31);
			\draw (x11)--(x21); \draw (x21)--(x31); \draw (x31)--(x11);
			\foreach \j in {1,2,3}{\foreach \t/\s in {1/2,2/3,3/4}{\draw (x\j\t)--(x\j\s);}}
			\draw[dashed] (x13) circle (0.42); \draw[dashed] (x14) circle (0.42);
			\draw[dashed] (x22) circle (0.42); \draw[dashed] (x23) circle (0.42);
			\draw[dashed] (x32) circle (0.42); \draw[dashed] (x33) circle (0.42);
			\node[font=\scriptsize] at ($(x14)+(1.3,0.15)$) {$\Lambda_{9}^{2}$};
		\end{tikzpicture}
	}
	\hspace{2em}
	\subfigure[Leg order, $\ell=4$]{
		\begin{tikzpicture}[scale=0.62,
			dnode/.style={circle,draw,fill=white,font=\tiny,inner sep=1.4pt},
			unode/.style={circle,fill=gray!50,font=\tiny,inner sep=1.4pt}]
			\node[unode] (x0) at (0,0) {$13$};
			\foreach \j/\a/\na/\fa/\nb/\fb/\nc/\fc/\nd/\fd in {1/90/4/unode/3/dnode/2/dnode/1/dnode, 2/210/8/unode/7/unode/6/unode/5/unode, 3/330/12/unode/11/unode/10/unode/9/unode}{
				\node[\fa] (x\j1) at ({\a}:0.95) {$\na$};
				\node[\fb] (x\j2) at ({\a}:1.9) {$\nb$};
				\node[\fc] (x\j3) at ({\a}:2.85) {$\nc$};
				\node[\fd] (x\j4) at ({\a}:3.8) {$\nd$};
			}
			\draw (x0)--(x11); \draw (x0)--(x21); \draw (x0)--(x31);
			\draw (x11)--(x21); \draw (x21)--(x31); \draw (x31)--(x11);
			\foreach \j in {1,2,3}{\foreach \t/\s in {1/2,2/3,3/4}{\draw (x\j\t)--(x\j\s);}}
			\draw[dashed] (x12) circle (0.42); \draw[dashed] (x13) circle (0.42);
			\node[font=\scriptsize] at ($(x12)+(1.3,0.0)$) {$\Lambda_{4}^{2}$};
		\end{tikzpicture}
	}
	
	\subfigure[Leg order, $\ell=9$]{
		\begin{tikzpicture}[scale=0.62,
			dnode/.style={circle,draw,fill=white,font=\tiny,inner sep=1.4pt},
			unode/.style={circle,fill=gray!50,font=\tiny,inner sep=1.4pt}]
			\node[unode] (x0) at (0,0) {$13$};
			\foreach \j/\a/\na/\fa/\nb/\fb/\nc/\fc/\nd/\fd in {1/90/4/dnode/3/dnode/2/dnode/1/dnode, 2/210/8/dnode/7/dnode/6/dnode/5/dnode, 3/330/12/unode/11/unode/10/unode/9/unode}{
				\node[\fa] (x\j1) at ({\a}:0.95) {$\na$};
				\node[\fb] (x\j2) at ({\a}:1.9) {$\nb$};
				\node[\fc] (x\j3) at ({\a}:2.85) {$\nc$};
				\node[\fd] (x\j4) at ({\a}:3.8) {$\nd$};
			}
			\draw (x0)--(x11); \draw (x0)--(x21); \draw (x0)--(x31);
			\draw (x11)--(x21); \draw (x21)--(x31); \draw (x31)--(x11);
			\foreach \j in {1,2,3}{\foreach \t/\s in {1/2,2/3,3/4}{\draw (x\j\t)--(x\j\s);}}
			\draw[dashed] (x11) circle (0.42); \draw[dashed] (x12) circle (0.42);
			\draw[dashed] (x21) circle (0.42); \draw[dashed] (x22) circle (0.42);
			\node[font=\scriptsize] at ($(x12)+(1.3,0.0)$) {$\Lambda_{9}^{2}$};
		\end{tikzpicture}
	}
	\hspace{2em}
	\subfigure[Leg order, $\ell=12$]{
		\begin{tikzpicture}[scale=0.62,
			dnode/.style={circle,draw,fill=white,font=\tiny,inner sep=1.4pt},
			unode/.style={circle,fill=gray!50,font=\tiny,inner sep=1.4pt}]
			\node[unode] (x0) at (0,0) {$13$};
			\foreach \j/\a/\na/\fa/\nb/\fb/\nc/\fc/\nd/\fd in {1/90/4/dnode/3/dnode/2/dnode/1/dnode, 2/210/8/dnode/7/dnode/6/dnode/5/dnode, 3/330/12/unode/11/dnode/10/dnode/9/dnode}{
				\node[\fa] (x\j1) at ({\a}:0.95) {$\na$};
				\node[\fb] (x\j2) at ({\a}:1.9) {$\nb$};
				\node[\fc] (x\j3) at ({\a}:2.85) {$\nc$};
				\node[\fd] (x\j4) at ({\a}:3.8) {$\nd$};
			}
			\draw (x0)--(x11); \draw (x0)--(x21); \draw (x0)--(x31);
			\draw (x11)--(x21); \draw (x21)--(x31); \draw (x31)--(x11);
			\foreach \j in {1,2,3}{\foreach \t/\s in {1/2,2/3,3/4}{\draw (x\j\t)--(x\j\s);}}
			\draw[dashed] (x11) circle (0.42); \draw[dashed] (x12) circle (0.42);
			\draw[dashed] (x21) circle (0.42); \draw[dashed] (x22) circle (0.42);
			\draw[dashed] (x32) circle (0.42); \draw[dashed] (x33) circle (0.42);
			\node[font=\scriptsize] at ($(x12)+(1.3,0.0)$) {$\Lambda_{12}^{2}$};
		\end{tikzpicture}
	}
	\caption{Illustration of the $\eta$-boundaries (Definition~\ref{def:boundary}) with $\eta=2$ induced by the reverse and leg orders on the instance of Figure~\ref{fig:varOrders} ($k=3$, $T=5$). In each panel, nodes are labeled by their position in the corresponding order, decided nodes ($[\ell-1]$) are white, undecided nodes ($[\ell:n]$) are gray, and the nodes of the $\eta$-boundary $\Lambda_{\ell}^{\eta}$ are marked with dashed halos. (a) In the reverse order the $k$ legs are processed in parallel, and the boundary contains variables of every leg throughout the construction: at $\ell=9$, the boundary already has the worst-case size $\eta k=6$. (b)--(d) In the leg order the boundary is confined to the leg currently being processed, plus at most $\eta$ variables adjacent to the clique for each previously completed leg: the worst-case size $\eta k=6$ is attained only while completing the last leg. In particular, at the same layer $\ell=9$, the boundary under the leg order (of size $2\eta=4$, panel (c)) is smaller than under the reverse order (of size $\eta k=6$, panel (a)).}
	\label{fig:boundaryOrders}
\end{figure}
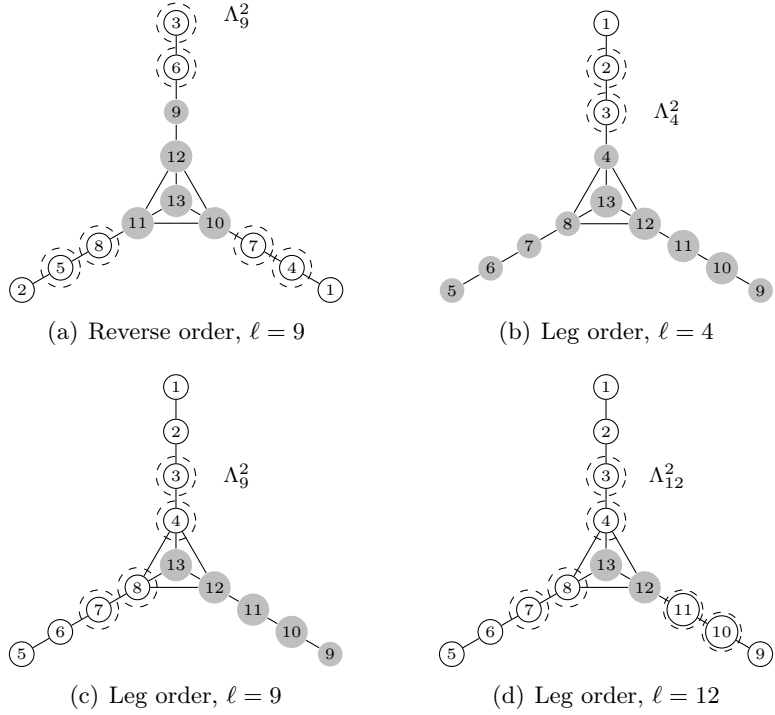

\subsection{Experiments}
We compare the three variable orders over instances of the stochastic hybrid model predictive control problem across different performance measures.

\paragraph{Data generation}
Each element is generated independently as $a_{t}^j\sim U(0,1)$, resulting in different dynamics across time periods and scenarios. Costs for the indicator variables are also generated independently as $c_t^j\sim U(0,1)$. Finally, instead of the linear expression $-2\bar s_1s_1-2\sum_{j=1}^k\sum_{t=2}^{T}\bar s_t^j s_t^j$ induced by the objective \eqref{eq:stochasticMPC_obj} (which can be transformed to an equivalent linear expression involving variables $\bm{x}$ after projecting out $\bm{s}$), we simply set an objective of the form $\bm{d^\top x}$, where each coefficient of $\bm{d}$ is generated from $U(0,1)$. While this choice of objective is not the most realistic, we point out that the choice of linear coefficients does not affect the construction of the decision diagram and thus does not significantly impact the results presented next. We generate instances for the numbers of scenarios $k\in \{3,4,5\}$ and time periods $T\in \{10,20,50\}$ and, for each combination of parameters, we generate five instances. 

\paragraph{Results}
Table~\ref{tab:varOrder} shows the results of these experiments, using Gurobi 12 with default settings, a 300-second time limit and a single thread. For each ordering, we report the CPU time needed to construct the diagram and solve the shortest-path problems in seconds, the number of arcs in the resulting diagrams, the number of instances solved within the time limit, and the error in objective value, which we compute and report only for instances solved to optimality by both methods as follows: 
\[
\text{error} = \frac{|\text{val}_{DD}-\text{val}^*|}{|\text{val}^*|},
\]
where $\text{val}_{DD}$ is the value obtained by the shortest path problem over the diagram and $\text{val}^*$ is the optimal value reported by Gurobi.

Table~\ref{tab:varOrder} illustrates that the variable ordering can indeed have a significant effect on the performance of the methods. The leg order results in considerably smaller diagrams across all values of $\epsilon$, reducing the average number of arcs by roughly half when compared to the natural order. Consequently, more instances are solved within the time limit: the natural order solves 76, the reverse order solves 88, and the leg order solves 112. For reference, Gurobi solves 90 instances within the time limit. Interestingly, Gurobi cannot solve any instances with $T = 50$ whereas the decision-diagram-based methods solve most of these instances when $k=3$ and even for $k=5$ the leg order method is competitive.  

Regarding the effect of $\epsilon$, we note that the errors for the reverse and leg orders are consistently below $0.05\%$ across all values of $\epsilon$, while the error for the natural order increases only marginally, to at most $0.15\%$ when $\epsilon = 10^{-2}$. 

\begin{landscape}    	
	\begin{table}[!ht]
		\begin{center}
			\singlespacing
			\footnotesize
			\renewcommand{\arraystretch}{0.72}
			\caption{Comparing variable orderings over instances of model predictive control problems. Each row reports averages over five instances, which are solved with a time limit of 300 seconds. For instances not solved within the time limit, we record a CPU time of 300 seconds.}
			\label{tab:varOrder}
			\resizebox{\linewidth}{!}{%
				\begin{tabular}{rcrrrrrrrrrrrrrrr}
					\hline
					{\bf } &     {\bf } &     {\bf } &          \multicolumn{ 4}{c}{{\bf Natural order}} &          \multicolumn{ 4}{c}{{\bf Reverse order}} &              \multicolumn{ 4}{c}{{\bf Leg order}} & \multicolumn{ 2}{c}{{\bf Gurobi}} \\
					
					{\bf $k$} & {\bf $\epsilon$} & {\bf $T$} & {\bf Time (s)} & {\bf Arcs} & {\bf \# Sol} &  {\bf Error} & {\bf Time (s)} & {\bf Arcs} & {\bf \# Sol} &  {\bf Error} & {\bf Time (s)} & {\bf Arcs} & {\bf \# Sol} &  {\bf Error} & {\bf Time (s)} & {\bf \# Sol} \\
					\hline
					\hline
					&            &         10 &       1.09 &      55613 &          5 &     0.00\% &       0.39 &      21552 &          5 &     0.00\% &       0.24 &      15366 &          5 &     0.00\% &       0.19 &          5 \\
					
					&      $10^{-5}$ &         20 &      15.55 &     203440 &          5 &     0.02\% &       7.72 &     123231 &          5 &     0.02\% &       2.16 &      61447 &          5 &     0.02\% &       0.27 &          5 \\
					
					&            &         50 &     275.02 &     360476 &          1 &          - &     257.72 &     455294 &          3 &          - &      29.38 &     179381 &          5 &          - &     300.00 &          0 \\
					
					&            &         10 &       0.50 &      20362 &          5 &     0.00\% &       0.28 &      13604 &          5 &     0.00\% &       0.14 &       7502 &          5 &     0.00\% &       0.19 &          5 \\
					
					3 &      $10^{-3}$ &         20 &       3.70 &      54240 &          5 &     0.03\% &       2.88 &      47279 &          5 &     0.03\% &       0.73 &      19888 &          5 &     0.03\% &       0.27 &          5 \\
					
					&            &         50 &      49.14 &     153455 &          5 &          - &      46.98 &     148202 &          5 &          - &       7.75 &      54342 &          5 &          - &     300.00 &          0 \\
					
					&            &         10 &       0.24 &      10150 &          5 &     0.02\% &       0.20 &       8448 &          5 &     0.02\% &       0.08 &       3936 &          5 &     0.02\% &       0.31 &          5 \\
					
					&      $10^{-2}$ &         20 &       1.49 &      22812 &          5 &     0.04\% &       1.20 &      21255 &          5 &     0.04\% &       0.36 &       8974 &          5 &     0.04\% &       0.36 &          5 \\
					
					&            &         50 &      17.59 &      64049 &          5 &          - &      17.76 &      63350 &          5 &          - &       3.50 &      23919 &          5 &          - &     300.00 &          0 \\
					\hline
					&            &         10 &      41.84 &    1649219 &          5 &     0.01\% &       5.23 &     255477 &          5 &     0.01\% &       2.45 &     171016 &          5 &     0.01\% &       0.31 &          5 \\
					
					&      $10^{-5}$ &         20 &     300.00 &     278637 &          0 &          - &     300.00 &     775666 &          0 &          - &      27.41 &     755537 &          5 &     0.03\% &       1.29 &          5 \\
					
					&            &         50 &     300.00 &     130321 &          0 &          - &     300.00 &     178879 &          0 &          - &     300.00 &     351381 &          0 &          - &     300.00 &          0 \\
					
					&            &         10 &       8.14 &     282816 &          5 &     0.01\% &       2.60 &     111420 &          5 &     0.01\% &       0.96 &      63288 &          5 &     0.01\% &       0.24 &          5 \\
					
					4 &      $10^{-3}$ &         20 &     190.78 &     536476 &          4 &     0.03\% &      59.65 &     461543 &          5 &     0.02\% &       5.72 &     162783 &          5 &     0.02\% &       0.98 &          5 \\
					
					&            &         50 &     300.00 &     172725 &          0 &          - &     300.00 &     205857 &          0 &          - &     170.06 &     450648 &          5 &          - &     300.00 &          0 \\
					
					&            &         10 &       2.01 &      65494 &          5 &     0.01\% &       1.19 &      49273 &          5 &     0.01\% &       0.33 &      19154 &          5 &     0.01\% &       0.39 &          5 \\
					
					&     $10^{-2}$ &         20 &      15.60 &     159412 &          5 &     0.05\% &      13.20 &     146660 &          5 &     0.05\% &       1.76 &      46658 &          5 &     0.05\% &       1.32 &          5 \\
					
					&            &         50 &     300.00 &     238995 &          0 &          - &     288.29 &     265780 &          1 &          - &      26.03 &     125066 &          5 &          - &     300.00 &          0 \\
					\hline
					&            &         10 &     300.00 &    2890545 &          0 &          - &      70.53 &    2807484 &          5 &     0.02\% &      30.16 &    1769758 &          5 &     0.02\% &       0.62 &          5 \\
					
					&      $10^{-5}$ &         20 &     300.00 &     188146 &          0 &          - &     300.00 &     214388 &          0 &          - &     300.00 &     573269 &          0 &          - &      13.16 &          5 \\
					
					&            &         50 &     300.00 &      83718 &          0 &          - &     300.00 &     104579 &          0 &          - &     300.00 &     210954 &          0 &          - &     300.00 &          0 \\
					
					&            &         10 &     175.06 &    4672707 &          5 &     0.02\% &      28.32 &    1000803 &          5 &     0.02\% &       5.91 &     361514 &          5 &     0.02\% &       0.28 &          5 \\
					
					5 &     $10^{-3}$ &         20 &     300.00 &     205703 &          0 &          - &     300.00 &     271160 &          0 &          - &      49.73 &     929330 &          5 &     0.03\% &      13.12 &          5 \\
					
					&            &         50 &     300.00 &     104746 &          0 &          - &     300.00 &     155526 &          0 &          - &     300.00 &     265620 &          0 &          - &     300.00 &          0 \\
					
					&            &         10 &      21.03 &     574846 &          5 &     0.15\% &       8.87 &     308757 &          5 &     0.03\% &       1.38 &      82697 &          5 &     0.03\% &       0.67 &          5 \\
					
					&      $10^{-2}$ &         20 &     266.28 &     297221 &          1 &     0.06\% &     242.83 &     803857 &          4 &     0.03\% &       7.50 &     175541 &          5 &     0.02\% &      13.28 &          5 \\
					
					&            &         50 &     300.00 &     122934 &          0 &          - &     300.00 &     151934 &          0 &          - &     258.73 &     390826 &          2 &          - &     300.00 &          0 \\
					\hline
					&            & {\bf Total} & {\bf 151.30} & {\bf 503676} &   {\bf 76} & {\bf 0.03\%} & {\bf 127.99} & {\bf 339676} &   {\bf 88} & {\bf 0.02\%} & {\bf 67.87} & {\bf 269622} &  {\bf 112} & {\bf 0.02\%} & {\bf 101.75} &   {\bf 90} \\
					\hline
				\end{tabular}%
			}
		\end{center}
	\end{table}
	
\end{landscape}

	\section{Computations with real data}\label{sec:realData}
	
	In this section, we report computations with real data.
	
	\subsection{Portfolio index tracking}\label{sec:portfolioComp}
	
We consider the portfolio index tracking problem: given a benchmark portfolio $\bm{\bar x}\in \R_+^n$ with $\bm{1^\top \bar x}=1$, the goal is to construct a portfolio holding at most $k$ assets that tracks the benchmark as closely as possible, where the tracking error is measured with respect to the estimated covariance of the returns. Factor models are standard in portfolio risk estimation \citep{sharpe1963simplified}; accordingly, we estimate the covariance as $\bm{Q}+\Diag(\bm{d})$, where $\bm{Q}=\bm{FF^\top}=\sum_{j=1}^m \bm{f_j}\bm{f_j^\top}$ is the low-rank covariance induced by the factors $\bm{F}=\begin{pmatrix}\bm{f_1}&\cdots&\bm{f_m}\end{pmatrix}\in \R^{n\times m}$, and $\bm{d}\in \R_+^n$ collects the specific variances not explained by the factors. The resulting optimization problem is
	\begin{subequations}\label{eq:indexTracking}
		\begin{align}
			\min_{\bm{x}, \bm{z}\in \R^n}\;& (\bm{x}-\bm{\bar x})^\top\left(\bm{Q}+\Diag(\bm{d})\right)(\bm{x}-\bm{\bar x})\\
			\text{s.t.}\;&\bm{1^\top x}=1,\quad \bm{1^\top z}\leq k,\quad
			\bm{0}\leq \bm{x}\leq \bm{z}, \quad \bm{z}\in \{0,1\}^n,
		\end{align}
	\end{subequations}
	where $k\in \Z_+$ is the target sparsity. We denote by ${X}_{\textsf{PIT}}$ the feasible region of \eqref{eq:indexTracking}.
	
Problem \eqref{eq:indexTracking} is closely related to cardinality-constrained mean--variance portfolio selection, which trades expected return against variance under budget and support constraints \citep{markowitz1952portfolio,bienstock1996computational,gao2011cardinality}, but its relaxation is much weaker. Indeed, the natural convex relaxation of \eqref{eq:indexTracking} always produces a trivial lower bound of 0 and a root gap of 100\%, because the solution $\bm{z}=\bm{x}=\bm{\bar x}$ is feasible for the relaxation.
	
	\subsubsection*{Instance construction}
	We generate instances from real stock market data, following the methodology of \citet{han2x2}. We use the daily stock return data provided by Boris Marjanovic in Kaggle ({\url{kaggle.com/borismarjanovic/price-volume-data-for-all-us-stocks-etfs}}), using all data since 1/1/1991 and retaining the stocks with available data in at least 99\% of the trading days since that date (returns corresponding to missing data are set to 0), resulting in $n=214$ securities. From these data, we compute the expected return $\bm{\bar \mu}$ and the sample covariance matrix $\bm{\bar\Sigma}$ of the daily returns (\%). We set the benchmark $\bm{\bar x}=\bm{\bar \mu}/\|\bm{\bar\mu}\|_1$ to ensure $\|\bm{\bar x}\|_1=1$ holds.
	
	The factor model is then constructed as follows. Given a target rank $m\in \Z_+$ and a target factor sparsity $\kappa\in \Z_+\cup\{\infty\}$, we compute a sparse eigendecomposition $\bm{\bar\Sigma}\approx\sum_{j=1}^m\bm{f_j}\bm{f_j^\top}$ using the truncated power method \citep{yuan2013truncated}, in which the factors are ordered by decreasing explained variance and the leading $m$ factors satisfy $\|\bm{f_j}\|_0\leq \kappa$. 
    We then set
	$$\bm{Q}=\sum_{j=1}^m\bm{f_j}\bm{f_j^\top}\qquad \text{and}\qquad \bm{d}=\text{diag}\Big(\bm{\bar \Sigma}-\bm{Q}\Big),$$
	that is, the leading $m$ (sparse) factors are retained exactly, while the variance explained by the remaining factors is approximated by its diagonal. The case $\kappa=\infty$ imposes no sparsity and corresponds to the usual truncated eigendecomposition. Larger values of $m$ and $\kappa$ thus yield more faithful approximations of $\bm{\bar\Sigma}$, but result in denser and higher-rank matrices $\bm{Q}$.
		Detailed results on the quality of the factor approximation are provided in the online companion (\S\ref{sec:portfolio-approx-quality}).

	\subsubsection*{Formulations}
	
	Expanding the objective of \eqref{eq:indexTracking}, the tracking error decomposes as
	$$(\bm{x}-\bm{\bar x})^\top\left(\bm{Q}+\Diag(\bm{d})\right)(\bm{x}-\bm{\bar x})=L(\bm{x})+\sum_{j=1}^m\left(\bm{f_j^\top x}\right)^2+\sum_{i=1}^nd_ix_i^2,$$
	where $L(\bm{x})\defeq \bm{\bar x^\top}\left(\bm{Q}+\Diag(\bm{d})\right)\bm{\bar x}-2\bm{\bar x^\top}\left(\bm{Q}+\Diag(\bm{d})\right)\bm{x}$ collects the constant and the linear terms. The formulations we compare share the term $L(\bm{x})$, and differ only in the strengthening applied to the quadratic terms. We consider two families of formulations.
	
	\smallskip\noindent$\bullet$ \textbf{The perspective formulation (\texttt{PER})}, in which the separable part is strengthened using the perspective reformulation and the factor terms are left as plain quadratics:
	\begin{align*}
		\min_{\bm{x},\, \bm{z},\, \bm{\tau}}\;& L(\bm{x})+  \sum_{j=1}^m \tau_j^2+\sum_{i=1}^n d_i\frac{x_i^2}{z_i}
		\text{  s.t.  }\;(\bm x, \bm z)\in X_{\mathsf{PIT}},\quad\tau_j=\bm{f_j^\top x}\qquad j\in [m].%
	\end{align*}
	
	\smallskip\noindent$\bullet$ \textbf{The decision diagram formulations (\texttt{DD1} and \texttt{DD2})}, in which, in addition to the perspective strengthening of the separable part, the quadratic associated with the leading $m'\in \Z_+$ factors ($m'\leq m$) is convexified using the results of \S\ref{sec:convexification}. Specifically, letting $\bm{F'}=\begin{pmatrix}\bm{f_1}&\cdots&\bm{f_{m'}}\end{pmatrix}$, we construct a decision diagram for $X_{\bm{F'},Z}$ with $Z=\{0,1\}^n$ and replace the term $\sum_{j=1}^{m'}(\bm{f_j^\top x})^2$ with an epigraph variable $x_0$ constrained through the extended formulation \eqref{eq:convexhull}:
	\begin{align*}
		\min_{\bm{x},\, \bm{z},\, \bm{\tau},\, x_0,\, \bm{w},\, \bm{r}}\;& L(\bm{x})+  x_0+  \sum_{j=m'+1}^m \tau_j^2+\sum_{i=1}^n d_i\frac{x_i^2}{z_i}\\
		\text{s.t.}\;&\tau_j=\bm{f_j^\top x}\;\; \forall j\in [m'+1:m],\quad
		(\bm x, \bm z)\in X_{\mathsf{PIT}}, \quad (x_0,\bm{x},\bm{z},\bm{w},\bm{r})\in R_{\bm{F'},Z}.
	\end{align*}
	Formulations \texttt{DD1} and \texttt{DD2} correspond to $m'=1$ and $m'=2$, respectively. Note that the cardinality constraint is not encoded in the diagram, but kept explicitly in the formulation. In our computations we use $\epsilon$-exact diagrams constructed with Algorithm~\ref{alg:constructionDD} and the hash-merging procedure of \S\ref{sec:hash-merging}, with precision $\epsilon=10^{-5}$ to identify zero rows in line~\ref{line:dependenceCheck} of Algorithm~\ref{alg:constructionDD}, and precision $\epsilon'=10^{-3}\sqrt{m'}$ for the call to \texttt{merge} in line~\ref{line:mergeCall}. Recall that, by the results of \S\ref{sec:lowRank}, the size of the diagrams is guaranteed to be $\mathcal{O}(n^{m'})$, and can be smaller if sparsity of $\bm{F'}$ can be exploited.
	
	\subsubsection*{Results}
	
	All formulations are solved with Gurobi 13 with default settings and a time limit of 600 seconds. Table~\ref{tab:portfolioResults} reports, for each combination of factor sparsity $\kappa\in\{5,10,25,50,\infty\}$, rank $m\in \{1,5,10,25\}$ and cardinality $k\in \{10,20,40\}$, either the time (in seconds) required to solve the instance to optimality or, if the time limit is reached, the end gap reported by the solver. For methods \texttt{DD1} and \texttt{DD2}, the time reported includes the total time to construct the decision diagram and solve the ensuing associated MIQO. For each instance, the best performance among the formulations tested is indicated in bold. 
	
	\begin{table}[!ht]%
		\begin{center}
			\caption{\linespread{0.8}\selectfont Each cell contains the time in seconds (if solved) or the end gap (if not) after running branch-and-bound for 600 seconds for instances with $n=214$ and cardinality $k=10/20/40$. Each row corresponds to the sparsity $\kappa$ of the factor model, and each column to its rank $m$. The best performance for each instance is shown in bold.}
			\label{tab:portfolioResults}
            
			\small\setlength{\tabcolsep}{4pt}
            \renewcommand{\arraystretch}{0.5}
			\begin{tabular}{c c|l l l l}
				\hline
				&&$\bm{m=1}$&$\bm{m=5}$&$\bm{m=10}$&$\bm{m=25}$\\
				\hline
				\multirow{3}{*}{$\bm{\kappa=5}$}&\texttt{PER}&$103/\bm{307}/1\%$&$4\%/13\%/\bm{4\%}$&$\bm{6}/22\%/9\%$&$\bm{25}/\bm{108}/\bm{11\%}$\\
				&\texttt{DD1}&$\bm{3}/462/\bm{516}$&$\bm{393}/\bm{415}/\bm{4\%}$&$2\%/\bm{4\%}/\bm{5\%}$&$29/40\%/25\%$\\
				&\texttt{DD2}&&$35\%/15\%/6\%$&$5\%/21\%/11\%$&$52\%/45\%/25\%$\\
				\hline
				\multirow{3}{*}{$\bm{\kappa=10}$}&\texttt{PER}&$1\%/4\%/1\%$&$557/12\%/\bm{4\%}$&$37\%/22\%/\bm{10\%}$&$\bm{97}/27\%/\bm{11\%}$\\
				&\texttt{DD1}&$\bm{247}/\bm{371}/\bm{548}$&$\bm{351}/\bm{466}/6\%$&$\bm{17\%}/\bm{10\%}/11\%$&$128/\bm{23\%}/18\%$\\
				&\texttt{DD2}&&$555/12\%/5\%$&$42\%/24\%/11\%$&$61\%/33\%/18\%$\\
				\hline
				\multirow{3}{*}{$\bm{\kappa=25}$}&\texttt{PER}&$1\%/4\%/0\%$&$19\%/11\%/1\%$&$31\%/13\%/3\%$&$42\%/19\%/6\%$\\
				&\texttt{DD1}&$1\%/4\%/0\%$&$14\%/11\%/1\%$&$22\%/13\%/3\%$&$46\%/22\%/11\%$\\
				&\texttt{DD2}&&$\bm{27}/\bm{40}/\bm{140}$&$\bm{50}/\bm{71}/\bm{95}$&$\bm{103}/\bm{210}/\bm{238}$\\
				\hline
				\multirow{3}{*}{$\bm{\kappa=50}$}&\texttt{PER}&$\bm{269}/4\%/0\%$&$\bm{11\%}/5\%/2\%$&$\bm{31\%}/\bm{16\%}/\bm{3\%}$&$\bm{4}/14\%/2\%$\\
				&\texttt{DD1}&$2\%/4\%/\bm{536}$&$21\%/\bm{4\%}/\bm{1\%}$&$38\%/\bm{16\%}/4\%$&$27\%/\bm{13}\%/\bm{1\%}$\\
				&\texttt{DD2}&&$23\%/10\%/3\%$&$39\%/18\%/6\%$&$39\%/17\%/5\%$\\
				\hline
				\multirow{3}{*}{$\bm{\kappa=\infty}$}&\texttt{PER}&$1\%/\bm{3}/\bm{3}$&$3\%/12\%/7\%$&$\bm{10}/\bm{287}/13\%$&$\bm{87}/\bm{8}/\bm{428}$\\
				&\texttt{DD1}&$\bm{42}/127/80$&$\bm{77}/\bm{144}/\bm{105}$&$414/24\%/\bm{2\%}$&$9\%/42\%/24\%$\\
				&\texttt{DD2}&&$28\%/77\%/98\%$&$84\%/90\%/99\%$&$96\%/98\%/99\%$\\
				\hline
			\end{tabular}
            
		\end{center}
	\end{table}
	
	Several patterns emerge from Table~\ref{tab:portfolioResults}. First, consider the rank-one setting $m=1$, in which formulation \texttt{DD1} convexifies the complete coupling term: \texttt{DD1} solves twice as many instances to optimality as \texttt{PER} (10 versus 5 of the 15 instances with $m=1$), often dramatically faster: for example, with dense factors ($\kappa=\infty$), \texttt{DD1} solves all three instances within 130 seconds, while \texttt{PER} fails to close the gap for $k=10$. Second, when the factors are sparse and the rank is moderate ($\kappa=25$ and $m\geq 5$), formulation \texttt{DD2} is the only method able to solve the instances to optimality, doing so in under 150 seconds, while both \texttt{PER} and \texttt{DD1} terminate with gaps of up to $46\%$: in this regime, convexifying two factors jointly captures a substantial portion of the risk while the sparsity of the factors keeps the associated formulations small. This advantage does not extend to $\kappa=50$, where most formulations fail to solve the instances with $m\geq 5$. Third, with fully dense factors ($\kappa=\infty$) and $m\geq 5$, formulation \texttt{DD2} is no longer viable: the rank-two diagrams associated with dense factors are too large, and the overhead of the resulting extended formulations outweighs the strengthening; \texttt{DD1}, in contrast, remains competitive, attaining the best performance in 5 of the 9 instances with $\kappa=\infty$ and $m\leq 10$. Finally, as the rank $m$ grows, the fraction of the risk convexified by \texttt{DD1} and \texttt{DD2} shrinks and their advantage narrows: for $m=25$, the perspective formulation attains the best performance in most of the instances.
	
	Additional information on parameter tuning is provided in the online companion (\S\ref{sec:detailed-portfolio}).

	\subsection{Sparse exponential smoothing with trend}\label{sec:sparseHolt}
	
	We now consider a sparse version of Holt's double exponential smoothing.
	
	\subsubsection{Description}
	
	Let $\{y_t\}_{t=1}^T$ be a univariate time series. Let $\alpha,\beta\in (0,1)$ be fixed smoothing parameters, let $s_t^\ell$ denote the filtered level after observing $y_t$, and let $s_t^b$ denote the corresponding trend. The model reads	\small\begin{subequations}\label{eq:sparseHolt}
		\begin{align}
			\min_{\bm{s},\bm{x},\bm{z}}\quad&
			\sum_{t=1}^{T-1}
			\left(y_{t+1}-s_t^\ell-s_t^b\right)^2
			+\lambda_0\sum_{t=1}^{T-2}\left(z^\ell_t+z^b_t\right)
			\label{eq:sparseHoltObjective}\\
			\text{s.t.}\quad&
			s_{t+1}^\ell=\alpha y_{t+1}+(1-\alpha)(s_t^\ell+s_t^b)+x_t^\ell
			\qquad \forall t\in[T-2]\label{eq:sparseHolt_level}\\
			\quad&
			s_{t+1}^b=\beta (s_{t+1}^\ell-s_t^\ell)+(1-\beta)s_t^b+x_t^b
			\qquad \forall t\in[T-2]\label{eq:sparseHolt_trend}\\
			&\sum_{t=1}^{T-2}\left(z^\ell_t+z^b_t\right)\leq k\label{eq:sparseHoltCardinality}\\
			&x^\ell_t(1-z^\ell_t)=0\quad
			x^b_t(1-z^b_t)=0
			\qquad \forall t\in[T-2]\label{eq:sparseHoltIndicators}\\
			&\{\bm{s_t}\in\R^2\}_{t=1}^{T-1},\quad
			\{\bm{x_t}\in\R^2\}_{t=1}^{T-2},\quad
			\{\bm{z_t}\in\{0,1\}^2\}_{t=1}^{T-2},
		\end{align}
	\end{subequations}\normalsize
	where $\bm{s_t}=(s_t^\ell,s_t^b)^\top$, $\bm{x_t}=(x_t^\ell,x_t^b)^\top$ and $\bm{z_t}=(z_t^\ell,z_t^b)^\top$. If $\bm{x}=\bm{0}$, then the dynamics \eqref{eq:sparseHolt_level}-\eqref{eq:sparseHolt_trend} are exactly those for Holt's forecasting method: the level $s_{t+1}^\ell$ is a convex combination of the most recent observation $y_{t+1}$ and level $s_t^\ell+s_t^b$; the trend $s_{t+1}^b$ is a convex combination of the most recent observed change $s_{t+1}^\ell-s_t^\ell$ and the previous estimated trend $s_t^b$. However, the presence of variables $x_t^\ell$ and $x_t^b$ allows for abrupt changes in the estimated levels and trends, while constraints \eqref{eq:sparseHoltCardinality}-\eqref{eq:sparseHoltIndicators} and penalty $\lambda_0\sum_{t=1}^{T-2}(z_t^\ell+z_t^b)$ ensure that these changes occur sparingly. The objective term $\sum_{t=1}^{T-1}
	\left(y_{t+1}-s_t^\ell-s_t^b\right)^2$ seeks the choices that result in the best predictions overall.   
	
	\paragraph{Projection}
	
	Substituting $s_{t+1}^\ell=\alpha y_{t+1}+(1-\alpha)(s_t^\ell+s_t^b)+x_t^\ell$ in \eqref{eq:sparseHolt_trend}, we find that dynamics \eqref{eq:sparseHolt_level}-\eqref{eq:sparseHolt_trend} are equivalent to	{\small\begin{equation}\label{eq:sparseHoltStateSpace}
			\begin{pmatrix}s_{t+1}^\ell\\ s_{t+1}^b\end{pmatrix}
			=
			\underbrace{\begin{pmatrix}
					1-\alpha&1-\alpha\\
					-\alpha\beta&1-\alpha\beta
			\end{pmatrix}}_{\bm{A}} 
			\begin{pmatrix}s_{t}^\ell\\ s_{t}^b\end{pmatrix}
			+
			\underbrace{\begin{pmatrix}
					1&0\\
					\beta&1
			\end{pmatrix}}_{\bm{G}}
			\begin{pmatrix}x_{t}^\ell\\ x_{t}^b\end{pmatrix}
			+
			\underbrace{\begin{pmatrix}
					\alpha\\ \alpha\beta
			\end{pmatrix}}_{\bm{c}}y_{t+1}
			\quad\forall t\in[T-2].
	\end{equation}}
	
	Letting $\bm{x_0}=\bm{s_1}$ and applying \eqref{eq:sparseHoltStateSpace} recursively, we find that $\bm{s_{t+1}}=\bm{A}^{t}\bm{x_0}+\sum_{i=1}^t\bm{A}^{t-i}\bm{G}\bm{x_i}+\sum_{i=1}^t\bm{A}^{t-i}\bm{c}y_{i+1},$
	and thus all variables $\bm{s}$ can be projected out, resulting in an MIQO with indicators of the form \eqref{eq:optimizationX}. Although variables $\bm{x_0}$ are not controlled by indicators, they can be incorporated easily by introducing associated variables $\bm{z^0}\in \{0,1\}^2$ and constraints $\bm{z^0}=\bm{1}$. %
	
	\paragraph{Connections with theory and previous work}
Problem \eqref{eq:sparseHolt} can be interpreted as a partially observed hybrid model predictive control model: the state $\bm{s_t}$ follows linear dynamics, the sparse corrections $\bm{x_t}$ act as control inputs subject to on/off decisions $\bm{z_t}$, and the objective penalizes one-step prediction errors and activation costs \citep{bemporad1999control,borrelli2017predictive}. It differs from the inverse-sparse setting studied by \citet{lee2024convexification} in two crucial respects, each of which renders those methods inapplicable. First, the quadratic term is rank-deficient and not invertible: it has rank $T-1$ while there are $2(T-1)$ continuous variables $\bm{x}$.
Second, the block-tridiagonal inverse structure studied by \citet{lee2024convexification} would arise here only after two modeling changes: replacing the prediction-error objective by a separable state-tracking objective $\sum_{t=1}^{T-1}\|\bm{s_t}-\bm{\bar s_t}\|_2^2$, and using a single indicator $z_t$ per period to control both $x_t^\ell$ and $x_t^b$. After projecting out the state variables under those assumptions, the resulting quadratic matrix $\bm{Q}$ has a block-tridiagonal inverse and the problem is polynomial-time solvable by dynamic programming. In contrast, model \eqref{eq:sparseHolt} retains the prediction-error objective and uses two indicators per time period.

	\subsubsection{Experiments}
	
	We now report experiments with real data. 

	\paragraph{Data description}
	We use the seasonally adjusted U.S. Initial Claims series identified by
	\texttt{ICSA} in the Federal Reserve Bank of St. Louis FRED database. The
	series measures the number of new claims for unemployment insurance and is
	reported weekly. The series page can be found at	\url{https://fred.stlouisfed.org/series/ICSA}, and CSV can be directly downloaded from \url{https://eco3min.fr/dataset/us-initial-claims.csv}. 
	The raw file contains 3,109 observations, from January 13,
	1967 through August 7, 2026. Each row contains an ISO date and the number of initial claims expressed in thousands. We aggregate data either annually or quarterly by averaging the number of claims across each year/quarter, resulting in two datasets $\{\tilde y_t\}_{t=1}^T$ consisting of either $T=60$ years or $T=239$ quarters.
	Finally, we scale the data as $\bm{y}=\bm{\tilde y}/\|\bm{\tilde y}\|_2$. 
	
	\paragraph{Methods and parameters}
	We set $\lambda_0=0$, $k\in \{5,10,15\}$, and $\alpha=\beta=\bar\alpha$ with $\bar\alpha\in \{0.5,0.7,0.9\}$. We use $M=10$ for all big-M constraints in these experiments. We test the following methods:
	
	\smallskip\noindent$\bullet$ \textbf{The natural formulation (\texttt{NAT})} solved by Gurobi, given by \eqref{eq:sparseHolt} but encoding indicator variables via big-M constraints $-M\bm{z_t}\leq \bm{x_t}\leq M\bm{z_t}$ for $t\in [T-2]$.
	
	\smallskip\noindent$\bullet$ \textbf{The full decision diagram formulations (\texttt{DD:full})}, obtained by projecting out the state variables $\bm{s}$, and using the hash-merging procedure with parameter $10^{-3}$ in line~\ref{line:mergeCall} of Algorithm~\ref{alg:constructionDD}, while using precision $\epsilon=10^{-5}$ to identify zero rows in line~\ref{line:dependenceCheck}. The cardinality constraint is included in the construction of the diagram: the left-hand side of \eqref{eq:sparseHoltCardinality} is the state associated with $Z$. 
	
	\smallskip\noindent$\bullet$ \textbf{The partial decision diagram relaxation (\texttt{DD:}$\tau$)}, obtained from a partial convexification of the problem. This method partitions the time horizon into $h=\left\lceil\frac{T-1}{\tau}\right\rceil$ disjoint blocks of size $\tau$ 
	as 
	$[T-1]=\mathcal I_1\mathbin{\dot\cup}\cdots
	\mathbin{\dot\cup}\mathcal I_h$ where	
	$\mathcal I_\kappa=[1+(\kappa-1)\tau: \kappa\tau ] \text{ for }\kappa\in [h-1],$
	and the last block $\mathcal I_h$ potentially contains fewer than $\tau$ time periods. %
	Let $R_{\kappa}$ with $\kappa\in [h-1]$ denote the convex hull of the set defined by variables $\{\bm{s_t}\in\R^2\}_{t=(\kappa-1)\tau}^{\kappa\tau},\;
		\{\bm{x_t}\in\R^2\}_{t=(\kappa-1)\tau}^{\kappa\tau-1},\;
		\{\bm{z_t}\in\{0,1\}^2\}_{t=(\kappa-1)\tau}^{\kappa\tau-1},\; w_\kappa\in \R,$ and constraints
	\small\begin{align*}
		&w_\kappa\geq \sum_{t\in \mathcal{I_\kappa}}(y_{t+1}-s_t^\ell-s_t^b)^2\\
		&s_{t+1}^\ell=\alpha y_{t+1}+(1-\alpha)(s_t^\ell+s_t^b)+x_t^\ell
		\qquad \forall t\in[(\kappa-1)\tau:\kappa\tau -1]\\
		\quad&
		s_{t+1}^b=\beta (s_{t+1}^\ell-s_t^\ell)+(1-\beta)s_t^b+x_t^b
		\qquad \forall t\in[(\kappa-1)\tau:\kappa\tau -1]\\
		&x^\ell_t(1-z^\ell_t)=0,\quad
		x^b_t(1-z^b_t)=0
		\qquad \forall t\in[(\kappa-1)\tau:\kappa\tau -1]
	\end{align*}\normalsize
	and let $R_h$ be defined analogously (but potentially involving fewer time periods); observe that there are no constraints involving variables $\bm{z}$. Then the relaxation of the partial decision diagram formulation $\texttt{DD:}\tau$ is 
	{\small\begin{subequations}\label{eq:chDecomp}
			\begin{align}
				\min_{\substack{\bm{s},\bm{x},\bm{z},\bm{w}\\\sum_{t=1}^{T-2}\left(z^\ell_t+z^b_t\right)\leq k}}\quad&
				\sum_{\kappa=1}^{h}w_\kappa\\
				\text{s.t.}\quad& \left(\bm{s_{[(\kappa-1)\tau:\kappa\tau ]}},\bm{x_{[(\kappa-1)\tau:\kappa\tau-1]}},\bm{z_{[(\kappa-1)\tau:\kappa\tau-1]}},w_\kappa\right)\in R_\kappa\label{eq:chConstraints}\quad\forall \kappa\in [h],
			\end{align}
	\end{subequations}}
	\noindent where every constraint \eqref{eq:chConstraints} is implemented via the extended SOCP formulation obtained from a decision diagram using precision $10^{-3}$ for the hash-merging procedure and $10^{-5}$ to identify zero rows. %
    Problem \eqref{eq:chDecomp} is a convex relaxation of \eqref{eq:sparseHolt} used to compute lower bounds.

	\paragraph{Results}

	All MIO and SOCP formulations are solved with Gurobi 12 with default settings and a time limit of 600 seconds. Figure~\ref{fig:real-profiles} compares the quality and computational cost of the lower bounds over the 18 real-data instances using empirical distribution profiles; detailed results are reported in the online companion (\S\ref{sec:detailed-sparseHolt}). The decision diagrams from \texttt{DD:20} can be too large in certain instances with low $\bar \alpha$, resulting in several memory and time limits. Nonetheless, \texttt{DD:10} achieves the best performance overall according to Figure~\ref{fig:real-profiles}: while not as effective as \texttt{NAT} or \texttt{DD:full} at proving gaps below 20\%, it dominates the other methods for other thresholds and is substantially faster. %

	\begin{figure}[t]
		\centering
		\includegraphics[width=0.8\textwidth, trim= {0cm 0cm 0cm 1cm},clip]{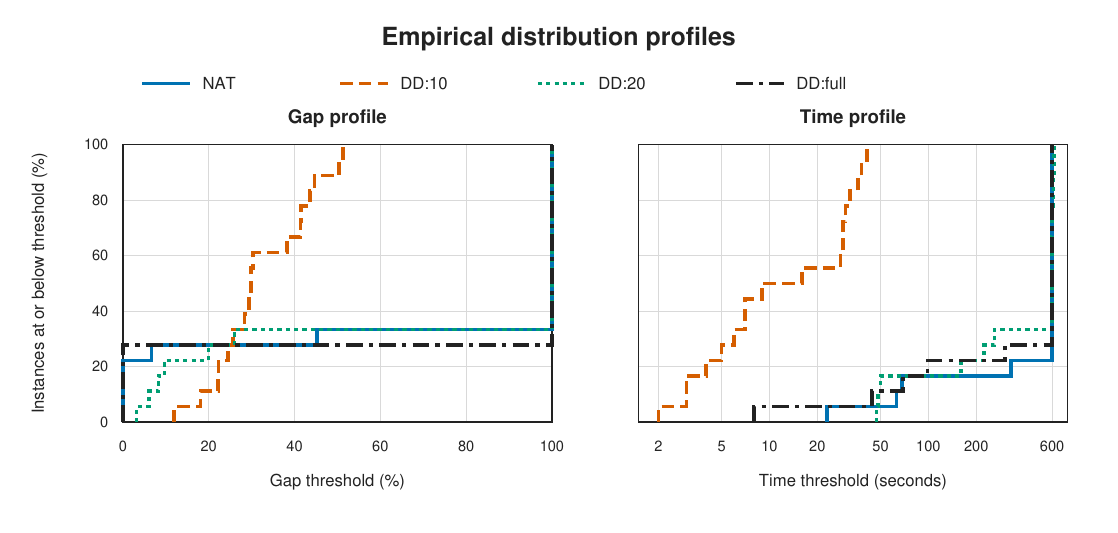}
		\caption{Empirical distribution profiles for lower-bound quality and
			computation time over the real-data instances. Higher and farther left
			is better. Colors and line styles identify the four methods.}
		\label{fig:real-profiles}
	\end{figure}

	\section*{Funding}
Soobin Choi and Andr\'es G\'omez were supported, in part, by grant FA9550-24-1-0086 from the Air Force Office of Scientific Research. Andr\'es G\'omez was additionally supported by grant N000142612117 from the Office of Naval Research. Salar Fattahi was supported, in part, by CAREER grant CCF-2337776 from the National Science Foundation and grant N000142612074 from the Office of Naval Research. Shaoning Han was partially supported by the NUS (Suzhou) Research Institute (Grant No. SZ2025-03). Leonardo Lozano was supported by grant N000142612073 from the Office of Naval Research.
\bibliography{references}

@article{lozano2025,
  title={A single-level reformulation of binary bilevel programs using decision diagrams},
  author={V{\'a}squez, Sebasti{\'a}n and Lozano, Leonardo and van Hoeve, Willem-Jan},
  journal={Mathematical Programming},
  pages={1--54},
  year={2025},
  publisher={Springer}
}

@article{Han2024c,
	
	author = {Han, Shaoning and Gómez, Andrés},
	
	title = {Compact Extended Formulations for Low-Rank Functions with Indicator Variables},
	
	journal = { Mathematics of Operations Research},
	
	year = {2024},
	
	url = {https://pubsonline.informs.org/doi/abs/10.1287/moor.2021.0281},
	
	doi = {https://doi.org/10.1287/moor.2021.0281}
	
}

@article{bhathena2026solving,
  title={Solving Convex Quadratic Optimization with Indicators Over Structured Graphs},
  author={Bhathena, Aaresh and Fattahi, Salar and G{\'o}mez, Andr{\'e}s and K{\"u}{\c{c}}{\"u}kyavuz, Simge},
  journal={arXiv preprint arXiv:2603.02103},
  year={2026}
}

@article{atamturk2025rank,
  title={Rank-one convexification for sparse regression},
  author={Atamturk, Alper and Gomez, Andres},
  journal={Journal of Machine Learning Research},
  volume={26},
  number={35},
  pages={1--50},
  year={2025}
}

@article{bhathena2025parametric,
  title={A parametric approach for solving convex quadratic optimization with indicators over trees},
  author={Bhathena, Aaresh and Fattahi, Salar and G{\'o}mez, Andr{\'e}s and K{\"u}{\c{c}}{\"u}kyavuz, Simge},
  journal={Mathematical Programming},
  year={2025},
  publisher={Springer}
}

@article{gomez2024real,
  title={Real-time solution of quadratic optimization problems with banded matrices and indicator variables},
  author={Gomez, Andres and Han, Shaoning and Lozano, Leonardo},
  journal={arXiv preprint arXiv:2405.03051},
  year={2024}
}

@article{bertsimas2024slowly,
  title={Slowly varying regression under sparsity},
  author={Bertsimas, Dimitris and Digalakis Jr, Vassilis and Li, Michael Lingzhi and Lami, Omar Skali},
  journal={Operations Research},
  year={2024},
  publisher={INFORMS}
}

@ARTICLE{lozano2018Binary,
author="Lozano, Leonardo and Smith, J. Cole",
title="A binary decision diagram based algorithm for solving a class of binary two-stage stochastic programs",
journal="Mathematical Programming",
  year = {2022},
  volume={191},
  pages={381--404}
}

@article{bergman2021dd,
  title={Decision diagram decomposition for quadratically constrained binary optimization},
  author={Bergman, David and Lozano, Leonardo},
  journal={INFORMS Journal on Computing},
  volume={33},
  number={1},
  pages={401--418},
  year={2021},
  publisher={INFORMS}
}

@article{castro2022decision,
  title={Decision diagrams for discrete optimization: A survey of recent advances},
  author={Castro, Margarita P and Cire, Andre A and Beck, J Christopher},
  journal={INFORMS Journal on Computing},
  volume={34},
  number={4},
  pages={2271--2295},
  year={2022},
  publisher={INFORMS}
}

@article{bergman2018discrete,
  title={Discrete nonlinear optimization by state-space decompositions},
  author={Bergman, David and Cire, Andre A},
  journal={Management Science},
  volume={64},
  number={10},
  pages={4700--4720},
  year={2018},
  publisher={Informs}
}

@book{bergman2016,
  title={Decision diagrams for optimization},
  author={Bergman, David and Cire, Andre A and Van Hoeve, Willem-Jan and Hooker, John},
  volume={1},
  year={2016},
  publisher={Springer}
}

@article{castro2020combinatorial,
  title={A combinatorial cut-and-lift procedure with an application to 0-1 second-order conic programming},
  author={Castro, Margarita P and Cire, Andre A and Beck, J Christopher},
  journal={Mathematical Programming},
  year={2021}
}

@article{cire2019network,
  title={A Network-Based Formulation for Scheduling Clinical Rotations},
  author={Cire, Andre A and Diamant, Adam and Yunes, Tallys and Carrasco, Alejandro},
  journal={Production and Operations Management},
  volume={28},
  number={5},
  pages={1186--1205},
  year={2019},
  publisher={Wiley Online Library}
}

@article{castro2020mdd,
  title={An MDD-based Lagrangian approach to the multicommodity pickup-and-delivery TSP},
  author={Castro, Margarita P and Cire, Andre A and Beck, J Christopher},
  journal={INFORMS Journal on Computing},
  volume={32},
  number={2},
  pages={263--278},
  year={2020},
  publisher={Informs}
}

@article{cire2013multivalued,
  title={Multivalued decision diagrams for sequencing problems},
  author={Cire, Andre A and Van Hoeve, Willem-Jan},
  journal={Operations Research},
  volume={61},
  number={6},
  pages={1411--1428},
  year={2013},
  publisher={INFORMS}
}

@article{liu2023graph,
  title={A graph-based decomposition method for convex quadratic optimization with indicators},
  author={Liu, Peijing and Fattahi, Salar and G{\'o}mez, Andr{\'e}s and K{\"u}{\c{c}}{\"u}kyavuz, Simge},
  journal={Mathematical Programming},
  volume={200},
  number={2},
  pages={669--701},
  year={2023},
  publisher={Springer}
}

@article{richard2010lifting,
  title={Lifting inequalities: a framework for generating strong cuts for nonlinear programs},
  author={Richard, Jean-Philippe P and Tawarmalani, Mohit},
  journal={Mathematical Programming},
  volume={121},
  number={1},
  pages={61--104},
  year={2010},
  publisher={Springer}
}

@article{frangioni2006perspective,
  title={Perspective cuts for a class of convex 0--1 mixed integer programs},
  author={Frangioni, Antonio and Gentile, Claudio},
  journal={Mathematical Programming},
  volume={106},
  number={2},
  pages={225--236},
  year={2006},
  publisher={Springer}
}

@article{balas1979disjunctive,
  title={Disjunctive programming},
  author={Balas, Egon},
  journal={Annals of discrete mathematics},
  volume={5},
  pages={3--51},
  year={1979},
  publisher={Elsevier}
}

@article{bemporad1999control,
  title={Control of systems integrating logic, dynamics, and constraints},
  author={Bemporad, Alberto and Morari, Manfred},
  journal={Automatica},
  volume={35},
  number={3},
  pages={407--427},
  year={1999},
  publisher={Elsevier}
}

@article{anstreicher2021quadratic,
  title={Quadratic optimization with switching variables: the convex hull for $n= 2$},
  author={Anstreicher, Kurt M and Burer, Samuel},
  journal={Mathematical Programming},
  volume={188},
  number={2},
  pages={421--441},
  year={2021},
  publisher={Springer}
}

@article{akturk2009strong,
  title={A strong conic quadratic reformulation for machine-job assignment with controllable processing times},
  author={Akt{\"u}rk, M Selim and Atamt{\"u}rk, Alper and G{\"u}rel, Sinan},
  journal={Operations Research Letters},
  volume={37},
  number={3},
  pages={187--191},
  year={2009},
  publisher={Elsevier}
}

@article{atamturk2018strong,
  title={Strong formulations for quadratic optimization with {M}-matrices and indicator variables},
  author={Atamt{\"u}rk, Alper and G{\'o}mez, Andr{\'e}s},
  journal={Mathematical Programming},
  volume={170},
  number={1},
  pages={141--176},
  year={2018},
  publisher={Springer}
}

@article{ceria1999convex,
  title={Convex programming for disjunctive convex optimization},
  author={Ceria, Sebasti{\'a}n and Soares, Jo{\~a}o},
  journal={Mathematical Programming},
  volume={86},
  number={3},
  pages={595--614},
  year={1999},
  publisher={Springer}
}

@inproceedings{dong2013valid,
  title={On valid inequalities for quadratic programming with continuous variables and binary indicators},
  author={Dong, Hongbo and Linderoth, Jeff},
  booktitle={International Conference on Integer Programming and Combinatorial Optimization},
  pages={169--180},
  year={2013},
  organization={Springer}
}

@article{gunluk2010perspective,
  title={Perspective reformulations of mixed integer nonlinear programs with indicator variables},
  author={G{\"u}nl{\"u}k, Oktay and Linderoth, Jeff},
  journal={Mathematical Programming},
  volume={124},
  number={1},
  pages={183--205},
  year={2010},
  publisher={Springer}
}

@article{bienstock1996computational,
  title={Computational study of a family of mixed-integer quadratic programming problems},
  author={Bienstock, Daniel},
  journal={Mathematical programming},
  volume={74},
  number={2},
  pages={121--140},
  year={1996},
  publisher={Springer}
}

@book{borrelli2017predictive,
  title={Predictive Control for Linear and Hybrid systems},
  author={Borrelli, Francesco and Bemporad, Alberto and Morari, Manfred},
  year={2017},
  publisher={Cambridge University Press}
}

@article{lobo1998applications,
  title={Applications of second-order cone programming},
  author={Lobo, Miguel Sousa and Vandenberghe, Lieven and Boyd, Stephen and Lebret, Herv{\'e}},
  journal={Linear algebra and its applications},
  volume={284},
  number={1-3},
  pages={193--228},
  year={1998},
  publisher={Elsevier}
}

@article{han2x2,
  title={2x2-Convexifications for convex quadratic optimization with indicator variables},
  author={Han, Shaoning and G{\'o}mez, Andr{\'e}s and Atamt{\"u}rk, Alper},
  journal={Mathematical Programming},
  volume = {202},
  number = {1},
  pages = {95-134},
  year={2023}
}

@article{gao2011cardinality,
  title={Cardinality constrained linear-quadratic optimal control},
  author={Gao, Jianjun and Li, Duan},
  journal={IEEE Transactions on Automatic Control},
  volume={56},
  number={8},
  pages={1936--1941},
  year={2011},
  publisher={IEEE}
}

@article{wei2022ideal,
  title={Ideal formulations for constrained convex optimization problems with indicator variables},
  author={Wei, Linchuan and G{\'o}mez, Andr{\'e}s and K{\"u}{\c{c}}{\"u}kyavuz, Simge},
  journal={Mathematical Programming},
  volume={192},
  number={1},
  pages={57--88},
  year={2022},
  publisher={Springer}
}

@inproceedings{wei2020convexification,
  title={On the convexification of constrained quadratic optimization problems with indicator variables},
  author={Wei, Linchuan and G{\'o}mez, Andr{\'e}s and K{\"u}{\c{c}}{\"u}kyavuz, Simge},
  booktitle={International conference on integer programming and combinatorial optimization},
  pages={433--447},
  year={2020},
  organization={Springer}
}

@article{frangioni2020decompositions,
  title={Decompositions of semidefinite matrices and the perspective reformulation of nonseparable quadratic programs},
  author={Frangioni, Antonio and Gentile, Claudio and Hungerford, James},
  journal={Mathematics of Operations Research},
  volume={45},
  number={1},
  pages={15--33},
  year={2020},
  publisher={INFORMS}
}

@article{shafiee2024constrained,
  title={Constrained optimization of rank-one functions with indicator variables},
  author={Shafiee, Soroosh and K{\i}l{\i}n{\c{c}}-Karzan, Fatma},
  journal={Mathematical Programming},
  pages={1--47},
  year={2024},
  publisher={Springer}
}

@article{lee2024convexification,
  title={Convexification of multi-period quadratic programs with indicators},
  author={Lee, Jisun and G{\'o}mez, Andr{\'e}s and Atamt{\"u}rk, Alper},
  journal={arXiv preprint arXiv:2412.17178},
  year={2024}
}

@book{garey1979computers,
  title={Computers and intractability},
  author={Garey, Michael R and Johnson, David S},
  year={1979},
  publisher={wh freeman New York}
}

@article{wei2024convex,
  title={On the convex hull of convex quadratic optimization problems with indicators},
  author={Wei, Linchuan and Atamt{\"u}rk, Alper and G{\'o}mez, Andr{\'e}s and K{\"u}{\c{c}}{\"u}kyavuz, Simge},
  journal={Mathematical Programming},
  volume={204},
  number={1},
  pages={703--737},
  year={2024},
  publisher={Springer}
}

@article{markowitz1952portfolio,
  title={Portfolio selection},
  author={Markowitz, Harry},
  journal={The Journal of Finance},
  volume={7},
  number={1},
  pages={77--91},
  year={1952}
}

@article{sharpe1963simplified,
  title={A simplified model for portfolio analysis},
  author={Sharpe, William F},
  journal={Management Science},
  volume={9},
  number={2},
  pages={277--293},
  year={1963}
}

@article{natarajan1995sparse,
  title={Sparse approximate solutions to linear systems},
  author={Natarajan, Balas Kausik},
  journal={SIAM Journal on Computing},
  volume={24},
  number={2},
  pages={227--234},
  year={1995}
}

@article{liu2025polyhedral,
  title={Polyhedral analysis of quadratic optimization problems with {S}tieltjes matrices and indicators},
  author={Liu, Peijing and Atamt{\"u}rk, Alper and G{\'o}mez, Andr{\'e}s and K{\"u}{\c{c}}{\"u}kyavuz, Simge},
  journal={Mathematical Programming},
  year={2025},
  note={In press, doi:10.1007/s10107-025-02272-7}
}

@article{derosa2024explicit,
  title={Explicit convex hull description of bivariate quadratic sets with indicator variables},
  author={De Rosa, Antonio and Khajavirad, Aida},
  journal={Mathematical Programming},
  year={2024},
  note={doi:10.1007/s10107-024-02173-1}
}

@book{golub2013matrix,
  title={Matrix Computations},
  author={Golub, Gene H and Van Loan, Charles F},
  edition={4th},
  publisher={Johns Hopkins University Press},
  year={2013}
}

@article{bergman2016discrete,
  title={Discrete optimization with decision diagrams},
  author={Bergman, David and Cire, Andre A and van Hoeve, Willem-Jan and Hooker, John N},
  journal={INFORMS Journal on Computing},
  volume={28},
  number={1},
  pages={47--66},
  year={2016}
}

@article{vanhoeve2022graph,
  title={Graph coloring with decision diagrams},
  author={van Hoeve, Willem-Jan},
  journal={Mathematical Programming},
  volume={192},
  pages={631--674},
  year={2022}
}

@article{yuan2013truncated,
  title={Truncated Power Method for Sparse Eigenvalue Problems.},
  author={Yuan, Xiao-Tong and Zhang, Tong},
  journal={Journal of Machine Learning Research},
  volume={14},
  number={4},
  year={2013}
}
	\bibliographystyle{plainnat}

\clearpage
\begin{center}
{\Large\bfseries Online Companion to ``Convexification of Mixed-Integer Quadratic Optimization via Decision Diagrams''}
\end{center}
\bigskip

\setcounter{section}{0}
\renewcommand{\thesection}{EC.\arabic{section}}
\renewcommand{\thesubsection}{\thesection.\arabic{subsection}}
\renewcommand{\thesubsubsection}{\thesubsection.\arabic{subsubsection}}
\setcounter{equation}{0}
\renewcommand{\theequation}{EC.\arabic{equation}}
\setcounter{figure}{0}
\renewcommand{\thefigure}{EC.\arabic{figure}}
\setcounter{table}{0}
\renewcommand{\thetable}{EC.\arabic{table}}
\setcounter{algorithm}{0}
\renewcommand{\thealgorithm}{EC.\arabic{algorithm}}
\setcounter{theorem}{0}
\renewcommand{\thetheorem}{EC.\arabic{theorem}}
\setcounter{proposition}{0}
\renewcommand{\theproposition}{EC.\arabic{proposition}}
\setcounter{lemma}{0}
\renewcommand{\thelemma}{EC.\arabic{lemma}}
\setcounter{corollary}{0}
\renewcommand{\thecorollary}{EC.\arabic{corollary}}
\setcounter{fact}{0}
\renewcommand{\thefact}{EC.\arabic{fact}}
\setcounter{assumption}{0}
\renewcommand{\theassumption}{EC.\arabic{assumption}}
\setcounter{definition}{0}
\renewcommand{\thedefinition}{EC.\arabic{definition}}
\setcounter{example}{0}
\renewcommand{\theexample}{EC.\arabic{example}}
\setcounter{remark}{0}
\renewcommand{\theremark}{EC.\arabic{remark}}

\providecommand{\theHsection}{}
\renewcommand{\theHsection}{EC.\arabic{section}}
\providecommand{\theHsubsection}{}
\renewcommand{\theHsubsection}{\theHsection.\arabic{subsection}}
\providecommand{\theHsubsubsection}{}
\renewcommand{\theHsubsubsection}{\theHsubsection.\arabic{subsubsection}}
\providecommand{\theHequation}{}
\renewcommand{\theHequation}{EC.\arabic{equation}}
\providecommand{\theHfigure}{}
\renewcommand{\theHfigure}{EC.\arabic{figure}}
\providecommand{\theHtable}{}
\renewcommand{\theHtable}{EC.\arabic{table}}
\providecommand{\theHalgorithm}{}
\renewcommand{\theHalgorithm}{EC.\arabic{algorithm}}
\providecommand{\theHtheorem}{}
\renewcommand{\theHtheorem}{EC.\arabic{theorem}}
\providecommand{\theHproposition}{}
\renewcommand{\theHproposition}{EC.\arabic{proposition}}
\providecommand{\theHlemma}{}
\renewcommand{\theHlemma}{EC.\arabic{lemma}}
\providecommand{\theHcorollary}{}
\renewcommand{\theHcorollary}{EC.\arabic{corollary}}
\providecommand{\theHfact}{}
\renewcommand{\theHfact}{EC.\arabic{fact}}
\providecommand{\theHassumption}{}
\renewcommand{\theHassumption}{EC.\arabic{assumption}}
\providecommand{\theHdefinition}{}
\renewcommand{\theHdefinition}{EC.\arabic{definition}}
\providecommand{\theHexample}{}
\renewcommand{\theHexample}{EC.\arabic{example}}
\providecommand{\theHremark}{}
\renewcommand{\theHremark}{EC.\arabic{remark}}

\newenvironment{resultrestatement}[1]
{\par\medskip\noindent{\scshape #1.}\ \itshape\ignorespaces}
{\par\nopagebreak[4]\medskip}
\section{A knapsack decision-diagram example}\label{sec:knapsack-example}

This section illustrates the state-transition construction reviewed in \S\ref{sec:ddComb} for a feasible space defined by a knapsack constraint.

\begin{example}[Knapsack case] \label{ex:knapsack}
	Consider $Z=\{\bm{z}\in \{0,1\}^n: \bm{a^\top z}\leq b\}$ for some $\bm{a}\in \Z_+^n$ and $b\in \Z_+$. Define states $\rho^\ell$ in which the state variable records the contribution to the left-hand side at decision stage $\ell$. As a result, the state space can be defined as $S_Z=[0:b]\cup\{\text{inf}\}$ where ``\text{inf}'' denotes an infeasible state. The initial state is $\rho^1=0$. The set of terminal states $T_Z$ includes the infeasible state and every state for which all variables have already been decided. The system transitions between states according to the transition function $\phi_Z: (S_Z \setminus T_Z) \times \{0,1\} \rightarrow S_Z$ defined as $\rho^{\ell+1} = \phi_Z(\rho^{\ell},\bar{z}_{\ell})$, where $\bar{z}_{\ell}$ is the value assigned to variable $z_{\ell}$, and 
	$$\phi_Z(\rho^{\ell},\bar{z}_{\ell})=\begin{cases}
		\text{inf}&\text{if }\rho^\ell+a_\ell \bar{z}_{\ell}>b\\
		\rho^\ell+a_\ell \bar{z}_{\ell}&\text{otherwise.}\end{cases}$$ 
	
	Figure~\ref{fig:DP}(a) presents the state transition graph for a knapsack constraint with $\bm{a^\top} = \begin{pmatrix}3&3&2\end{pmatrix}$ and $b = 5$. The corresponding decision diagram in Figure~\ref{fig:DP}(b) is obtained in this example by removing the infeasible states and merging all equivalent nodes, i.e., nodes that share the same feasible completion set. Note that all points in $Z$ correspond to exactly one $\rootnode$--$4$ path in the diagram and vice versa. Moreover, all partial solutions that share the same feasible completion set correspond to a partial path ending at the same node, e.g., $\begin{pmatrix}0&0\end{pmatrix}$, $\begin{pmatrix}0&1\end{pmatrix}$, and $\begin{pmatrix}1&0\end{pmatrix}$ all share the same completion set $\{0,1\}$ and correspond to partial paths ending at node $3$. 		\hfill $\blacksquare$

	\begin{figure}[ht]
		\centering
		\subfigure[]{
			\centering
			\begin{tikzpicture}
				\filldraw[black] (0,0)  node[anchor=center] {$0$};
				\draw[->,black,thick,dashed] (0,-0.25) -- (-1,-1);
				\draw[->,black,thick] (0,-0.25) -- (1,-1);
				\filldraw[black] (-1,-1.25)  node[anchor=center] {$0$};
				\filldraw[black] (1,-1.25) node[anchor=center] {$3$};
				
				\draw[->,black,thick,dashed] (-1,-1.5) -- (-2,-2.5);
				\draw[->,black,thick] (-1,-1.5) -- (0,-2.5);
				\filldraw[black] (-2.2,-2.75)  node[anchor=center] {$0$};
				
				\draw[->,black,thick] (1,-1.5) -- (1.75,-2.5);
				\draw[->,black,thick,dashed] (1,-1.5) -- (0,-2.5);
				\filldraw[black] (1.75,-2.75) node[anchor=center] {$\text{inf}$};
				\filldraw[black] (0,-2.75) node[anchor=center] {$3$};
				
				\filldraw[black] (-3.3,-4.25) node[anchor=center] {$0$};
				\filldraw[black] (-2,-4.25) node[anchor=center] {$2$};
				\filldraw[black] (-1,-4.25) node[anchor=center] {$3$};
				\filldraw[black] (0,-4.25) node[anchor=center] {$5$};
				
				\draw[->,black,thick,dashed] (-2.25,-3) -- (-3.2,-4);
				\draw[->,black,thick] (-2.25,-3) -- (-2,-4);
				\draw[->,black,thick,dashed] (0,-3) -- (-1,-4);
				\draw[->,black,thick] (0,-3) -- (0,-4);
				\filldraw[black] (-4.5,0) node[anchor=center] {$z_1$};
				\filldraw[black] (-4.5,-1.25) node[anchor=center] {$z_2$};
				\filldraw[black] (-4.5,-2.75) node[anchor=center] {$z_3$};
			\end{tikzpicture}
		}
		\hspace{3em}
		\subfigure[]{
			\centering
			\begin{tikzpicture}     
				\node[main node] (r) at (3,-1) {$\;\rootnode\;$};
				\node[main node] (u1)  at (2,-2)  {$1$};
				\node[main node] (u2)  at (4,-2)  {$2$};
				\node[main node] (u3)  at (3,-3.5) {$3$};        
				\node[main node] (t) at (3,-5) {$4$};
				
				\path[every node/.style={font=\sffamily\small}]
				(r) 
				edge[zero arc] node [left, arc text] {} (u1)
				edge[one arc] node [right, arc text] {} (u2)
				(u1) 
				edge[one arc, bend left=20] node [right, arc text] {} (u3)
				edge[zero arc, bend right=20] node [left, arc text] {} (u3)
				(u2)                
				edge[zero arc] node [right, arc text] {} (u3)
				(u3)                                
				edge[zero arc, bend right=20] node [left, arc text] {} (t) 
				edge[one arc, bend left=20] node [right, arc text] {} (t);
			\end{tikzpicture} 
		}
		\caption{State transition graph and corresponding decision diagram for a feasible space defined by knapsack constraint $3z_1+3z_2+2z_3 \leq 5$.} 
		\label{fig:DP} 
	\end{figure}
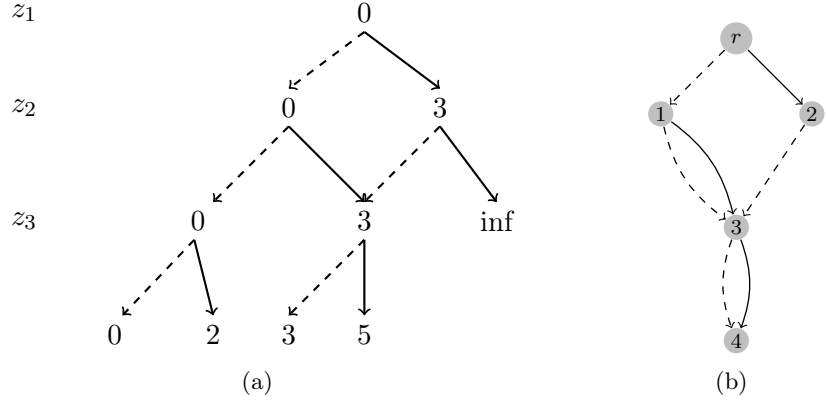
	
\end{example}

\section{Additional Technical Improvements}\label{sec:additional-improvements}

\subsection{Hash Merging}\label{sec:hash-merging}

Algorithm~\ref{alg:mergeHash} describes a third implementation based on hash maps, which we call \underline{hash merging}; it requires a single pass over the nodes and is suitable for practical use. Hash merging works as follows: each node is assigned the key $K=\left(\rho,\lfloor \bm{H}\rceil_{\bar\epsilon}\right)$, where $\bar\epsilon\defeq\epsilon/\sqrt{m}$, $\lfloor \bm{H}\rceil_{\bar\epsilon}\defeq \bm{H}$ if $\epsilon=0$, and $\lfloor \bm{H}\rceil_{\bar\epsilon}$ is the entrywise rounding of $\bm{H}/\bar\epsilon$ to the nearest integer if $\epsilon>0$. The first node encountered with a given key is designated as the representative of that key; every subsequent node with the same key is deleted, and the arcs pointing to it are redirected to the representative.
\begin{algorithm}[H]
	\singlespacing
	\caption{Implementation of $\texttt{merge}(L,A,\epsilon)$ using hash maps}
	\label{alg:mergeHash}
	\begin{algorithmic}[1]
		\Require List of nodes $L$ of a layer, list of arcs $A$, tolerance $\epsilon\geq 0$.
		\Ensure Merged list of nodes $L$, with the arcs of $A$ redirected accordingly.
		
		\State $M\gets \emptyset$ \Comment{Hash map from keys to representative nodes, initially empty}
		
		\State $r\gets \emptyset$ \Comment{Hash map from deleted nodes to representatives, initially empty}
		
		\For{\textbf{each} $(i,\rho,\bm{H})\in L$}
		
		\State $K\gets \left(\rho,\lfloor \bm{H}\rceil_{\bar\epsilon}\right)$ \Comment{Key associated with the node, with $\bar\epsilon=\epsilon/\sqrt{m}$}
		
		\If{$M$ contains key $K$}
		
		\State $r[i]\gets M[K]$, $L\gets L\setminus \left\{(i,\rho,\bm{H})\right\}$ \Comment{Merges node $i$ into $M[K]$}
		
		\Else
		
		\State $M[K]\gets i$ \Comment{Node $i$ is the representative of key $K$}
		
		\EndIf
		
		\EndFor
		
		\For{\textbf{each} $(i,j,\ell,\bar z,\bm{u})\in A$ \textbf{such that} $r[j]$ is defined}
		
		\State Replace arc with $(i,r[j],\ell,\bar z,\bm{u})$ \Comment{Redirects arc to $r[j]$}
		
		\EndFor
	\end{algorithmic}
\end{algorithm}

Algorithm~\ref{alg:mergeHash} should be interpreted as a conservative approximation of \texttt{maximal merging}: it performs a subset of the merges allowed by the specification, but never an inadmissible one. Indeed, two nodes assigned the same key satisfy $\rho_1=\rho_2$ and $\max_{i,j}\left|(\bm{H_1})_{ij}-(\bm{H_2})_{ij}\right|<\bar\epsilon$ (or $\bm{H_1}=\bm{H_2}$ if $\epsilon=0$); since each row has $m$ entries, this implies $\norm{\bm{H_1}-\bm{H_2}}_{2,\infty}< \sqrt{m}\,\bar\epsilon=\epsilon$, thus all merges performed by Algorithm~\ref{alg:mergeHash} are admissible. The approximation is due to two distinct effects, both of which may only prevent admissible merges. First, the quantization operates entrywise: two states satisfying $\norm{\bm{H_1}-\bm{H_2}}_{2,\infty}\leq \epsilon$ may differ by more than $\bar\epsilon$ in some entry, and thus receive different keys. Second, even states that are entrywise within $\bar\epsilon$ of each other may fall in adjacent cells of the quantization grid and receive different keys. In both cases the nodes simply remain unmerged: the hash-based implementation thus trades some merging opportunities for speed, potentially increasing the size of the resulting diagram but never compromising its validity; if $\epsilon=0$, no approximation is incurred and exactly the nodes with identical states are merged. Since computing and hashing a key requires $\mathcal{O}\left((n-\ell)m\right)$ operations and hash map queries take constant expected time, Algorithm~\ref{alg:mergeHash} runs in expected time linear in the number of nodes and arcs of the layer.

\subsection{Further Implementation Improvements}

We discuss three additional improvements to Algorithm~\ref{alg:constructionDD} that can enhance its practical performance.

\smallskip\noindent\emph{Full row rank.} If $\bm{F}$ has full row rank---equivalently, if $\bm{Q}=\bm{FF^\top}\succ\bm{0}$---then no first row $\bm{u}$ ever vanishes. In this case, the condition $\|\bm{u}\|_2>\epsilon$ in line~\ref{line:dependenceCheck} can be replaced with the condition $\bar z=1$ alone. Doing so removes the risk of falsely detecting linear dependence, that is, of incorrectly treating a nearly dependent row as the zero row when $\bm{F}$ is ill-conditioned. 

\smallskip\noindent\emph{Re-normalization.} Instead of using directly the state matrix, a canonical representative of each state obtained by normalizing each row can be used instead. This canonical representation potentially increases the number of nodes that can be merged and improves the numerical behavior of the method. Indeed, rescaling a row of a state matrix by a nonzero constant affects neither the transition vectors nor the arcs generated: the projection $\bm{I}-\bm{uu^\top}/\|\bm{u}\|_2^2$ in \eqref{eq:transition} and the normalized transition vectors of Definition~\ref{def:exactDD} are invariant to rescalings of $\bm{u}$, and each row of the state matrix evolves independently of the scale of the other rows. Nodes whose state matrices coincide after normalization thus generate identical sub-diagrams and identical arc information, and merging them is valid even when their unnormalized states differ. Moreover, since row norms shrink after each projection, renormalizing at every iteration prevents the loss of precision associated with rows of vanishing norm.

To implement this change, it suffices to normalize every row of $\bm{F}$ at the start of the algorithm, renormalize rows of $\bm{G}$ after operation $\bm{G}\gets \bm{R}(\bm{I}-\bm{u_au_a^\top})$ in line~\ref{line:GSStep} (after checking for zero rows), and delete the normalization step in line~\ref{line:normalization}. The drawback of this operation, however, is that additional normalizations are performed, increasing the cost of each operation in line~\ref{line:GSStep}.

\smallskip\noindent\emph{Exploiting sparsity.} Consider for example a tridiagonal matrix $\bm{FF^\top}$ as discussed in Example~\ref{ex:tridiagonal}, where $\bm{F}$ preserves the tridiagonal structure. It is easy to show that, since the support of rows $\bm{F_{\{i\}}}$ and $\bm{F_{\{i+j\}}}$ do not overlap for $j\geq 2$, the Gram-Schmidt operation in line~\ref{line:GSStep} only affects row $\ell+1$ of the state matrix. As a consequence, at level $\ell$ there is no need to store rows $\ell+j$ for $j\geq 2$. The net effect of this change is to reduce memory requirements, reduce the cost of matrix multiplications in line~\ref{line:GSStep}, and decrease the cost of computing and comparing keys in Algorithm~\ref{alg:mergeHash}. The same arguments can be used for the case when $\bm{FF^\top}$ is banded (discarding rows corresponding to $\ell+j$ where $j$ is greater than the bandwidth) and, as we show in \S\ref{sec:tree}, similar arguments apply in even more general cases.
Additionally, if matrix $\bm{F}$ is sparse and structured (e.g., banded), each state matrix will also be sparse. In this case, using libraries for sparse matrix multiplications can significantly reduce the number of operations needed.

\section{Decision diagrams with exact arithmetic}\label{sec:exact arithmetic}

If $\bm{F}\in \mathbb{Q}^{n\times m}$, then the decision diagram in Algorithm~\ref{alg:constructionDD} can be constructed exactly over the rationals. To do so, the normalization of the transition vectors, which may introduce irrational numbers, is avoided: the arcs store the unnormalized vectors $\bm{u_a}\gets\bm{u}$, the Gram-Schmidt step is carried out as $\bm{G}\gets\bm{R}\left(\bm{I}-\bm{uu^\top}/\|\bm{u}\|_2^2\right)$, which involves rational operations only, and the condition $\|\bm{u}\|_2>\epsilon$ is replaced with $\bm{u}\neq\bm{0}$. All states are then rational, and \texttt{merge} with $\epsilon=0$ reduces to testing equality of rational matrices. The normalized transition vectors of Definition~\ref{def:exactDD} never need to be formed explicitly: it suffices to replace the arc lengths of Proposition~\ref{prop:pathLength} with $l_a=c_{\ell(a)}\nu_a-\frac{1}{4}\left(\bm{\delta^\top u_a}\right)^2/\|\bm{u_a}\|_2^2$, which are rational as well.

Moreover, the SOCP relaxations involving in \S\ref{sec:convexification} can also be adapted to involve rational coefficients only. Indeed, if $\bm{F}\in \mathbb{Q}^{n\times m}$ and that the decision diagram is constructed using exact rational arithmetic, the arcs then store rational, unnormalized transition vectors $\bm{u_a}$, whereas the vectors appearing in Definition~\ref{def:socp} are the possibly irrational unit vectors $\bm{u_a}/\|\bm{u_a}\|_2$. A representation of $R_{\bm{F},Z}$ involving rational data only can nonetheless be recovered from \eqref{eq:convexhull}. First, fix $w_a=0$ for every arc with $\bm{u_a}=\bm{0}$. Then, for the remaining arcs, perform the change of variables $\tilde w_a\gets w_a/\|\bm{u_a}\|_2$: constraint \eqref{eq:convexhull_x} becomes
$\bm{F^\top x}=\sum_{a\in A}\bm{u_a}\tilde w_a$
and constraint \eqref{eq:convexhull_nonlinear} becomes
$x_0\geq \sum_{a\in A}\|\bm{u_a}\|_2^2\,\tilde w_a^2/r_a$,
where $\bm{u_a}\in \mathbb{Q}^m$ and $\|\bm{u_a}\|_2^2\in \mathbb{Q}$, and the remaining constraints \eqref{eq:convexhull_z}-\eqref{eq:convexhull_nonneg} have $0/1$ coefficients. Therefore, all coefficients of the resulting ideal extended formulation are rational.

\section{Detailed Computational Results}\label{sec:detailed-computational-results}
This section supplements the computational study in \S\ref{sec:realData} with detailed results and parameter-tuning information.

\subsection{Portfolio Index Tracking}\label{sec:detailed-portfolio}

\subsubsection{Quality of the Factor Approximation}\label{sec:portfolio-approx-quality}

Table~\ref{tab:portfolioApprox} reports the quality of the approximation $\bm{\bar\Sigma}\approx \bm{Q}+\Diag(\bm{d})$ as a function of $m$ and $\kappa$. With dense factors ($\kappa=\infty$), a handful of factors suffice to approximate the sample covariance almost exactly, with errors below $0.5\%$ for $m\geq 5$. Under sparsity constraints the errors are naturally larger, decreasing from approximately $9\%$ to $6\%$ as $m$ and $\kappa$ increase; in all cases, the share of the total variance attributed to the specific terms $\bm{d}$ is below $1\%$, indicating that the low-rank part carries the bulk of the risk.

\begin{table}[!ht]
	\begin{center}
		\caption{Quality of the approximation $\bm{\bar\Sigma}\approx \bm{Q}+\Diag(\bm{d})$, where $\bm{Q}$ has rank $m$ and factor sparsity $\kappa$. Errors are computed as $\text{error}=\frac{\|\bm{\bar\Sigma}-\bm{Q}-\Diag(\bm{d})\|_2^2}{\|\bm{\bar\Sigma}\|_2^2}$, and the diagonal dominance of the resulting matrix is $\text{diag}=\frac{\sum_{i\in [n]}d_i^2}{\sum_{i\in [n]}(Q_{ii}+d_i)^2}$.}
		\label{tab:portfolioApprox}
		\renewcommand{\arraystretch}{0.5}
		\begin{tabular}{c |c c c c}
			\hline
			&$\bm{m=1}$&$\bm{m=5}$&$\bm{m=10}$&$\bm{m=25}$\\
			\hline
			\multirow{2}{*}{$\bm{\kappa=5}$}&error=$9.36\%$&error=$9.46\%$&error=$9.17\%$&error=$8.89\%$\\
			&diag=$0.69\%$&diag=$0.29\%$&diag=$0.25\%$&diag=$0.22\%$\\
			\hline
			\multirow{2}{*}{$\bm{\kappa=10}$}&error=$9.36\%$&error=$9.48\%$&error=$8.84\%$&error=$8.41\%$\\
			&diag=$0.69\%$&diag=$0.29\%$&diag=$0.25\%$&diag=$0.22\%$\\
			\hline
			\multirow{2}{*}{$\bm{\kappa=25}$}&error=$9.35\%$&error=$9.50\%$&error=$7.81\%$&error=$7.54\%$\\
			&diag=$0.69\%$&diag=$0.29\%$&diag=$0.24\%$&diag=$0.22\%$\\
			\hline
			\multirow{2}{*}{$\bm{\kappa=50}$}&error=$9.35\%$&error=$8.63\%$&error=$6.67\%$&error=$5.94\%$\\
			&diag=$0.69\%$&diag=$0.29\%$&diag=$0.24\%$&diag=$0.17\%$\\
			\hline
			\multirow{2}{*}{$\bm{\kappa=\infty}$}&error=$9.34\%$&error=$0.47\%$&error=$0.12\%$&error=$0.04\%$\\
			&diag=$0.69\%$&diag=$0.21\%$&diag=$0.18\%$&diag=$0.15\%$\\
			\hline
		\end{tabular}
	\end{center}
\end{table}

\subsubsection{Parameter Tuning}

Most of the parameters used in \S\ref{sec:portfolioComp} (choices of $k$, $\kappa$, $m$, $m'$, $\epsilon$, $\epsilon'$, and time limits) were decided before carrying out the computations and not revisited afterward, with the following exceptions. 

\noindent $\bullet$ We also tested target sparsity of the portfolio $k=5$, but the resulting instances were too easy so we do not report this setting. Indeed, note that the $\binom{214}{5}\approx 3.5\cdot 10^9$ potential portfolios are almost within the reach of pure enumeration methods.

\noindent $\bullet$ We considered using \texttt{DD3} (that is, decision diagram convexification with $m'=3$), but quickly realized that the resulting decision diagrams were too large for effective use within Gurobi's branch-and-bound framework. After initial experimentation, we decided against testing this setting for all parameter combinations.

\noindent $\bullet$ Prior to testing the methods on the index-tracking problems \eqref{eq:indexTracking}, we also applied them to the standard long-only cardinality-constrained mean--variance portfolio model, which trades expected return against portfolio variance under budget and support constraints \citep{markowitz1952portfolio,bienstock1996computational,gao2011cardinality}. Using the same dataset, the root gaps of the perspective formulation are already very small (typically less than 5\%), and method \texttt{PER} solves all instances within a few seconds. In such ``easy'' instances, the more sophisticated \texttt{DD1} and \texttt{DD2} are unnecessarily complex and do not result in any practical benefits.

\subsection{Sparse Exponential Smoothing with Trend}\label{sec:detailed-sparseHolt}

\subsubsection{Run-Level Results}

All MIO and SOCP formulations are solved with Gurobi 12 with default settings and a time limit of 600 seconds. Figure~\ref{fig:real-profiles} compares the quality and computational cost of
the lower bounds over the 18 real-data instances using empirical distribution
profiles, while Table~\ref{tab:realHolt} provides detailed information for each run. It reports, for each combination of smoothing $\bar\alpha\in\{0.5,0.7,0.9\}$, cardinality $k\in \{5,10,15\}$, and choice of block size $\tau\in \{10,20\}$ (for method $\texttt{DD:}\tau$):

\noindent $\bullet$ the time (in seconds) required to run the method, corresponding to the branch-and-bound time for \texttt{NAT}, the time to construct the decision diagram for \texttt{DD:full}, and the time to solve the SOCP continuous relaxation for $\texttt{DD:}\tau$.

\noindent $\bullet$ the lower bound expressed as a gap with respect to the best known upper bound $\text{UB}$ for a given instance, computed as $\text{gap}=(\text{UB}-\text{LB}_{\texttt{method}})/\text{UB}$. The upper bound corresponds to the objective value found by either \texttt{NAT} or \texttt{DD:full} if either method solves to optimality, and to the best solution found by \texttt{NAT} otherwise. The lower bound $\text{LB}_{\texttt{NAT}}$ is the final bound after branch-and-bound, and $\text{LB}_{\texttt{DD:}\tau}$ is the objective value for the SOCP relaxation; for method \texttt{DD:full}, the gap is either 0\% if the method terminates or 100\% otherwise.

\noindent $\bullet$ the best objective value found by a given method; for \texttt{DD:full}, the value is obtained by manually evaluating the solution obtained from the decision diagram (it differs by at most 0.04\% from the value obtained by directly computing the shortest path from the decision diagram). We point out that no numerical issues were detected in any instance involving a decision diagram, that is, the objective values of the solutions obtained are at least as good as those of the best solutions obtained from \texttt{NAT}, and in some cases better.

\noindent $\bullet$ the number of branch-and-bound nodes used by \texttt{NAT}.

\noindent $\bullet$ the number of arcs from the decision diagrams; for $\texttt{DD:}\tau$, this value is the sum of arcs of all decision diagrams obtained from each block.

We observe that \texttt{NAT} is effective for instances with $k=5$; despite root relaxations of 100\% in all cases, it is able to reduce or close the gap effectively during branch-and-bound. However, it fails dramatically in instances with $k\in \{10,15\}$, resulting in all cases in trivial lower bounds of $0$ and gaps of 100\%. As a consequence, it appears in the performance profile in Figure~\ref{fig:real-profiles} (left) as an almost flat line. Decision diagram methods are highly sensitive to the parameter $\bar\alpha$, performing much better for larger values of this parameter. Method \texttt{DD:full} can solve several instances where $\alpha=0.9$, as well as the setting $(\alpha,k)=(0.7,5)$, but fails to construct the full decision diagram in other settings. It has a similar profile in Figure~\ref{fig:real-profiles} (left) as \texttt{NAT}, although the specific instances solved by each method are different.

\begin{landscape}
	\begin{table}[!ht]
		\singlespacing
		\centering
		\caption{Detailed results with real data.}
		\label{tab:realHolt}
		\scriptsize
		\setlength{\tabcolsep}{2pt}
		\renewcommand{\arraystretch}{0.78}
		\resizebox{0.80\linewidth}{!}{%
			\begin{tabular}{c c | c| c c c c c| c c c c c}
				\multirow{2}{*}{$\bm{\bar \alpha}$} &\multirow{2}{*}{$\bm{k}$} &\multirow{2}{*}{\textbf{method}}&\multicolumn{5}{c|}{\underline{\textbf{Annual }$\bm{(T=60)}$}}&\multicolumn{5}{c}{\underline{\textbf{Quarterly }$\bm{(T=239)}$}}\\
				&&&\textbf{time}&\textbf{gap}&\textbf{val}&\textbf{bb\_nodes}&\textbf{dd\_arcs}&\textbf{time}&\textbf{gap}&\textbf{val}&\textbf{bb\_nodes}&\textbf{dd\_arcs}\\
				\hline
				\multirow{14}{*}{0.5}&\multirow{4}{*}{5}&\texttt{NAT}&68&0\%&1.00&102,206&-&600&45\%&1.00&107,346&-\\
				&&\texttt{\texttt{DD:full}}&\multicolumn{5}{l|}{memory limit}&\multicolumn{5}{l}{memory limit}\\
				&&\texttt{DD:10}&2&30\%&-&-&29,134&16&51\%&-&-&124,638\\
				&&\texttt{DD:20}&\multicolumn{5}{l|}{memory limit}&\multicolumn{5}{l}{memory limit}\\
				&\multirow{4}{*}{10}&\texttt{NAT}&600&100\%&1.00&608,356&-&600&100\%&1.00&238,463&-\\
				&&\texttt{\texttt{DD:full}}&\multicolumn{5}{l|}{memory limit}&\multicolumn{5}{l}{memory limit}\\
				&&\texttt{DD:10}&3&41\%&-&-&29,134&29&38\%&-&-&124,638\\
				&&\texttt{DD:20}&\multicolumn{5}{l|}{memory limit}&\multicolumn{5}{l}{memory limit}\\		
				&\multirow{4}{*}{15}&\texttt{NAT}&600&100\%&1.00&166,493&-&600&100\%&1.00&242,768&-\\
				&&\texttt{\texttt{DD:full}}&\multicolumn{5}{l|}{memory limit}&\multicolumn{5}{l}{memory limit}\\
				&&\texttt{DD:10}&5&50\%&-&-&29,134&29&45\%&-&-&124,638\\
				&&\texttt{DD:20}&\multicolumn{5}{l|}{memory limit}&\multicolumn{5}{l}{memory limit}\\	
				
				\hline
				\multirow{14}{*}{0.7}&\multirow{4}{*}{5}&\texttt{NAT}&63&0\%&1.00&39,892&-&600&7\%&1.00&252,224&-\\
				&&\texttt{\texttt{DD:full}}&99&0\%&1.00&-&1,850,663&\multicolumn{5}{l}{memory limit}\\
				&&\texttt{DD:10}&7&25\%&-&-&30,536&38&26\%&-&-&130,582\\
				&&\texttt{DD:20}&\multicolumn{5}{l|}{time limit solving SOCP}&\multicolumn{5}{l}{time limit solving SOCP}\\
				&\multirow{4}{*}{10}&\texttt{NAT}&600&100\%&1.00&330,205&-&600&100\%&1.00&245,125&-\\
				&&\texttt{\texttt{DD:full}}&\multicolumn{5}{l|}{memory limit}&\multicolumn{5}{l}{memory limit}\\
				&&\texttt{DD:10}&9&30\%&-&-&30,353&41&30\%&-&-&130,582\\
				&&\texttt{DD:20}&\multicolumn{5}{l|}{time limit solving SOCP}&\multicolumn{5}{l}{time limit solving SOCP}\\		
				&\multirow{4}{*}{15}&\texttt{NAT}&600&100\%&1.00&489,837&-&600&100\%&1.00&312,328&-\\
				&&\texttt{\texttt{DD:full}}&\multicolumn{5}{l|}{memory limit}&\multicolumn{5}{l}{memory limit}\\
				&&\texttt{DD:10}&7&44\%&-&-&30,536&36&28\%&-&-&130,582\\
				&&\texttt{DD:20}&\multicolumn{5}{l|}{time limit solving SOCP}&\multicolumn{5}{l}{time limit solving SOCP}\\		
				\hline
				\multirow{14}{*}{0.9}&\multirow{4}{*}{5}&\texttt{NAT}&23&0\%&1.00&17,892&-&330&0\%&1.00&104,633&-\\
				&&\texttt{\texttt{DD:full}}&8&0\%&1.00&-&160,982&304&0\%&1.00&-&722,587\\
				&&\texttt{DD:10}&4&12\%&-&-&21,926&30&18\%&-&-&91,876\\
				&&\texttt{DD:20}&50&8\%&-&-&78,200&160&3\%&-&-&317,334\\
				&\multirow{4}{*}{10}&\texttt{NAT}&600&100\%&1.00&433,019&-&600&100\%&1.00&296,202&-\\
				&&\texttt{\texttt{DD:full}}&44&0\%&1.00&-&687,930&\multicolumn{5}{l}{time limit constructing DD}\\
				&&\texttt{DD:10}&6&29\%&-&-&21,926&28&22\%&-&-&91,876\\
				&&\texttt{DD:20}&48&20\%&-&-&78,200&224&6\%&-&-&317,334\\		
				&\multirow{4}{*}{15}&\texttt{NAT}&600&100\%&1.06&601,433&-&600&100\%&1.00&242,861&-\\
				&&\texttt{\texttt{DD:full}}&69&0\%&1.00&-&1,226,189&\multicolumn{5}{l}{memory limit}\\
				&&\texttt{DD:10}&3&42\%&-&-&21,926&32&22\%&-&-&91,876\\
				&&\texttt{DD:20}&47&26\%&-&-&78,200&261&10\%&-&-&317,334\\		
				
			\end{tabular}%
		}
	\end{table}
\end{landscape}

\subsubsection{Parameter Tuning}

We also tested smoothing parameters $\bar\alpha\in \{0.8,0.95\}$, and the results are as expected: the performance of decision diagram methods for $\bar \alpha=0.8$ lies between the performance reported for $\bar \alpha\in \{0.7,0.9\}$, and the setting $\bar\alpha=0.95$ performed even better than the settings reported here. We did not include these results to avoid biasing the results towards instances in which decision diagram methods are better than \texttt{NAT}, in order to have a more balanced discussion.

We tested the setting $\texttt{DD}:\tau$ with $\tau\in \{2,3,5\}$. The gaps in these cases were equal or close to 100\%, so we do not report these results. We also experimented with decision diagrams using the hash-merging procedure of \S\ref{sec:hash-merging} and precision $\epsilon=10^{-2}$ (instead of $10^{-3}$ as reported in the paper). The resulting diagrams are substantially more compact (involving roughly one-tenth as many arcs), but in some instances the solutions obtained from the diagrams were suboptimal due to the coarser aggregation.

Finally, we experimented with solving formulations $\texttt{DD}:\tau$ with branch-and-bound, but the results were not promising. Just processing the root node of the branch-and-bound tree takes, on average, three times as long as solving the SOCP relaxation (the times reported in Table~\ref{tab:realHolt}), due to presolve and other features specific to solving MIOs. Just solving the root node already consumes a significant portion of the 600-second budget, and only a few hundred, or even only a few dozen, branch-and-bound nodes are then explored, which is insufficient to further close the relaxation gap or find good feasible solutions.

\section{Proofs of Formal Results}\label{sec:online-companion-proofs}
This section restates each proposition, theorem, lemma, and corollary immediately before its proof. The restatements retain the numbering used in the main manuscript and introduce no duplicate labels or theorem counters.

\subsection{Lemma~\ref{prop:Gexplicit}}
\begin{resultrestatement}{Lemma~\ref{prop:Gexplicit}}
	Letting $\bm{B^\ell}(\bm{v})$ be any orthonormal basis of $\Row(\bm{F_v})$, function $g^\ell(\bm{v})$ satisfies
		$$g^\ell(\bm{v})=\bm{c_{[\ell-1]}^\top v}-\frac{1}{4}\bm{\delta^\top B^\ell}(\bm{v})^\top \bm{B^\ell}(\bm{v})\bm{\delta}.$$
\end{resultrestatement}

\begin{proof}
	Clearly $\Row(\bm{F_v})=\Row\left(\bm{B^\ell}(\bm{v})\right)$, thus we find from \eqref{eq:defG} that
	\begin{align*}
		g^\ell(\bm{v})=&\bm{c_{[\ell-1]}^\top v}+\min_{\bm{y}\in \Row\left(\bm{B^\ell}(\bm{v})\right)}\;\bm{\delta^\top y}+\left\|\bm{y}\right\|_2^2\\
		=&\bm{c_{[\ell-1]}^\top v}+\min_{\bm{y}\in \R^m,\bm{\alpha}\in \R^q}\;\bm{\delta^\top y}+\left\|\bm{y}\right\|_2^2\text{ s.t. } \bm{y}=\bm{B^\ell}(\bm{v})^\top\bm{\alpha}\\
		=&\bm{c_{[\ell-1]}^\top v}+\min_{\bm{\alpha}\in \R^q}\;\bm{\delta^\top}\bm{B^\ell}(\bm{v})^\top\bm{\alpha}+\left\|\bm{B^\ell}(\bm{v})^\top\bm{\alpha}\right\|_2^2\\
		=&\bm{c_{[\ell-1]}^\top v}+\min_{\bm{\alpha}\in \R^q}\;\bm{\delta^\top}\bm{B^\ell}(\bm{v})^\top\bm{\alpha}+\left\|\bm{\alpha}\right\|_2^2\tag{$\because \bm{B^\ell}(\bm{v})\bm{B^\ell}(\bm{v})^\top=\bm{I}$}.
	\end{align*}
	
	To solve the optimization problem $\min_{\bm{\alpha}\in \R^q}\left\{\bm{\delta^\top B^\ell}(\bm{v})^\top\bm{\alpha}+\|\bm{\alpha}\|_2^2\right\}$, it suffices to set the gradient of the objective function to $\bm{0}$ and solve for $\bm{\alpha}$. In particular, we find that
	\begin{align*}
		\bm{B^\ell}(\bm{v})\bm{\delta}+2\bm{\alpha}=\bm{0}\Leftrightarrow \bm{\alpha}=-\frac{1}{2} \bm{B^\ell}(\bm{v})\bm{\delta}.
	\end{align*}
	Substituting $\bm{\alpha}$ with its optimal value, we find the result.
\end{proof}

\subsection{Proposition~\ref{prop:hDecomposition}}
\begin{resultrestatement}{Proposition~\ref{prop:hDecomposition}}
	For $\ell\in [2:n]$, function $h^\ell$ can be rewritten as
		{\small\begin{equation*}
				\begin{aligned}
					h^\ell(\bm{v})=
					g^\ell(\bm{v})+\min_{\substack{\bm{w}\in\{0,1\}^{n+1-\ell}\\
							\bm{y}\in \R^m}}\;&\bm{c_{[\ell:n]}^\top w}+\bm{\delta^\top y}+\left\| \bm{y}\right\|_2^2\\
					\text{s.t.}\;&\bm{w}\in Z^\ell(\bm{v}),\;\bm{y}\in\Row\left(\Diag(\bm{w})\bm{F_{[\ell:n]}}\bm{G_\perp^\ell}(\bm{v})\right),
				\end{aligned}
		\end{equation*}}
		where $\bm{G_\perp^\ell}(\bm{v})$ and $g^\ell(\bm{v})$ are defined in Definition~\ref{def:proj} and \ref{def:functionG}, respectively.
\end{resultrestatement}

\begin{proof}
	We focus on the set associated with the continuous variables $\bm{y}$ in \eqref{eq:optFixedV}, $$Y^\ell(\bm{v},\bm{w})\defeq \{\bm{y}\in \R^m: \bm{y}\in \Row(\bm{F_{(v,w)}})\}.$$
	We find from Proposition~\ref{prop:GramSchmidt} that $\bm{y}\in Y^\ell(\bm{v},\bm{w})$ if and only if $\bm{y}=\bm{y_1}+\bm{y_2}$ with $\bm{y_1}\in \Row\left(\bm{F_v}\right)$ and $\bm{y_2}\in \Row\left(\bm{F_{(0,w)}G_\perp^\ell}(\bm{v})\right)=\Row\left(\Diag(\bm{w})\bm{F_{[\ell:n]}}\bm{G_\perp^\ell}(\bm{v})\right)$. Moreover, since every row of $\bm{F_{[\ell:n]}}\bm{G_\perp^\ell}$ is orthogonal to $\Row\left(\bm{F_v}\right)$, we find that $\|\bm{y}\|_2^2=\|\bm{y_1}+\bm{y_2}\|_2^2=\|\bm{y_1}\|_2^2+\|\bm{y_2}\|_2^2$ and problem \eqref{eq:optFixedV} reduces to 
	\begin{align*}
		h^\ell(\bm{ v})= 
		\bm{c_{[\ell-1]}^\top v}+\min_{\substack{\bm{w}\in \{0,1\}^{n+1-\ell}\\\bm{y_1},\bm{y_2}\in \R^m}}\;&\bm{c_{[\ell:n]}^\top w}+\bm{\delta^\top y_1}+\bm{\delta^\top y_2}+\left\|\bm{y_1}\right\|_2^2+\left\|\bm{y_2}\right\|_2^2\\
		\text{ s.t. }\;&\bm{y_1}\in \Row\left(\bm{F_v}\right)\\
		&\bm{w}\in Z^\ell(\bm{v}),\;\bm{y_2}\in \Row\left(\Diag(\bm{w})\bm{F_{[\ell:n]}}\bm{G_\perp^\ell}(\bm{v})\right).
	\end{align*}
	Noting that the problem decomposes in variables $\bm{y_1}$ and $(\bm{w},\bm{y_2})$, and that the problem associated with $\bm{y_1}$ is precisely \eqref{eq:defG}, we obtain the desired result.
\end{proof}

\subsection{Proposition~\ref{prop:state}}
\begin{resultrestatement}{Proposition~\ref{prop:state}}
	Given any arc-specified path $(a_1,a_2,\ldots,a_{\ell-1})$ in $\mathcal{G}$ such that $\tail{a_1}=s^1$ and $\head{a_{\ell-1}}=s^\ell$, let $\bm{\bar H^\ell}$ be the matrix stored by state $s^\ell$ and $\bm{v}\in \{0,1\}^{\ell-1}$ be the partial solution represented by the path, that is, $v_i=\nu_{a_i}$. Then,
		\begin{enumerate}
			\item the non-zero transition vectors in $\{\bm{u_{a_1}},\dots, \bm{u_{a_{\ell-1}}}\}$ define an orthonormal basis of $\Row\left(\bm{F_v}\right)$;
			\item the identity $\bm{\bar H^\ell}=\bm{F_{[\ell:n]}}\bm{G_\perp^\ell}(\bm{v})$ holds.
		\end{enumerate}
\end{resultrestatement}

\begin{proof}
	We prove the result by induction on the layer $\ell$.
	
	\paragraph{Base case} If $\ell=1$, then there are no arcs in the path and  $\Row\left(\bm{F_v}\right)=\bm{0}=\text{span}\;\emptyset$, verifying point 1. Since $\bm{\bar H^1}=\bm{F}$, $\bm{G_\perp^1}=\bm{I}$ and by convention, and $\bm{F_{[1:n]}}=\bm{F}$, point 2 is also satisfied. 
	\paragraph{Inductive step} Assume that the result holds for path $(a_1,a_2,\ldots,a_{\ell-1})$: 
	
	\noindent $\bullet$ The matrix $\bm{\bar U^\ell}$ with rows corresponding to non-zero vectors of $\{\bm{u_{a_{i}}}\}_{i=1}^{\ell-1}$ defines an orthonormal basis of $\Row(\bm{F_v})$; thus, by definition of $\bm{G_\perp^\ell}(\bm{v})$, the identity $\bm{G^\ell}(\bm{v})=\bm{I}-(\bm{\bar U^\ell})^\top\bm{\bar U^\ell}=\bm{I}-\sum_{i=1}^{\ell-1}\bm{u_{a_i}u_{a_i}^\top}$ holds. 
	
	\noindent$\bullet$  The matrix stored by state $\head{a_{\ell-1}}$ satisfies $\bm{\bar H^\ell}=\bm{F_{[\ell:n]}}\bm{G_\perp^\ell}(\bm{v})$.
	
	We now prove the result for path $(a_1,a_2,\ldots,a_{\ell-1},a_\ell)$ with associated vector $\begin{pmatrix}\bm{v}\\ \nu_{a_\ell}\end{pmatrix}$, depending on the value of $\nu_{a_\ell}$.
	
	\paragraph{Case $\bm{\nu_{a_\ell}=0}$} In this case $\bm{u_{a_\ell}}=\bm{0}$. Since $\Row\left(\bm{F_v}\right)=\Row\left(\bm{F_{(v,0)}}\right)$ and the non-zero vectors in the path are unchanged, point 1 of the proposition is verified. Moreover \begin{align*}\bm{\bar H^{\ell+1}}=&\bm{\bar H_\downarrow^{\ell}}\\
		=&\left(\bm{F_{[\ell:n]}}\bm{G_\perp^\ell}(\bm{v})\right)_\downarrow\tag{$\because$ induction hypothesis}\\
		=&\bm{F_{[\ell+1:n]}}\bm{G_\perp^\ell}(\bm{v})\\
		=&\bm{F_{[\ell+1:n]}}\bm{G_\perp^{\ell+1}}\begin{pmatrix}\bm{v}\\0\end{pmatrix},
	\end{align*}
	where the last equality holds since $\Row(\bm{F_{v}})=\Row(\bm{F_{(v,0)}})$ (thus  the projection matrix onto the orthogonal complement is unchanged).
	
	\paragraph{Case $\bm{\nu_{a_\ell}=1}$} The results hold since execution of the transition function and adding $\bm{u_{a_\ell}}$ to the basis (if non-zero) corresponds exactly to an iteration of the Gram-Schmidt process (lines~\ref{line:remove}-\ref{line:addBasis} of Algorithm~\ref{alg:MGS}), and the result follows directly from Proposition~\ref{prop:GSInduction}.
\end{proof}

\subsection{Proposition~\ref{prop:pathLength}}
\begin{resultrestatement}{Proposition~\ref{prop:pathLength}}
	Assume that Assumption~\ref{ass:unbounded} holds. Given an exact decision diagram $\DD$, vectors $\bm{c},\bm{d}\in \R^n$ such that problem \eqref{eq:optimizationX} is not unbounded, and vector $\bm{\delta}\in \R^m$ satisfying $\bm{F\delta}=\bm{d}$, define the length of an arc $a\in A$ as $l_a=c_{\ell(a)}\nu_{a}-\frac{1}{4}\left(\bm{\delta^\top u_{a}}\right)^2$. Then the length of any arc-specified path $(a_1,a_2,\ldots,a_{\ell-1})$ in $\mathcal{G}$  is the partial cost $g^\ell(\bm{v})$ defined in Definition~\ref{def:functionG}, where 
		$\bm{v}\in \{0,1\}^{\ell-1}$ is the partial solution represented by the path, that is, the vector satisfying $v_i=\nu_{a_i}$. In particular, any shortest path between the root and a terminal node corresponds to an optimal solution of~\eqref{eq:optimizationX}.
\end{resultrestatement}

\begin{proof}
	The length of path $(a_1,a_2,\ldots,a_{\ell-1})$ is
	\begin{align}
		&\sum_{i\in [\ell-1]}c_{\ell(a_i)}\nu_{a_i}-\frac{1}{4}\sum_{i\in [\ell-1]}\left(\bm{\delta^\top u_{a_i}}\right)^2\notag\\
		=&\sum_{i\in [\ell-1]}c_{i}v_i-\frac{1}{4}\bm{\delta^\top}\left(\sum_{i\in [\ell-1]}\bm{u_{a_i}u_{a_i}^\top}\right)\bm{\delta}\notag\\
		=&\sum_{i\in [\ell-1]}c_{i}v_i-\frac{1}{4}\bm{\delta^\top}\bm{B^\ell}(\bm{v})^\top\bm{B^\ell}(\bm{v})\bm{\delta},\label{eq:gDerivation}
	\end{align}
	where $\bm{B^\ell}(\bm{v})$ is the matrix whose rows correspond to non-zero vectors $\{\bm{u_{a_i}}\}_{i=1}^{\ell-1}$, which by Proposition~\ref{prop:state} defines an orthonormal basis of $\Row(\bm{F_v})$. Expression \eqref{eq:gDerivation} is precisely the closed form expression of $g^\ell(\bm{v})$ given in Lemma~\ref{prop:Gexplicit}.
	
	The last sentence of the proposition then follows since any feasible solution $\bm{z}\in Z$ is represented by a root-terminal path, and (as just proven) the length of any such path is $g^{n+1}(\bm{z})$; thus by taking the shortest among these paths, we recover an optimal solution of \eqref{eq:optXSimple}.
\end{proof}

\subsection{Lemma~\ref{lem:closedFormState}}
\begin{resultrestatement}{Lemma~\ref{lem:closedFormState} (Closed form of the states)}
	Suppose that $\bm{Q}=\bm{FF^\top}$ is invertible. For any $\ell\in [2:n+1]$ and $\bm{v}\in \{0,1\}^{\ell-1}$ with support $T=\left\{i\in [\ell-1]: v_i=1\right\}$, the state matrix of the node of the exact decision diagram corresponding to $\bm{v}$ is
		\[
		\bm{H^\ell}(\bm{v}) = \bm{F_{[\ell:n]}}-\bm{Q_{[\ell:n],T}} \left(\bm{Q_{T,T}}\right)^{-1} \bm{F_{T}},
		\]
		with the convention that $\bm{H^\ell}(\bm{v})=\bm{F_{[\ell:n]}}$ if $T=\emptyset$.
\end{resultrestatement}

\begin{proof}
	If $T=\emptyset$, then $\Row(\bm{F_v})=\{\bm{0}\}$, thus $\bm{G_\perp^\ell}(\bm{v})=\bm{I}$ and Proposition~\ref{prop:state} gives $\bm{H^\ell}(\bm{v})=\bm{F_{[\ell:n]}}$. Otherwise, since $\bm{Q}\succ\bm{0}$, the rows of $\bm{F}$ are linearly independent; in particular, $\bm{F_v}$ has full row rank and $\bm{F_vF_v^\top}=\bm{Q_{T,T}}$ is invertible. Combining Propositions~\ref{prop:state} and \ref{prop:closedFormProjection}, we find that
	\begin{align*}
		\bm{H^\ell}(\bm{v})=\bm{F_{[\ell:n]}}\bm{G_\perp^\ell}(\bm{v})
		&=\bm{F_{[\ell:n]}}\left(\bm{I}-\bm{F_v^\top}\left(\bm{F_vF_v^\top}\right)^{-1}\bm{F_v}\right)\\
		&=\bm{F_{[\ell:n]}}-\bm{Q_{[\ell:n],T}}\left(\bm{Q_{T,T}}\right)^{-1}\bm{F_{T}},
	\end{align*}
	where the last equality uses $\bm{F_{[\ell:n]}}\bm{F_v^\top}=\bm{Q_{[\ell:n],T}}$ and $\bm{F_v}=\bm{F_T}$.
\end{proof}

\subsection{Proposition~\ref{prop:valid}}
\begin{resultrestatement}{Proposition~\ref{prop:valid} (Validity)}
	It holds that $X_{\bm{F},Z}\subseteq \text{proj}_{(x_0,\bm{x},\bm{z})}R_{\bm{F},Z}$.
\end{resultrestatement}

\begin{proof}
	Consider any $(x_0,\bm{x},\bm{z})\in X_{\bm{F},Z}$. We show how to find $(\bm{w},\bm{r})$ such that $(x_0,\bm{x},\bm{z},\bm{w},\bm{r})\in R_{\bm{F},Z}$. First, since $\bm{z}$ is binary and feasible, there exists a unique arc-specified path $(a_1,\dots, a_n)$ in the diagram representing this solution where $\ell(a_1)=1$ and $\ell(a_n)=n$; letting $\bm{r}\in \{0,1\}^A$ such that $r_{a_i}=1$ for arcs in the path and $r_a=0$ otherwise, we find that constraints \eqref{eq:convexhull_z}-\eqref{eq:convexhull_nonneg} are satisfied. Since $\bm{F^\top x}=\bm{F^\top} (\bm{z}\circ \bm{x})=\left(\Diag(\bm{z})\bm{F}\right)^\top\bm{x}$, it follows that $\bm{F^\top x}\in \Row(\Diag(\bm{z})\bm{F})=\Row(\bm{F_z})$. By point 1 of Proposition~\ref{prop:state}, the non-zero transition vectors in path $(a_1,\dots, a_n)$ are an orthonormal basis of $\Row(\bm{F_z})$, thus $$\bm{F^\top x}=\sum_{i\in [n]:\bm{u_{a_i}}\neq \bm{0}}\bm{u_{a_i}} y_i$$ for some suitable weights $\bm{y}\in \R^n$. Letting $\bm{w}\in \R^A$ be the vector such that $w_{a_i}=y_i$ for arcs in the path with non-zero transition vectors, and $w_a=0$ for all other arcs, we find that constraint \eqref{eq:convexhull_x} is satisfied. Finally, note that 
	\begin{align*}
		\|\bm{F^\top x}\|_2^2&=\left\|\sum_{i\in [n]:\bm{u_{a_i}}\neq \bm{0}}\bm{\bm{u_{a_i}}}w_{a_i}\right\|_2^2\\
		&=\sum_{i\in [n]:\bm{u_{a_i}}\neq \bm{0}}\|\bm{\bm{u_{a_i}}}w_{a_i}\|_2^2\tag{$\because$ \text{vectors $\bm{u_{a_i}}$ are orthogonal}}\\
		&=\sum_{i\in [n]:\bm{u_{a_i}}\neq \bm{0}}w_{a_i}^2\tag{$\because$ \text{vectors $\bm{u_{a_i}}$ have unit norm}}\\
		&=\sum_{i\in [n]:\bm{u_{a_i}}\neq \bm{0}}\frac{w_{a_i}^2}{r_{a_i}}\tag{$\because\; r_{a_i}=1$ for vectors in the path}\\
		&=\sum_{a\in A}\frac{w_{a}^2}{r_{a}}\tag{$\because\; w_{a}=0\implies w_a^2/r_a=0$ for other vectors},
	\end{align*}
	thus we have verified that $x_0\geq \|\bm{F^\top x}\|_2^2$ implies constraint \eqref{eq:convexhull_nonlinear}, thus all constraints in \eqref{eq:convexhull} are satisfied. 
\end{proof}

\subsection{Theorem~\ref{theo:hullDD}}
\begin{resultrestatement}{Theorem~\ref{theo:hullDD}}
	If Assumption~\ref{ass:unbounded} holds, then $\text{cl conv}(X_{\bm{F},Z})= \text{proj}_{(x_0,\bm{x},\bm{z})}R_{\bm{F},Z}$.
\end{resultrestatement}

\begin{proof}
	Consider an arbitrary optimization problem over $X_{\bm{F},Z}$, described in \eqref{eq:optimizationX}, and the optimization of the same linear function over $R_{\bm{F},Z}$:
	\begin{subequations}\label{eq:optimizationR}
		\begin{align}
			\min_{\bm{x},\bm{z},\bm{w},\bm{r}}&\;\bm{d^\top x}+\bm{c^\top z}+\sum_{a\in A}\frac{w_a^2}{r_a}\\
			\text{s.t.}\;&\eqref{eq:convexhull_x}-\eqref{eq:convexhull_nonneg}.
		\end{align}
	\end{subequations}
	We show that, under Assumption~\ref{ass:unbounded}, both problems are equivalent: either both of them are unbounded, or there is an optimal solution of \eqref{eq:optimizationR} whose projection onto $(x_0,\bm{x},\bm{z})$ is also optimal for \eqref{eq:optimizationX} with the same objective value.
	
	\noindent $\bullet$ If there does not exist $\bm{\delta}\in \R^m$ such that $\bm{F\delta}=\bm{d}$, then by Assumption~\ref{ass:unbounded} problem \eqref{eq:optimizationX} is unbounded. Since $R_{\bm{F},Z}$ is a relaxation of $X_{\bm{F},Z}$ by Proposition~\ref{prop:valid}, problem \eqref{eq:optimizationR} is also unbounded.
	
	\noindent $\bullet$ If there exists $\bm{\delta}\in \R^m$ such that $\bm{F\delta}=\bm{d}$, then by Assumption~\ref{ass:unbounded} problem \eqref{eq:optimizationX} is not unbounded. In this case we can rewrite  \eqref{eq:optimizationR} as 
	\begin{align*}
		&\min_{\bm{x},\bm{z},\bm{w},\bm{r}}\;\bm{\delta^\top F^\top x}+\bm{c^\top z}+\sum_{a\in A}\frac{w_a^2}{r_a};
		\text{s.t.}\;\eqref{eq:convexhull_x}-\eqref{eq:convexhull_nonneg}\\
		=&\min_{\bm{w},\bm{r}}\;\bm{\delta^\top}\left(\sum_{a\in A}\bm{u_a}w_a\right)+\bm{c^\top}\left(\sum_{\substack{a\in A\\\nu_a=1}}\bm{e_{\ell(a)}}r_a\right)+\sum_{a\in A}\frac{w_a^2}{r_a}\;
		\text{s.t.}\;\eqref{eq:convexhull_f3}-\eqref{eq:convexhull_nonneg}
	\end{align*}
	by projecting out the original variables $(\bm{x},\bm{z})$ using constraints \eqref{eq:convexhull_x}-\eqref{eq:convexhull_z}. Observe that since variables $\bm{w}$ no longer appear in any constraint, we find by taking derivatives with respect to $w_a$ and setting to zero that $w_a^*=-\frac{r_a}{2}\bm{\delta^\top u_a}$ is optimal for $w_a$. Replacing $\bm{w}$ with their optimal values, we find that problem \eqref{eq:optimizationR} reduces to 
	\begin{align}
		=&\min_{\bm{r}}\;-\frac{1}{2}\sum_{a\in A}\left(\bm{\delta^\top u_a}\right)^2r_a+\sum_{\substack{a\in A\\\nu_a=1}}c_{\ell(a)}r_a+\frac{1}{4}\sum_{a\in A}\left(\bm{\delta^\top u_a}\right)^2r_a\;
		\text{s.t.}\;\eqref{eq:convexhull_f3}-\eqref{eq:convexhull_nonneg}\notag\\
		=&\min_{\bm{r}}\;\sum_{a\in A}\left(c_{\ell(a)}\nu_a-\frac{1}{4}\left(\bm{\delta^\top u_a}\right)^2\right)r_a\;
		\text{s.t.}\;\eqref{eq:convexhull_f3}-\eqref{eq:convexhull_nonneg}.\label{eq:optPath}
	\end{align}
	Since constraints \eqref{eq:convexhull_f3}-\eqref{eq:convexhull_nonneg} are the constraints associated with the (unimodular) path polytope and the objective function is linear in $\bm{r}$, there exists an optimal solution that is integral in $\bm{r}$ and represents a path in the graph $\mathcal{G}$. Since the arc costs are precisely the arc lengths described in Proposition~\ref{prop:pathLength}, we find from this proposition that the optimal solution to this shortest path problem corresponds to the optimal solution $\bm{\bar z}\in \{0,1\}^n$ of optimization problem \eqref{eq:optimizationX}, and that the two problems have the same objective value. By repeating the same arguments of Proposition~\ref{prop:valid}, we find that $$\sum_{a\in A}\bm{u_a}w_a\in \Row\left(\Diag(\bm{\bar z})\bm{F}\right)\Leftrightarrow \sum_{a\in A}\bm{u_a}w_a=\bm{F^\top}\Diag(\bm{\bar z})\bm{y}$$
	for some $\bm{y}\in \R^m$, and letting $\bm{\bar x}=\Diag(\bm{\bar z})\bm{y}$ we obtain a solution for the continuous variables that satisfies all constraints in $X_{\bm{F},Z}$. 
\end{proof}

\subsection{Proposition~\ref{prop:lowRankSize}}
\begin{resultrestatement}{Proposition~\ref{prop:lowRankSize}}
	The number of nodes in the exact decision diagram for $X_{\bm{F}}$ is $\mathcal{O}(n^m)$.
\end{resultrestatement}

\begin{proof}
	From Proposition~\ref{prop:state} (point 1), the non-zero transition vectors of any given path describe an orthonormal basis of a subspace of $\R^m$; thus, there can be at most $m$ non-zero vectors on any path. Moreover, matrix $\bm{G_\perp^\ell}(\bm{v})$ in point 2 of the proposition, and the overall state matrix $\bm{\bar H^\ell}$, is determined by these non-zero vectors. Finally, note that any $m$ non-zero vectors result in the same space $\R^m$ with associated projection matrix $\bm{G_\perp^\ell}(\bm{v})=\bm{0}$. It follows that at any given layer $\ell\geq m$ the maximum number of distinct states is $1+\sum_{i=0}^{m-1}\binom{\ell-1}{i}=\mathcal{O}(\ell^{m-1})$: one absorbing state, plus one state for each choice of at most $m-1$ non-zero rows of $\bm{F_{[\ell-1]}}$ (including the empty choice, corresponding to the state $\bm{F_{[\ell:n]}}$).
	Summing this bound over all $n+1$ layers yields the desired result.
\end{proof}

\subsection{Lemma~\ref{lem::number-mincut}}
\begin{resultrestatement}{Lemma~\ref{lem::number-mincut}}
	If $\supp(\Q)$ is a connected rooted tree with $n$ nodes and
		$k$ leaves, then the number of distinct minimal cuts in
		$\mathcal{G}_{\Q}$ is at most $
		\left(\frac{n-1}{k}+2\right)^k.$
\end{resultrestatement}

\begin{proof}
	We prove the claim by induction on the number of nodes.
	Let $\mathcal{T}=\supp(\Q)$, and let $M(\mathcal{T})$ denote the number of minimal cuts in its augmented graph.
	
	\paragraph{Base case} If $\mathcal{T}$ consists of a single node, then $n=k=1$. Its augmented graph is the path $s_0$--$r$--$s_1$, where $r$ is the root, and hence $M(\mathcal{T})=2=((1-1)/1+2)^1$.
	
	\paragraph{Induction step} Now suppose that the root $r$ has $d\geq1$ children, and let $\mathcal{T}_1,\ldots,\mathcal{T}_d$ be the rooted subtrees induced by those children. For each $j\in[d]$, let $n_j$ and $k_j$ denote the numbers of nodes and leaves of $\mathcal{T}_j$, respectively. Thus, $\sum_{j=1}^d n_j=n-1$ and $\sum_{j=1}^d k_j=k$.
	
	We first claim that $M(\mathcal{T})=1+\prod_{j=1}^d M(\mathcal{T}_j)$. Let $S$ be a minimal cut in the augmented graph of $\mathcal{T}$. If $r\notin S$, then the connectivity of $S$, together with the fact that $s_0$ is adjacent only to $r$, implies that $S=\{s_0\}$. Thus, there is exactly one minimal cut that does not contain the root.
	
	Now suppose that $r\in S$. For each $j\in[d]$, let $U_j=S\cap V(\mathcal{T}_j)$ with $V(\mathcal{T}_j)$ denoting the node set of $\mathcal{T}_j$, and augment $\mathcal{T}_j$ by connecting a local source to its root and a local sink to all its leaves. Then the local source together with $U_j$ forms a minimal cut in this augmented graph. Its source side is connected because $S$ is connected and does not contain $s_1$: for every node of $U_j$, the path in $S$ to $r$ must pass through the root of $\mathcal{T}_j$. Its sink side is also connected. Indeed, every node of $V(\mathcal{T}_j)\setminus U_j$ has a path to $s_1$ in the complement of $S$. Since $r\in S$, this path cannot leave $\mathcal{T}_j$ through its edge to $r$ and must instead reach $s_1$ through a leaf of $\mathcal{T}_j$. It therefore induces a path to the local sink in the augmented graph of $\mathcal{T}_j$.
	
	Conversely, choose independently a minimal cut in the augmented graph of each $\mathcal{T}_j$, and let $U_j$ be the tree vertices on its source side. Then $S=\{s_0,r\}\cup\bigcup_{j=1}^d U_j$ is a minimal cut in the augmented graph of $\mathcal{T}$. Its source side is connected through $r$, and its sink side is connected because the sink sides of the selected cuts become connected when their local sinks are identified with $s_1$. This proves the claimed recursion.
	
	By the induction hypothesis, $M(\mathcal{T}_j)\leq\left(\frac{n_j-1}{k_j}+2\right)^{k_j}$ for every $j\in[d]$. The weighted arithmetic-geometric mean inequality, with weights $k_j/k$, therefore gives
	\begin{align*}
		\prod_{j=1}^d M(\mathcal{T}_j)
		\leq\prod_{j=1}^d\left(\frac{n_j-1}{k_j}+2\right)^{k_j}
		\leq\left(\sum_{j=1}^d\frac{k_j}{k}
		\left(\frac{n_j-1}{k_j}+2\right)\right)^k
		=\left(2+\frac{n-1-d}{k}\right)^k.
	\end{align*}
	Set $x=2+\frac{n-1-d}{k}$. Since $n-1=\sum_{j=1}^d n_j\geq d$, we have $x\geq2$. Moreover, the binomial theorem implies
	\[
	\left(x+\frac{d}{k}\right)^k-x^k\geq kx^{k-1}\left(\frac{d}{k}\right)=dx^{k-1}\geq1.
	\]
	Hence,
	\[
	M(\mathcal{T})
	=1+\prod_{j=1}^d M(\mathcal{T}_j)
	\leq1+x^k
	\leq\left(x+\frac{d}{k}\right)^k
	=\left(\frac{n-1}{k}+2\right)^k,
	\]
	which completes the induction.
\end{proof}

\subsection{Theorem~\ref{thm::identical-states-zero}}
\begin{resultrestatement}{Theorem~\ref{thm::identical-states-zero}}
	Suppose that $\supp(\Q)$ is a connected tree. For any $\ell\in [2:n]$ and $\bm{v} \in \{0,1\}^{\ell-1}$ let $S = S_1^\ell(\bm{v})$ be the associated one-cut. Then, in the decision diagram, the state matrix of the node corresponding to assignment $\bm{v}$ is
		\[
		\bm{H^\ell}(\bm{v}) 
		= \bm{F_{[\ell:n]}}-\bm{Q_{[\ell:n], S}}\left(\bm{Q_{S, S}}\right)^{-1}\bm{F_{S}}.
		\]
\end{resultrestatement}

\begin{proof}
	Let $\bar S=\left\{i\in [\ell-1]: v_i=1\right\}$ and let $T=\bar S\setminus S$. We have $\bar S= T\cup S$. From Definition~\ref{def:one-cut} and the discussion immediately following it, there is no edge in $\supp(\bm{Q})$ between $S$ and $T$. Applying Lemma~\ref{lem:closedFormState} to $\bm{v}$, whose support is $\bar S$, we find that
	\begin{align*}\bm{H^\ell}(\bm{v})&=\bm{F_{[\ell:n]}}-\bm{Q_{[\ell:n],\bar S}}\left(\bm{Q_{\bar S,\bar S}}\right)^{-1}\bm{F_{\bar S}}\\
		&= \bm{F_{[\ell:n]}}-\bm{Q_{[\ell:n],S\cup T}} \left(\bm{Q_{S\cup T,S\cup  T}}\right)^{-1} \bm{F_{S\cup T}}\\
		&\stackrel{(a)}{=} \bm{F_{[\ell:n]}}-\begin{bmatrix}
			\bm{Q_{[\ell:n],S}} & \bm{Q_{[\ell:n], T}}
		\end{bmatrix}\begin{bmatrix}
			\left(\bm{Q_{S,S}}\right)^{-1} & 0\\
			0 & \left(\bm{Q_{T, T}}\right)^{-1}
		\end{bmatrix}\begin{bmatrix}
			\bm{F_{S}}\\
			\bm{F_{T}}
		\end{bmatrix}\\
		&\stackrel{(b)}{=} \bm{F_{[\ell:n]}}-\bm{Q_{[\ell:n],S}}\left(\bm{Q_{ S,S}}\right)^{-1}\bm{F_{S}},
	\end{align*}
	where, in $(a)$, we use the fact that $\bm{Q_{S,T}} = 0$ to compute the inverse; moreover, in $(b)$, we use the fact that $[\ell:n]\subset \tilde S_1(\bm{v})$ (see Definition~\ref{def:one-cut}), and hence, $\bm{Q_{[\ell:n],T}}=0$. This completes the proof.
\end{proof}

\subsection{Corollary~\ref{cor::polsize-zero}}
\begin{resultrestatement}{Corollary~\ref{cor::polsize-zero}}
	Suppose that $\supp(\Q)$ is a connected rooted tree with $k$ leaves. In this case, the number of nodes in the decision diagram constructed according to the variable order induced by the labeling of $\supp(\Q)$ is at most $n\left(\frac{n-1}{k}+2\right)^k=\mathcal{O}\left(n^{k+1}\right)$.
\end{resultrestatement}

\begin{proof}
	By Theorem~\ref{thm::identical-states-zero}, the number of distinct states at any layer of the decision diagram is at most the number of distinct one-cuts in $\mathcal{G}_{\Q}$.  Lemma~\ref{lem::number-mincut} shows that the number of such cuts is at most $\left(\frac{n-1}{k}+2\right)^k$.
	Therefore, each layer $\ell\in[2:n]$ of the decision diagram contains at most $\left(\frac{n-1}{k}+2\right)^k$ distinct states, while layers $1$ and $n+1$ contain a single node each (the root and the terminal node, respectively). The total number of nodes after merging identical states is thus at most $(n-1)\left(\frac{n-1}{k}+2\right)^k+2\leq n\left(\frac{n-1}{k}+2\right)^k$.
\end{proof}

\subsection{Lemma~\ref{lem::inverse}}\label{sec:proofLemInverse}
\begin{resultrestatement}{Lemma~\ref{lem::inverse}}
	Let $S$, $T$ and $V$ be a partition of $[n]$, with $S,T\neq \emptyset$ and $V$ possibly empty, such that no edge of $\supp(\bm{A})$ connects $S$ and $V$, that is, $\bm{A_{S,V}}=\bm{0}$. Then
		\begin{align*}
			\bm{Q_{S,T}} \left(\bm{Q_{T,T}}\right)^{-1} \;=\; - \left(\bm{A_{S,S}}\right)^{-1} \bm{A_{S,T}}.
		\end{align*}
\end{resultrestatement}

\begin{proof}
	Let $T^c = [n]\setminus T = S \cup V$, and let $\bm{I_{S,T^c}}$ denote the submatrix of the $n\times n$ identity matrix with rows indexed by $S$ and columns indexed by $T^c$. Since $\bm{A}\succ \bm{0}$, the block $\bm{A_{T^c,T^c}}$ is invertible, and the Schur-complement identity for the inverse of a partitioned matrix, applied to $\bm{Q}=\bm{A}^{-1}$, yields
	\begin{equation*}
		\left(\bm{Q_{T,T}}\right)^{-1} = \bm{A_{T,T}}-\bm{A_{T,T^c}}\left(\bm{A_{T^c,T^c}}\right)^{-1}\bm{A_{T^c,T}}.
	\end{equation*}
	
	Moreover, consider the matrix-inverse identity
	{\footnotesize\begin{align*}
			\bm{QA}=\bm{I}\Leftrightarrow \begin{pmatrix}\bm{Q_{S,S}}&\bm{Q_{S,T}}&\bm{Q_{S,V}}\\
				\bm{Q_{T,S}}&\bm{Q_{T,T}}&\bm{Q_{T,V}}\\
				\bm{Q_{V,S}}&\bm{Q_{V,T}}&\bm{Q_{V,V}}\end{pmatrix} \begin{pmatrix}\bm{A_{S,S}}&\bm{A_{S,T}}&\bm{0}_{S,V}\\
				\bm{A_{T,S}}&\bm{A_{T,T}}&\bm{A_{T,V}}\\
				\bm{0}_{V,S}&\bm{A_{V,T}}&\bm{A_{V,V}}\end{pmatrix}=\begin{pmatrix}\bm{I_{S,S}}&\bm{0}_{S,T}&\bm{0}_{S,V}\\
				\bm{0}_{T,S}&\bm{I_{T,T}}&\bm{0}_{T,V}\\
				\bm{0}_{V,S}&\bm{0}_{V,T}&\bm{I_{V,V}}\end{pmatrix}
	\end{align*}}
	The blocks of the matrix-inverse identity with rows indexed by $S$ and columns indexed by $T$ read \begin{align*}
		&\bm{0}_{S,T}=\bm{Q_{S,S}A_{S,T}}+\bm{Q_{S,T}}\bm{A_{TT}}+\bm{Q_{S,V}A_{V,T}} =\bm{Q_{S,T}}\bm{A_{TT}}+ \bm{Q_{S,T^c}A_{T^c,T}}\\
		\implies\;& \bm{Q_{S,T}A_{TT}}=-\bm{Q_{S,T^c}A_{T^c,T}}.\end{align*}
	Similarly, the blocks with rows indexed by $S$ and columns indexed by $T^c$ read
	\begin{align*}
		&\begin{pmatrix}\bm{I_{S,S}}&\bm{0}_{S,V}\end{pmatrix}=\begin{pmatrix}\bm{Q_{S,S}A_{S,S}}+\bm{Q_{S,T}A_{T,S}} & \bm{Q_{S,T}A_{T,V}}+\bm{Q_{S,V}A_{V,V}}\\
		\end{pmatrix}
		\\
		\Leftrightarrow\;&\begin{pmatrix}\bm{I_{S,S}}-\bm{Q_{S,S}A_{S,S}}&\bm{0}_{S,V}-\bm{Q_{S,V}A_{V,V}}\end{pmatrix}=\begin{pmatrix}\bm{Q_{S,T}A_{T,S}} & \bm{Q_{S,T}A_{T,V}}\\
		\end{pmatrix}
		\\
		\Leftrightarrow\;&\bm{I_{S,T^c}}-\bm{Q_{S,T^c}A_{T^c,T^c}}=\bm{Q_{S,T}A_{T,T^c}}.
	\end{align*}
	
	Combining these derived Schur complement and matrix inverse identities, we find that
	\begin{align*}
		\bm{Q_{S,T}}\left(\bm{Q_{T,T}}\right)^{-1} &= \bm{Q_{S,T}A_{T,T}}-\bm{Q_{S,T}A_{T,T^c}}\left(\bm{A_{T^c,T^c}}\right)^{-1}\bm{A_{T^c,T}}\\
		&= -\bm{Q_{S,T^c}A_{T^c,T}}-\left(\bm{I_{S,T^c}}-\bm{Q_{S,T^c}A_{T^c,T^c}}\right)\left(\bm{A_{T^c,T^c}}\right)^{-1}\bm{A_{T^c,T}}\\
		&=-\bm{I_{S,T^c}}\left(\bm{A_{T^c,T^c}}\right)^{-1}\bm{A_{T^c,T}}.
	\end{align*}
	Finally, since $\bm{A_{S,V}} = \bm{0}$ and $\bm{A_{V,S}}=\bm{0}$, we may write $\bm{A_{T^c,T^c}} = \begin{pmatrix}
		\bm{A_{S,S}} & \bm{0}\\
		\bm{0} & \bm{A_{V,V}}
	\end{pmatrix}$, and thus $\bm{I_{S,T^c}}\left(\bm{A_{T^c,T^c}}\right)^{-1} = \begin{pmatrix}
		\left(\bm{A_{S,S}}\right)^{-1} & \bm{0}
	\end{pmatrix}$, leading to
	\begin{align*}
		-\bm{I_{S,T^c}}\left(\bm{A_{T^c,T^c}}\right)^{-1}\bm{A_{T^c,T}} = -\begin{pmatrix}
			\left(\bm{A_{S,S}}\right)^{-1} & \bm{0}
		\end{pmatrix}\begin{pmatrix}
			\bm{A_{S,T}}\\
			\bm{A_{V,T}}
		\end{pmatrix} = - \left(\bm{A_{S,S}}\right)^{-1}\bm{A_{S,T}},
	\end{align*}
	thereby completing the proof; if $V=\emptyset$, the last step holds trivially with $\bm{A_{T^c,T^c}}=\bm{A_{S,S}}$.
\end{proof}

\subsection{Theorem~\ref{thm::identical-states-one}}
\begin{resultrestatement}{Theorem~\ref{thm::identical-states-one}}
	Suppose that $\supp(\bm{A})$ is a connected tree. For any $\ell\in [2:n]$ and $\bm{v} \in \{0,1\}^{\ell-1}$, let $S = S_0^\ell(\bm{v})$ be the associated zero-cut and let $\Gamma=\Gamma_{\bm{A}}(S)$ be the set of neighbors of $S$ in $\supp(\bm{A})$. Then, in the decision diagram, the state matrix of the node corresponding to assignment $\bm{v}$ is
		\[
		\bm{H^\ell}(\bm{v})
		= \bm{F_{[\ell:n]}} + \left(\bm{A_{S,S}^{-1}}\right)_{[\ell:n]}\bm{A_{S,\Gamma}} \bm{F_{\Gamma}},
		\]
		where $\left(\bm{A_{S,S}^{-1}}\right)_{[\ell:n]}$ denotes the submatrix of $\bm{A_{S,S}^{-1}}$ corresponding to rows with indices in $[\ell:n]\subseteq S$, and the second term is understood to be $\bm{0}$ if $\Gamma=\emptyset$.
\end{resultrestatement}

\begin{proof}
	Let $T = \left\{i\in [\ell-1]:v_i=1\right\}$ and $V = [n]\setminus (S\cup T)$. By the properties of zero-cuts discussed after Definition~\ref{def:zero-cut}, sets $S$, $T$ and $V$ indeed partition $[n]$, and $\Gamma\subseteq T$ (in turn implying that $\Gamma\cap V=\emptyset$ and in particular there are no edges in $\supp(\A)$ between $S$ and $V$).
	
	First suppose that $T=\emptyset$, that is, $\bm{v}=\bm{0}_{\ell-1}$. Then $S=[n]$, $\Gamma=\emptyset$, and Lemma~\ref{lem:closedFormState} gives $\bm{H^\ell}(\bm{v})=\bm{F_{[\ell:n]}}$, matching the claimed identity. In the remainder of the proof we assume $T\neq \emptyset$.
	
	By Lemma~\ref{lem:closedFormState}, we have
	\begin{equation}\label{eq:invTreeBase}
		\bm{H^\ell}(\bm{v}) = \bm{F_{[\ell:n]}}-\bm{Q_{[\ell:n],T}} \left(\bm{Q_{T,T}}\right)^{-1} \bm{F_{T}}.
	\end{equation}
	Since no edge in $\supp(\A)$ connects $S$ and $V$, Lemma~\ref{lem::inverse} applies to the partition $(S,T,V)$. Restricting the resulting identity to the rows indexed by $[\ell:n]\subseteq S$, we find that
	\begin{align*}
		\bm{Q_{[\ell:n],T}}\left(\bm{Q_{T,T}}\right)^{-1}= -\left(\bm{A_{S,S}^{-1}}\right)_{[\ell:n]}\bm{A_{S,T}}.
	\end{align*}
	Finally, since $\Gamma\subseteq T$ and $\bm{A_{S,\{j\}}}=\bm{0}$ for every $j\in T\setminus \Gamma$, we have $\bm{A_{S,T}F_{T}}=\bm{A_{S,\Gamma}F_{\Gamma}}$. Substituting the last two identities into \eqref{eq:invTreeBase} completes the proof.
\end{proof}

\subsection{Corollary~\ref{cor::polsize-one}}
\begin{resultrestatement}{Corollary~\ref{cor::polsize-one}}
	Suppose that $\supp(\bm{A})$ is a connected rooted tree with $k$ leaves. Then the number of nodes in the decision diagram constructed according to the node ordering induced by the labeling of $\supp(\bm{A})$ is at most $n\left(\frac{n-1}{k}+2\right)^k=\mathcal{O}\left(n^{k+1}\right)$.
\end{resultrestatement}

\begin{proof}
	By Theorem~\ref{thm::identical-states-one}, the number of distinct states at any layer $\ell$ of the decision diagram is at most the number of distinct zero-cuts $S_0^\ell(\bm{v})$. As discussed after Definition~\ref{def:zero-cut}, the map $S_0^\ell(\bm{v})\mapsto S_0^\ell(\bm{v})\cup\{s_0\}$ assigns to each zero-cut a distinct minimal cut of $\mathcal{G}_{\bm{A}}$, and by Lemma~\ref{lem::number-mincut} there are at most $\left(\frac{n-1}{k}+2\right)^k$ such cuts. The result follows by summing over the layers of the diagram.
\end{proof}

\subsection{Theorem~\ref{thm:approxSummary}}
\begin{resultrestatement}{Theorem~\ref{thm:approxSummary} (Summary of the main results of the section)}
	Suppose that $G=\supp(\Q)$ or $G=\supp(\Q^{-1})$ has polynomial volume growth with parameters $\gamma,\delta>0$, and that the associated $\eta$-boundaries satisfy $\omega_{G}(\eta)\leq \delta'\eta^{\gamma'}$ for all $\eta\geq 1$, for some parameters $\gamma',\delta'>0$. Then, for every $\epsilon>0$, there exists an $\epsilon$-exact decision diagram for $X_{\bm{F}}$ with the number of nodes bounded by
		\[
		(n+1)\cdot 2^{\delta'\eta_\epsilon^{\gamma'}},
		\] where $\eta_\epsilon=\mathcal{O}\left(\log(1/\epsilon)\right)$ with constants depending only on $\gamma$, $\delta$ and the extreme eigenvalues of $\Q$.
\end{resultrestatement}

\begin{proof}
	The result follows by combining Corollary~\ref{cor::epsilon-Q-size} (if $G=\supp(\Q)$) or Corollary~\ref{cor::epsilon-A-size} (if $G=\supp(\Q^{-1})$) with the assumed bound on $\omega_G(\eta)$, and by observing that the values of $\eta$ prescribed in \eqref{eq::m-lb-Q} and \eqref{eq::m-lb-A} satisfy $\eta=\mathcal{O}\left(\log(1/\epsilon)\right)$.
\end{proof}

\subsection{Auxiliary lemmas and notation for Theorem~\ref{thm::approx-Q} and Theorem~\ref{thm::approx-A}}\label{sec:auxNotation}

We record some useful structural properties of positive definite matrices. Let $I \subseteq [n]$, and let $V$ and $W$ be a partition of $I$, i.e., $I = V \cup W$ and $V \cap W = \emptyset$. Denoting by $\bm{S}\defeq \Q_{II}/\Q_{VV}= \Q_{WW}-\Q_{WV}\Q_{VV}^{-1}\Q_{VW}$ the Schur complement of $\Q_{VV}$ in $\Q_{II}$, we have
\begin{align}
	\Q_{II}=\begin{pmatrix}
		\Q_{VV} & \Q_{VW}\\
		\Q_{WV} & \Q_{WW}
	\end{pmatrix},\;\,
	\left(\Q_{II}\right)^{-1} = \begin{pmatrix}
		\Q^{-1}_{VV}+\Q^{-1}_{VV}\Q_{VW}\bm{S}^{-1}\Q_{WV}\Q^{-1}_{VV} & -\Q^{-1}_{VV}\Q_{VW}\bm{S}^{-1}\\
		-\bm{S}^{-1}\Q_{WV}\Q^{-1}_{VV} & \bm{S}^{-1}
	\end{pmatrix}.\label{eq::inverse}
\end{align} The following lemma is borrowed from \citet{bhathena2026solving} and characterizes the decaying structure of $\left(\Q_{II}\right)^{-1}$.
\begin{lemma}[Lemma 5 in \citealp{bhathena2026solving}]\label{lem:decayingInv}
	Let $\Q\in \mathbb{R}^{n\times n}$ be a symmetric positive definite matrix, and $I\subseteq [n]$ any subset of its rows/columns. Let $\pi: I\to \{1,\dots, |I|\}$ be the canonical indexing map that assigns to each $i \in I$ its corresponding row/column position within the submatrix $\Q_{II}$.  For any $i,j\in I$, we have
	\begin{align}
		\left|\left(\left(\Q_{II}\right)^{-1}\right)_{\pi(i),\pi(j)}\right|\le C_1\sigma^{d_{\supp(\Q)}(i,j)},
	\end{align}
	where $C_1\defeq\frac{1}{\mu_{\min}(\Q)}\ \max\left\{1,\frac{(1+\sqrt{\kappa_2})^2}{2\kappa_2}\right\}$.
\end{lemma}

\paragraph{Zero-padding convention}
A hat over a matrix indicates that a vector or matrix whose rows and columns are indexed by subsets of $[n]$ is zero-padded in the remaining coordinates, so that sums and differences of quantities indexed by different subsets are well defined. For instance, given $T\subseteq[n]$ and $\bm{F}\in\R^{n\times m}$, the matrix $\hat{\bm{F}}_T\in\R^{n\times m}$ coincides with $\bm{F_T}$ in the rows indexed by $T$ and is zero elsewhere; in particular, $\hat{\bm{F}}_T=\Diag(\bm{z})\bm{F}$, where $\bm{z}\in\{0,1\}^n$ is the indicator vector of $T$. Similarly, given $S,T\subseteq[n]$ and $\bm{M}\in\R^{n\times n}$, the matrix $\hat{\bm{M}}_{S,T}\in\R^{n\times n}$ coincides with $\bm{M_{S,T}}$ in the entries indexed by $S\times T$ and is zero elsewhere. Inverses are always taken before zero-padding: for example, $(\hat{\bm{M}}_{T,T})^{-1}$ denotes the zero-padding of $(\bm{M_{T,T}})^{-1}$, because the padded matrix itself is singular whenever $T\neq[n]$.

\subsection{Theorem~\ref{thm::approx-Q}}

\begin{resultrestatement}{Theorem~\ref{thm::approx-Q}}
	Fix any $\ell\in [2:n]$ and $i \in [\ell:n]$, and consider $\bm{v_1}, \bm{v_2} \in \{0,1\}^{\ell-1}$ that are $\eta$-similar in $\supp(\Q)$ with respect to $i$. Then,
		\begin{align*}
			\norm{\bm{H^\ell}(\bm{v_1})_{\{i\}}-\bm{H^\ell}(\bm{v_2})_{\{i\}}}_2\leq \left(\frac{10\,\mu^{3.5}_{\max}(\Q)}{\mu^{3}_{\min}(\Q)}\right)\sqrt{\Delta_{i,1}\Delta_{i,\eta}}\,\sigma^\eta,
		\end{align*}
		where $\bm{H^\ell}(\bm{v})_{\{i\}}$ is the row of the state matrix $\bm{H^\ell}(\bm{v})$ corresponding to index $i$.
\end{resultrestatement}

\begin{proof} We use in this proof the additional lemmas and notation introduced in \ref{sec:auxNotation}. 
	For $s\in \{1,2\}$, let $T^{(s)}=\{j\in [\ell-1]: (\bm{v_s})_j=1\}$ denote the support of $\bm{v_s}$; if $T^{(s)}=\emptyset$, the corresponding terms below are understood to vanish. Recall from Lemma~\ref{lem:closedFormState} that
	\begin{align*}
		\bm{H^\ell}(\bm{v_s}) &= \F_{[\ell:n]}-\Q_{[\ell:n],T^{(s)}} \left(\Q_{T^{(s)}T^{(s)}}\right)^{-1} \F_{T^{(s)}}.
	\end{align*}
	Therefore, using the zero-padding convention above, we have
	{\small\begin{align*}
			&\bm{H^\ell}(\bm{v_1})_{\{i\}}-\bm{H^\ell}(\bm{v_2})_{\{i\}} 
			= \hat\Q_{\{i\},T^{(2)}} \left(\hat\Q_{T^{(2)}T^{(2)}}\right)^{-1} \hat\F_{T^{(2)}}-\hat\Q_{\{i\},T^{(1)}} \left(\hat\Q_{T^{(1)}T^{(1)}}\right)^{-1}\hat\F_{T^{(1)}} \\
			&= \underbrace{\left(\hat\Q_{\{i\},T^{(2)}}-\hat\Q_{\{i\},T^{(1)}}\right)\left(\hat\Q_{T^{(2)}T^{(2)}}\right)^{-1}\hat\F_{T^{(2)}}}_{:= A_1}
			+ \underbrace{\hat\Q_{\{i\},T^{(1)}}\left(\left(\hat\Q_{T^{(2)}T^{(2)}}\right)^{-1}-\left(\hat\Q_{T^{(1)}T^{(1)}}\right)^{-1}\right)\hat\F_{T^{(2)}}}_{:= A_2}\\
			&\quad+\underbrace{\hat\Q_{\{i\},T^{(1)}}\left(\hat\Q_{T^{(1)}T^{(1)}}\right)^{-1}\left(\hat\F_{T^{(2)}}-\hat\F_{T^{(1)}}\right)}_{:= A_3}.
	\end{align*}}%
	We bound each term in the above equation separately.
	First, we argue $\hat\Q_{\{i\},T^{(2)}} = \hat\Q_{\{i\},T^{(1)}}$, which in turn implies $A_1 = 0$. On the one hand, if $j\in [n]$ satisfies $Q_{ij}= 0$, then $(\hat\Q_{\{i\},T^{(2)}})_j = (\hat\Q_{\{i\},T^{(1)}})_j=0$ and the equality for that entry $j$ holds. On the other hand, if $Q_{ij}\neq 0$ then $d(i,j)\leq 1\leq \eta$, and $\eta$-similarity of $\bm{v_1}$ and $\bm{v_2}$ ensures that $j\in T^{(1)}\cap T^{(2)}$, implying $(\hat\Q_{\{i\},T^{(2)}})_j = (\hat\Q_{\{i\},T^{(1)}})_j=Q_{ij}$.

	Second, to bound $A_2$, let $V = \{j\in [\ell-1]: (\bm{v_1})_j=(\bm{v_2})_j=1,\; d(i,j)\leq \eta\}$ and $W^{(s)} = \{j\in [\ell-1]: (\bm{v_s})_j=1,\; d(i,j)> \eta\}$ for $s\in\{1,2\}$; By $\eta$-similarity, $T^{(s)}=V\cup W^{(s)}$. Note also that $\hat\Q_{\{i\},T^{(1)}}$ is supported on $V$, since $d(i,j)>\eta\geq 1$ implies $\Q_{ij}=0$ for $j\in W^{(1)}$. Denoting by $\bm{S^{(s)}}\defeq\Q_{T^{(s)}T^{(s)}}/\Q_{VV}$ the corresponding Schur complements, invoking~\eqref{eq::inverse} with the partition $T^{(s)}=V\cup W^{(s)}$, and observing that the terms involving $\Q_{VV}^{-1}$ alone cancel, straightforward algebraic manipulation yields
	{\small\begin{align*}
		\norm{A_2}_2 \leq&\!\!\!\! \sum_{s\in\{1,2\}}\!\!\!\!
		\left(\norm{\Q_{\{i\},V}\Q^{-1}_{VV}\Q_{VW^{(s)}}\left(\bm{S^{(s)}}\right)^{-1}\Q_{W^{(s)}V}\Q^{-1}_{VV}}_2\!\!+\norm{\Q_{\{i\},V}\Q^{-1}_{VV}\Q_{VW^{(s)}}\left(\bm{S^{(s)}}\right)^{-1}}_2\right)\norm{\F_{T^{(2)}}}_2\\
		\leq & \sum_{s\in\{1,2\}}\left(1+\norm{\Q_{W^{(s)}V}}_2\norm{\Q^{-1}_{VV}}_2\right)\norm{\Q_{\{i\},V}\Q^{-1}_{VV}\Q_{VW^{(s)}}}_2\norm{\left(\bm{S^{(s)}}\right)^{-1}}_2\norm{\F_{T^{(2)}}}_2.
	\end{align*}}
	Invoking Lemma~\ref{lem::decay}, we have
	{\small\begin{align*}
			\max\left\{\norm{\Q_{\{i\},V}\Q^{-1}_{VV}\Q_{VW^{(1)}}}_2,\norm{\Q_{\{i\},V}\Q^{-1}_{VV}\Q_{VW^{(2)}}}_2\right\}\leq \frac{2\mu^2_{\max}(\Q)}{\mu_{\min}(\Q)}\sqrt{\Delta_{i,1}\Delta_{i,\eta}}\, \sigma^\eta.
	\end{align*}}%
	Moreover, we have $\norm{\F_{T^{(2)}}}_2 \le \norm{\F}_2 \le \sqrt{\mu_{\max}(\Q)}$, $\max\{\norm{\Q_{VW^{(1)}}}_2, \norm{\Q_{VW^{(2)}}}_2\} \le \mu_{\max}(\Q)$, and $\norm{\Q_{VV}^{-1}}_2 \le 1/\mu_{\min}(\Q)$.
	Furthermore, by properties of the Schur complement, we have $\mu_{\min}\!\left(\bm{S^{(s)}}\right) \ge \mu_{\min}\!\left(\Q_{T^{(s)}T^{(s)}}\right) \ge \mu_{\min}(\Q)$, implying $\norm{(\bm{S^{(s)}})^{-1}}_2 \le 1/\mu_{\min}(\Q)$.
	Combining these bounds, we obtain
	\begin{align*}
		\norm{A_2}_2 \le \frac{8\,\mu_{\max}^{3.5}(\Q)}{\mu_{\min}^{3}(\Q)} \sqrt{\Delta_{i,1}\,\Delta_{i,\eta}}\, \sigma^\eta.
	\end{align*}
	Third, we bound $A_3$. Since $\hat\F_{T^{(2)}}-\hat\F_{T^{(1)}}$ is supported on $W^{(1)}\cup W^{(2)}$ and $\hat\Q_{\{i\},T^{(1)}}$ is supported on $V$, using~\eqref{eq::inverse} and straightforward algebra, we obtain
	\begin{align*}
		\norm{A_3}_2
		&\le \norm{\Q_{\{i\},V}\Q^{-1}_{VV}\Q_{VW^{(1)}}\left(\bm{S^{(1)}}\right)^{-1}}_2 \, \norm{\hat\F_{T^{(2)}}-\hat\F_{T^{(1)}}}_2 \\
		&\le \norm{\Q_{\{i\},V}\Q^{-1}_{VV}\Q_{VW^{(1)}}}_2 \, \norm{\left(\bm{S^{(1)}}\right)^{-1}}_2 \cdot 2\norm{\F}_2 \\
		&\le \frac{2\,\mu_{\max}^{2.5}(\Q)}{\mu_{\min}^{2}(\Q)} \sqrt{\Delta_{i,1}\,\Delta_{i,\eta}}\, \sigma^\eta.
	\end{align*}
	Combining the above bounds and noting that $\frac{\mu_{\max}^{2.5}(\Q)}{\mu_{\min}^2(\Q)}\leq \frac{\mu_{\max}^{3.5}(\Q)}{\mu_{\min}^3(\Q)}$ completes the proof.
\end{proof}

\subsection{Corollary~\ref{cor::epsilon-Q-size}}

\begin{resultrestatement}{Corollary~\ref{cor::epsilon-Q-size}}
	Suppose $\supp(\Q)$ has polynomial volume growth with parameters $\gamma,\delta>0$ (Definition~\ref{def::vol-Q}). Given any $\epsilon>0$, there exists an $\epsilon$-exact decision diagram with at most $(n+1)\cdot 2^{\omega_{\Q}(\eta)}$ nodes, where $\omega_{\Q}(\eta)$ is the maximum size of the $\eta$-boundaries of $\supp(\Q)$ (Definition~\ref{def:boundary}) and
		\begin{align*}
			\eta \;=\;
			\left\lceil\max\left\{
			\frac{2\log(C_{\Q}\delta / \epsilon)}{\log(1/\sigma)}, \;
			\frac{2\gamma}{\log(1/\sigma)}, \;
			\left(\frac{\gamma}{\log(1/\sigma)}\right)^2
			\right\}\right\rceil.
		\end{align*}
\end{resultrestatement}

\begin{proof}
	Consider the decision diagram obtained by merging, at every layer $\ell\in[2:n+1]$, all nodes whose associated partial solutions coincide on the $\eta$-boundary $\Lambda_{\ell}^{\eta}$; this is precisely the \underline{neighborhood merging} implementation of \texttt{merge} announced in \S\ref{sec:pseudocode}. The number of distinct nodes at layer $\ell$ is then at most $2^{|\Lambda_{\ell}^{\eta}|}\leq 2^{\omega_{\Q}(\eta)}$ and, since the diagram has $n+1$ layers, its total size is at most $(n+1)\cdot 2^{\omega_{\Q}(\eta)}$.
	
	It remains to verify that the resulting diagram is $\epsilon$-exact. By the discussion preceding the corollary, it suffices to ensure that
	\begin{align*}
		C_{\Q} \sqrt{\Delta_{1} \Delta_{\eta}}\, \sigma^\eta\leq \epsilon &\Longleftarrow C_{\Q}\delta \eta^{\gamma/2}\sigma^\eta\leq \epsilon\\
		&\iff \log\left(\eta^{\gamma/2}\sigma^\eta\right)\leq \log\left(\frac{\epsilon}{C_{\Q}\delta}\right)\\
		&\iff \frac{2\log(1/\sigma)}{\gamma}\cdot \eta - \log(\eta)\geq \frac{2\log\left(\frac{C_{\Q}\delta}{\epsilon}\right)}{\gamma}\\
		&\stackrel{(a)}{\Longleftarrow} \frac{\log(1/\sigma)}{\gamma}\cdot \eta\geq \frac{2\log\left(\frac{C_{\Q}\delta}{\epsilon}\right)}{\gamma}\\
		&\iff \eta\geq \frac{2\log\left(\frac{C_{\Q}\delta}{\epsilon}\right)}{\log(1/\sigma)},
	\end{align*}
	where in $(a)$ we use \citet[Claim 1]{bhathena2026solving}, which implies $\frac{\log(1/\sigma)}{\gamma} \eta\geq \log(\eta)$ for every $\eta\geq \max\left\{1, \frac{2\gamma}{\log(1/\sigma)}, \left(\frac{\gamma}{\log(1/\sigma)}\right)^2\right\}$. Therefore, if $\eta$ is chosen according to \eqref{eq::m-lb-Q}, then all merged nodes have state matrices within $\norm{\cdot}_{2,\infty}$-distance $\epsilon$ of each other, completing the proof.
\end{proof}

\subsection{Inverse-sparse structures}\label{sec:invSparse}
Next, we extend the results of \S\ref{sec:sparse} to the case where $\supp(\Q^{-1})$ is sparse. As in \S\ref{sec:invTree}, to streamline the presentation, we define $\A \defeq \Q^{-1}$ throughout this section; all graph-related quantities (distances, neighborhoods, similarity and boundaries) are now computed with respect to $G=\supp(\A)$.

Before proceeding, we present a slight generalization of Lemma~\ref{lem::inverse}, which does not require any sparsity assumption.

\begin{lemma}\label{lem::invQ2}
	Let $W$ and $V$ be disjoint subsets of $[n]$ with $V\neq\emptyset$, and let $V^c=[n]\setminus V$. Then,
	\begin{align}
		\Q_{WV} \bigl(\Q_{VV}\bigr)^{-1}
		\;=\; -\, \bm{I}_{W,V^c} \bigl(\A_{V^c V^c}\bigr)^{-1} \A_{V^c V},
	\end{align}
	where $\bm{I}_{W,V^c}$ denotes the submatrix of the $n\times n$ identity matrix with rows indexed by $W$ and columns indexed by $V^c$.
\end{lemma}

\begin{proof}
	The proof follows directly from the arguments provided in the proof of Lemma~\ref{lem::inverse} in \ref{sec:proofLemInverse}, before the sparsity of $\A$ is invoked.
\end{proof}

Next, we establish the analog of Theorem~\ref{thm::approx-Q} for the case where $\supp(\A)$ is sparse. Note that, in contrast with Theorem~\ref{thm::approx-Q}, the right-hand side of the bound involves the sizes of all neighborhoods of radius at least $\eta$: since $\Q$ is in general dense in this setting, every row of the state matrix depends on all decisions, and the influence of far-away decisions is controlled through the decay properties of $\A$.

\begin{theorem}\label{thm::approx-A}
	Fix any $\ell\in[2:n]$ and $i \in [\ell:n]$, and consider $\bm{v_1}, \bm{v_2} \in \{0,1\}^{\ell-1}$ that are $\eta$-similar in $\supp(\A)$ with respect to $i$. Then,
	\begin{align*}
		\norm{
			\bm{H^\ell}(\bm{v_1})_{\{i\}}
			- \bm{H^\ell}(\bm{v_2})_{\{i\}}
		}_2
		\;\leq\;
		\frac{8 \, \mu_{\max}^{2.5}(\Q)}{\mu_{\min}^{3}(\Q)}\,
		C_1\left(\sum_{t=\eta}^{n} \sqrt{\Delta_t} \, \sigma^{t}\right),
	\end{align*}
	where $C_1\defeq\frac{1}{\mu_{\min}(\Q)}\max\left\{1,\frac{(1+\sqrt{\kappa_2})^2}{2\kappa_2}\right\}$.
\end{theorem}

\begin{proof} This proof uses the auxiliary notation introduced in \ref{sec:auxNotation}. 
	
	For $s \in \{1,2\}$, let $T^{(s)} = \{j \in [\ell-1] : (\bm{v_s})_j = 1\}$ and $T^{(s)}_c = [n] \setminus T^{(s)}$; note that $[\ell:n]\subseteq T_c^{(s)}$, and in particular $i\in T_c^{(s)}$.
	Moreover, define $V = \{j \in [\ell-1] : (\bm{v_1})_j = (\bm{v_2})_j = 1,\ d(i,j) \leq \eta\}$ and
	$W^{(s)} = \{j \in [\ell-1] : (\bm{v_s})_j = 1,\ d(i,j) > \eta\}$ for $s \in \{1,2\}$, so that $T^{(s)}=V\cup W^{(s)}$ by $\eta$-similarity.
	We begin with the following identity:
	\begin{align*}
		\bm{H^\ell}(\bm{v_1})_{\{i\}}-\bm{H^\ell}(\bm{v_2})_{\{i\}} = &\hat\Q_{\{i\},T^{(2)}} \left(\hat\Q_{T^{(2)}T^{(2)}}\right)^{-1} \hat\F_{T^{(2)}}-\hat\Q_{\{i\},T^{(1)}} \left(\hat\Q_{T^{(1)}T^{(1)}}\right)^{-1} \hat\F_{T^{(1)}} \\
		=& \bm{I_{\{i\}}}\left(\hat \A_{T^{(1)}_c T^{(1)}_c}\right)^{-1} \hat \A_{T^{(1)}_c T^{(1)}}\hat\F_{T^{(1)}}-\bm{I_{\{i\}}}\left(\hat \A_{T^{(2)}_c T^{(2)}_c}\right)^{-1} \hat \A_{T^{(2)}_c T^{(2)}}\hat\F_{T^{(2)}},
	\end{align*}
	where the first equality follows from Lemma~\ref{lem:closedFormState}, and the second equality follows from Lemma~\ref{lem::invQ2}, with $\bm{I_{\{i\}}}$ denoting the $i$-th row of the $n\times n$ identity matrix. This identity leads naturally to a decomposition similar to that used in the proof of Theorem~\ref{thm::approx-Q}:
	{\small\begin{align*}
			\bm{H^\ell}(\bm{v_1})_{\{i\}}-\bm{H^\ell}(\bm{v_2})_{\{i\}}=& \underbrace{\bm{I_{\{i\}}}\left(\hat\A_{T^{(1)}_c T^{(1)}_c}\right)^{-1}\left(\hat\A_{T^{(1)}_c T^{(1)}}-\hat\A_{T^{(2)}_c T^{(2)}}\right)\hat\F_{T^{(1)}}}_{:= A_1}\\
			&+\underbrace{\bm{I_{\{i\}}}\left(\left(\hat\A_{T^{(1)}_c T^{(1)}_c}\right)^{-1}-\left(\hat\A_{T^{(2)}_c T^{(2)}_c}\right)^{-1}\right)\hat\A_{T^{(2)}_c T^{(2)}}\hat\F_{T^{(1)}}}_{:= A_2}\\
			&+\underbrace{\bm{I_{\{i\}}}\left(\hat\A_{T^{(2)}_c T^{(2)}_c}\right)^{-1}\hat\A_{T^{(2)}_c T^{(2)}}\left(\hat\F_{T^{(1)}}-\hat\F_{T^{(2)}}\right)}_{:= A_3}.
	\end{align*}}%
	We control each term separately, starting with $A_3$. Observe that for every $j \in V$, we have
	$\bigl(\hat\F_{T^{(1)}}\bigr)_{j} = \bigl(\hat\F_{T^{(2)}}\bigr)_{j}$,
	and therefore $\hat\A_{T^{(2)}_c T^{(2)}}\left(\hat\F_{T^{(1)}}-\hat\F_{T^{(2)}}\right) = \hat\A_{T^{(2)}_c W^{(2)}}\left(\hat\F_{T^{(1)}}-\hat\F_{T^{(2)}}\right)_{W^{(2)}}$. Let $\Gamma^{(2)}\defeq \Gamma_{\A}\left(W^{(2)}\right)\cap T_c^{(2)}$ denote the set of nodes of $T_c^{(2)}$ that are adjacent to $W^{(2)}$ in $\supp(\A)$.
	For every $j \in T_c^{(2)} \setminus \Gamma^{(2)}$, we have
	$\bigl(\hat\A_{T^{(2)}_c W^{(2)}}\bigr)_{j} = \bm{0}$.
	Consequently,
	\begin{align*}
		\norm{A_3}_2&\leq \norm{\left(\left(\hat\A_{T^{(2)}_c T^{(2)}_c}\right)^{-1}\right)_{\{i\},\Gamma^{(2)}}}_2\norm{\hat\A_{\Gamma^{(2)} W^{(2)}}}_2\norm{\hat\F_{T^{(1)}}-\hat\F_{T^{(2)}}}_2.
	\end{align*}
	Note that $\norm{\hat\F_{T^{(1)}}-\hat\F_{T^{(2)}}}_2 \leq 2\norm{\F}_2 \leq 2\sqrt{\mu_{\max}(\Q)}$.
	and $\norm{\hat\A_{\Gamma^{(2)} W^{(2)}}}_2 \leq \norm{\A}_2 = \mu_{\max}(\A)=1/\mu_{\min}(\Q)$. Next, we control $\norm{((\hat\A_{T^{(2)}_c T^{(2)}_c})^{-1})_{\{i\},\Gamma^{(2)}}}_2$. Since every $j\in \Gamma^{(2)}$ is adjacent to a node at distance greater than $\eta$ from $i$, we have $d(i,j) \geq \eta$ for every $j \in \Gamma^{(2)}$. Let $M_t$ be the set of nodes at distance exactly $t$ from $i$. Then, we have
	\begin{align*}
		\norm{\left(\left(\hat\A_{T^{(2)}_c T^{(2)}_c}\right)^{-1}\right)_{\{i\},\Gamma^{(2)}}}_2\leq \sum_{t=\eta}^n\sqrt{|M_t|}\,C_1\sigma^{t}\leq C_1\left(\sum_{t=\eta}^n\sqrt{\Delta_t}\,\sigma^{t}\right),
	\end{align*}
	where the first inequality follows by grouping the columns of $\Gamma^{(2)}$ according to their distance from $i$ and applying Lemma~\ref{lem:decayingInv} to the positive definite matrix $\A_{T^{(2)}_c T^{(2)}_c}$. Combining these bounds, we obtain
	\begin{align*}
		\norm{A_3}_2\leq \frac{2\sqrt{\mu_{\max}(\Q)}}{\mu_{\min}(\Q)}\,C_1\left(\sum_{t=\eta}^n\sqrt{\Delta_t}\,\sigma^{t}\right).
	\end{align*}
	
	We now turn to controlling $A_1$.
	Let $M_{\leq \eta}$ denote the set of nodes of $\supp(\A)$ whose distance from $i$ is at most $\eta$, and let $M_{>\eta}\defeq[n]\setminus M_{\leq\eta}$.
	By the $\eta$-similarity assumption, we have
	$T^{(1)} \cap M_{\leq \eta} = T^{(2)} \cap M_{\leq \eta}$,
	which implies
	$T^{(1)}_c \cap M_{\leq \eta} = T^{(2)}_c \cap M_{\leq \eta}$.
	Consequently, for every $j$ such that $d(i,j)\leq \eta-1$ (so that $j$ and all of its neighbors belong to $M_{\leq \eta}$), the rows $\bigl(\hat\A_{T_c^{(1)} T^{(1)}}\bigr)_{j}$ and $\bigl(\hat\A_{T_c^{(2)} T^{(2)}}\bigr)_{j}$ coincide.
	Using this observation, and letting $M_{\geq \eta}\defeq \{j\in[n]: d(i,j)\geq \eta\}$, we can write
	\begin{align*}
		\norm{A_1}_2 \leq& \norm{\left(\left(\hat\A_{T^{(1)}_c T^{(1)}_c}\right)^{-1}\right)_{\{i\}, M_{\geq \eta}}}_2\left(\norm{\hat \A_{ M_{\geq \eta} T^{(1)}}}_2+\norm{\hat \A_{ M_{\geq \eta} T^{(2)}}}_2\right)\norm{\hat\F_{T^{(1)}}}_2\\
		\leq& \frac{2\sqrt{\mu_{\max}(\Q)}}{\mu_{\min}(\Q)}\sum_{t=\eta}^n\norm{\left(\left(\hat\A_{T^{(1)}_c T^{(1)}_c}\right)^{-1}\right)_{\{i\}, M_{t}}}_2\\
		\leq& \frac{2\sqrt{\mu_{\max}(\Q)}}{\mu_{\min}(\Q)}\,C_1\left(\sum_{t=\eta}^n\sqrt{\Delta_t}\,\sigma^{t}\right),
	\end{align*}
	where the last inequality again follows from Lemma~\ref{lem:decayingInv}.
	
	Finally, we control $A_2$.
	Let $V_c = \{j \in [\ell-1] : d(i,j) \leq \eta,\ j \notin V\}$ denote the decided nodes near $i$ that are fixed to zero, for $s \in \{1,2\}$ let
	$W^{(s)}_c = \{j \in [\ell-1] : d(i,j) > \eta,\ j \notin W^{(s)}\}$ denote the decided nodes far from $i$ that are fixed to zero in $\bm{v_s}$, and define $\tilde V \defeq [\ell:n] \cup V_c$.
	With these definitions, $T_c^{(s)}=\tilde V\cup W_c^{(s)}$ and we can express
	\begin{align*}
		\A_{T_c^{(1)} T_c^{(1)}}=\begin{pmatrix}
			\A_{\tilde V \tilde V} & \A_{\tilde V W_c^{(1)}}\\
			\A_{W_c^{(1)} \tilde V} & \A_{W_c^{(1)} W_c^{(1)}}
		\end{pmatrix}, \quad \A_{T_c^{(2)} T_c^{(2)}}=\begin{pmatrix}
			\A_{\tilde V \tilde V} & \A_{\tilde V W_c^{(2)}}\\
			\A_{W_c^{(2)} \tilde V} & \A_{W_c^{(2)} W_c^{(2)}}
		\end{pmatrix},
	\end{align*}
	where the two matrices share the block $\A_{\tilde V\tilde V}$ and $i\in \tilde V$. We may therefore apply~\eqref{eq::inverse} to both inverses; since the terms involving $(\A_{\tilde V\tilde V})^{-1}$ alone cancel, and denoting by $\bm{P^{(s)}}\defeq \A_{T_c^{(s)}T_c^{(s)}}/ \A_{\tilde V \tilde V}$ the corresponding Schur complements, proceeding analogously to the proof of Theorem~\ref{thm::approx-Q}, we obtain
	{\footnotesize\begin{align*}
			&\norm{A_2}_2
			\leq \norm{\left(\left(\hat\A_{T^{(1)}_c T^{(1)}_c}\right)^{-1}-\left(\hat\A_{T^{(2)}_c T^{(2)}_c}\right)^{-1}\right)_{\{i\}}}_2\norm{\hat \A_{T_c^{(2)} T^{(2)}}}_2\norm{\hat\F_{T^{(1)}}}_2\\
			&\leq \frac{\sqrt{\mu_{\max}(\Q)}}{\mu_{\min}(\Q)}\!\sum_{s=1}^2\!\left(\norm{\left(\left(\A_{\tilde V \tilde V}\right)^{-1}\right)_{\{i\}}\!\!\!\A_{\tilde V W_c^{(s)}}\left(\bm{P^{(s)}}\right)^{-1}\!\!\A_{W_c^{(s)}\tilde V}\left(\A_{\tilde V \tilde V}\right)^{-1}}_2 \!\!\!+\norm{\left(\left(\A_{\tilde V \tilde V}\right)^{-1}\right)_{\{i\}}\A_{\tilde V W_c^{(s)}}\left(\bm{P^{(s)}}\right)^{-1}}_2\right).
	\end{align*}}%
	Upon defining $\Gamma_c^{(s)}$ as the set of nodes of $\tilde V$ that are adjacent to $W_c^{(s)}$ in $\supp(\A)$,
	we note that $\bigl(\A_{\tilde V W_c^{(s)}}\bigr)_{j} = \bm{0}$ unless $j \in \Gamma_c^{(s)}$, and that every $j \in \Gamma_c^{(s)}$ satisfies $d(i,j) \geq \eta$.
	Hence, by the same argument used to bound $A_3$,
	\begin{align*}
		\norm{\left(\left(\A_{\tilde V \tilde V}\right)^{-1}\right)_{\{i\}}\A_{\tilde V W_c^{(s)}}}_2 \leq \norm{\left(\left(\A_{\tilde V \tilde V}\right)^{-1}\right)_{\{i\}, \Gamma_c^{(s)}}}_2\norm{\A_{\Gamma_c^{(s)} W_c^{(s)}}}_2
		\leq \frac{1}{\mu_{\min}(\Q)}\,C_1\left(\sum_{t=\eta}^n\sqrt{\Delta_t}\,\sigma^{t}\right)
	\end{align*}
	for $s\in\{1,2\}$. Combining this bound with our expression for $A_2$, and using
	\begin{align*}
		\norm{\left(\bm{P^{(s)}}\right)^{-1}}_2 &\leq \frac{1}{\mu_{\min}(\A)}=\mu_{\max}(\Q)\text{ and }
		\norm{\A_{W_c^{(s)}\tilde V}\left(\A_{\tilde V \tilde V}\right)^{-1}}_2 \leq \frac{\mu_{\max}(\A)}{\mu_{\min}(\A)}=\frac{\mu_{\max}(\Q)}{\mu_{\min}(\Q)},
	\end{align*}
	we conclude that
	\begin{align*}
		\norm{A_2}_2 \;\leq\;
		\frac{4\, \mu_{\max}^{2.5}(\Q)}{\mu_{\min}^{3}(\Q)}\,
		C_1 \left(\sum_{t=\eta}^n\sqrt{\Delta_t}\,\sigma^{t}\right).
	\end{align*}
	Finally, since $\kappa_2\geq 1$,
	\[
	\frac{\sqrt{\mu_{\max}(\Q)}}{\mu_{\min}(\Q)}
	\leq \frac{\mu_{\max}^{2.5}(\Q)}{\mu_{\min}^{3}(\Q)}.
	\]
	Combining this inequality with the bounds on $A_1$, $A_2$, and $A_3$ completes the proof.
\end{proof}

As in the case where $\supp(\Q)$ is sparse, Theorem~\ref{thm::approx-A} implies that all partial solutions that coincide on the $\eta$-boundary $\Lambda_{\ell}^{\eta}$ of $\supp(\A)$ are $\eta$-similar with respect to every $i\in[\ell:n]$, and thus satisfy
\[
\norm{\bm{H^\ell}(\bm{v_1}) - \bm{H^\ell}(\bm{v_2})}_{2,\infty}
\;\leq\;
C_{\Q}' \left(\sum_{t=\eta}^{n} \sqrt{\Delta_t} \, \sigma^{t}\right),
\;\;\text{where}\;\;
C_{\Q}'\defeq \frac{8\,\mu_{\max}^{2.5}(\Q)}{\mu_{\min}^{3}(\Q)}C_1.
\]
This leads to the following counterpart of Corollary~\ref{cor::epsilon-Q-size}.
\begin{corollary}\label{cor::epsilon-A-size}
	Suppose $\supp(\A)$ has polynomial volume growth with parameters $\gamma,\delta>0$ (Definition~\ref{def::vol-Q}). Given any $\epsilon>0$, there exists an $\epsilon$-exact decision diagram with at most $(n+1)\cdot 2^{\omega_{\A}(\eta)}$ nodes, where $\omega_{\A}(\eta)$ is the maximum size of the $\eta$-boundaries of $\supp(\A)$ and
	{\small\begin{align}\label{eq::m-lb-A}
			\eta \;=\;
			\left\lceil\max\left\{
			\frac{2\log\left(C_{\Q}'\sqrt{\delta}/ \epsilon\right)+2\log\left(1/(1-\sqrt{\sigma})\right)}{\log(1/\sigma)}, \;
			\frac{2\gamma}{\log(1/\sigma)}, \;
			\left(\frac{\gamma}{\log(1/\sigma)}\right)^2
			\right\}\right\rceil.
	\end{align}}%
\end{corollary}

\begin{proof}
	As in the proof of Corollary~\ref{cor::epsilon-Q-size}, consider the diagram obtained by merging, at every layer $\ell$, all nodes whose partial solutions coincide on the $\eta$-boundary $\Lambda_\ell^\eta$ of $\supp(\A)$ (again, neighborhood merging); its size is at most $(n+1)\cdot 2^{\omega_{\A}(\eta)}$, and it suffices to verify that the merged states are within $\norm{\cdot}_{2,\infty}$-distance $\epsilon$. We have
	\begin{align*}
		C_{\Q}' \left(\sum_{t=\eta}^{n} \sqrt{\Delta_t} \, \sigma^{t}\right)\leq \epsilon
		& \Longleftarrow  C_{\Q}'\sqrt{\delta}\left(\sum_{t=\eta}^n t^{\gamma/2}\sigma^t\right)\leq \epsilon\\
		& \stackrel{(a)}{\Longleftarrow}  C_{\Q}'\sqrt{\delta}\left(\sum_{t=\eta}^n \sigma^{t/2}\right)\leq \epsilon\\
		& \Longleftarrow C_{\Q}'\sqrt{\delta}\,\sigma^{\eta/2}\left(\sum_{t=0}^\infty (\sqrt{\sigma})^t\right)\leq \epsilon\\
		& \iff \left(\frac{C_{\Q}'\sqrt{\delta}}{1-\sqrt{\sigma}}\right)\sigma^{\eta/2}\leq \epsilon\\
		& \Longleftarrow \eta\geq \frac{2\log\left(C_{\Q}'\sqrt{\delta}/\epsilon\right)+2\log\left(1/(1-\sqrt{\sigma})\right)}{\log(1/\sigma)},
	\end{align*}
	where in $(a)$, we again use \citet[Claim 1]{bhathena2026solving}, which implies $t^{\gamma/2}\leq \sigma^{-t/2}$ for every $t\geq \eta$, provided that $\eta\geq \max\left\{1,
	\frac{2\gamma}{\log(1/\sigma)}, \;
	\left(\frac{\gamma}{\log(1/\sigma)}\right)^2
	\right\}$. Therefore, if $\eta$ is chosen according to \eqref{eq::m-lb-A}, all merged nodes have state matrices within $\norm{\cdot}_{2,\infty}$-distance $\epsilon$ of each other, completing the proof.
\end{proof}

As in the sparse case, the value of $\eta$ in \eqref{eq::m-lb-A} is independent of $n$ and grows only logarithmically in $1/\epsilon$; thus, the size of the $\epsilon$-exact decision diagram scales linearly in $n$, multiplied by the boundary factor $2^{\omega_{\A}(\eta)}$. In particular, if $\A=\Q^{-1}$ is banded with bandwidth $k$, then $\omega_{\A}(\eta)\leq k\eta$ and the size is polynomial in $1/\epsilon$. We emphasize that, in contrast with the case of $\supp(\Q)$, this constitutes a new result even in the banded case.

\subsection{Proposition~\ref{prop:epsQuality}}
\begin{resultrestatement}{Proposition~\ref{prop:epsQuality}}
	Suppose that $\bm{Q}\succ\bm{0}$, let $\bm{\delta}\in\R^m$ satisfy $\bm{F\delta}=\bm{d}$, and let $\mathcal{D}_\epsilon$ be an $\epsilon$-exact decision diagram obtained via neighborhood merging as in Corollary~\ref{cor::epsilon-Q-size} or Corollary~\ref{cor::epsilon-A-size}, with arc lengths defined as in Proposition~\ref{prop:pathLength} using the states stored in the diagram. For $\bm{z}\in Z$, let $L(\bm{z})$ denote the length of the unique root--terminal path of $\mathcal{D}_\epsilon$ encoding $\bm{z}$ and $g^{n+1}(\bm{z})$ is the actual objective value associated with $\bm{z}$. Then
		$$\left|L(\bm{z})-g^{n+1}(\bm{z})\right|\leq \frac{\|\bm{\delta}\|_2^2}{\sqrt{\mu_{\min}(\Q)}}\, n\epsilon\qquad \forall \bm{z}\in Z.$$
\end{resultrestatement}

\begin{proof}
	First, we claim that every node of $\mathcal{D}_\epsilon$ stores the exact state of some partial solution belonging to its merging class. Indeed, the root stores $\bm{F}$, and if a node of layer $\ell$ stores $\bm{H^\ell}(\bm{\bar w})$ for a member $\bm{\bar w}$ of its class, then its children store the exact states $\bm{H^{\ell+1}}((\bm{\bar w},\bar z))$ for $\bar z\in \{0,1\}$. Moreover, $(\bm{\bar w},\bar z)$ belongs to the class of the corresponding child: every $j\leq \ell-1$ in $\Lambda_{\ell+1}^{\eta}$ belongs to $\Lambda_{\ell}^{\eta}$ as any $i\in [\ell+1:n]$ with $d(i,j)\le\eta$ also belongs to $[\ell:n]$, and on these coordinates all members of the parent class agree, while coordinate $\ell$ equals $\bar z$ for all partial solutions reaching the child.
	
	Now fix $\bm{z}\in Z$ and a layer $\ell$ with $z_\ell=1$, let $\bm{v}=\bm{z_{[\ell-1]}}$, and let $\bm{\bar v}$ be the class member whose exact state is stored at the node of layer $\ell$ visited by the path encoding $\bm{z}$. Since $\bm{v}$ and $\bm{\bar v}$ coincide on $\Lambda_\ell^\eta$, they are $\eta$-similar with respect to every $i\in [\ell:n]$ and, by the choice of $\eta$ in \eqref{eq::m-lb-Q} or \eqref{eq::m-lb-A}, the first rows $\bm{h}$ and $\bm{\bar h}$ of $\bm{H^\ell}(\bm{v})$ and $\bm{H^\ell}(\bm{\bar v})$ satisfy $\|\bm{h}-\bm{\bar h}\|_2\leq \epsilon$. We next argue that $\|\bm{h}\|_2^2$ and $\|\bm{\bar h}\|_2^2$ are bounded below by $\mu_{\min}(\Q)$. Indeed, since $\bm{h}^\top$ is the first row of $\bm{H^\ell}(\bm{v})=\bm{F_{[\ell:n]}}\bm{G_\perp^\ell}(\bm{v})$ (Proposition~\ref{prop:state}), we have $\bm{h}^\top=\bm{f_\ell^\top}(\bm{I}-\bm{G})$, where $\bm{f_\ell}$ is the $\ell$-th row of $\bm{F}$ and $\bm{G}=\bm{I}-\bm{G_\perp^\ell}(\bm{v})$ is the orthogonal projection matrix onto $\Row(\bm{F_v})$; since $\bm{I}-\bm{G}$ is symmetric and idempotent, it follows that $\|\bm{h}\|_2^2=\bm{f_\ell^\top}(\bm{I}-\bm{G})\bm{f_\ell}$. Letting $T=\{j\in [\ell-1]: v_j=1\}$ and expressing $\bm{G}=\bm{F_T^\top}(\bm{F_TF_T^\top})^{-1}\bm{F_T}$  (Proposition~\ref{prop:closedFormProjection}), the identities $\bm{F_T}\bm{F_T^\top}=\bm{Q_{TT}}$ and $\bm{F_T}\bm{f_\ell}=\bm{Q_{T,\{\ell\}}}$ yield $$\|\bm{h}\|_2^2=\bm{f_\ell}^\top (\bm{I}-\bm{F_T^\top}(\bm{F_TF_T^\top})^{-1}\bm{F_T})\bm{f_\ell}=Q_{\ell\ell}-\bm{Q_{\{\ell\},T}}\left(\bm{Q_{TT}}\right)^{-1}\bm{Q_{T,\{\ell\}}},$$ which is precisely the diagonal entry corresponding to index $\ell$ of the Schur complement $\Q_{II}/\Q_{TT}$ with $I=[n]$. Moreover, by \eqref{eq::inverse} applied with $V=T$ and $W=[n]\setminus T$, the inverse of this Schur complement is a principal submatrix of $\bm{Q^{-1}}$: its largest eigenvalue is thus at most $\mu_{\max}(\bm{Q^{-1}})=1/\mu_{\min}(\Q)$, and every diagonal entry of the Schur complement is at least $\mu_{\min}(\Q)$. The same argument applies to $\bm{\bar h}$, with $\bm{\bar v}$ in place of $\bm{v}$. A direct computation then shows that the normalized vectors satisfy $\left\|\frac{\bm{h}}{\|\bm{h}\|_2}-\frac{\bm{\bar h}}{\|\bm{\bar h}\|_2}\right\|_2\leq \frac{2\epsilon}{\sqrt{\mu_{\min}(\Q)}}$. Therefore, the lengths of the corresponding arcs of $\mathcal{D}_\epsilon$ and of the exact decision diagram differ by at most
	$$\frac{1}{4}\left|\left(\bm{\delta^\top}\tfrac{\bm{\bar h}}{\|\bm{\bar h}\|_2}\right)^2-\left(\bm{\delta^\top}\tfrac{\bm{h}}{\|\bm{h}\|_2}\right)^2\right|\leq \frac{1}{2}\|\bm{\delta}\|_2^2\left\|\tfrac{\bm{\bar h}}{\|\bm{\bar h}\|_2}-\tfrac{\bm{h}}{\|\bm{h}\|_2}\right\|_2\leq \frac{\|\bm{\delta}\|_2^2}{\sqrt{\mu_{\min}(\Q)}}\,\epsilon.$$
	Arcs with $\bar z=0$ have identical lengths in both diagrams, and each root--terminal path contains at most $n$ arcs with $\bar z=1$; since, by Proposition~\ref{prop:pathLength}, the length of the path encoding $\bm{z}$ in the exact diagram is $g^{n+1}(\bm{z})$, summing the per-arc errors yields the claim. 
\end{proof}

\end{document}